\documentclass[12pt, a4paper]{amsart}

\usepackage[utf8]{inputenc}
\usepackage{fancyhdr}
\usepackage{geometry}  
\usepackage{hyperref}
\hypersetup{
  colorlinks   = true, 
  urlcolor     = blue, 
  linkcolor    = blue, 
  citecolor   = purple 
}
\usepackage{subcaption}
\usepackage{float}

\usepackage{graphicx} 
\usepackage{graphics}
\usepackage{color}
\usepackage{tikz-cd}
\usepackage{tikz}
\usetikzlibrary{calc}

\usepackage{listings}
\usepackage{cite} 
\usepackage{enumerate}
\usepackage{enumitem}

\usepackage{amsmath, amsthm, amsfonts, amssymb, amssymb}
\usepackage{mathrsfs}
\usepackage{mathtools}
\usepackage{relsize}
\usepackage[all]{xy}

\usepackage[symbol]{footmisc}
\usepackage[toc,page]{appendix}

\newcommand\ZZ{\mathbb{Z}}

\newcommand\CC{\mathbb{C}}
\newcommand\PP{\mathbb{P}}

\newcommand\DD{\mathcal{D}}

\newcommand\OO{\mathcal{O}}

\newcommand\Q{\mathcal{Q}}

\newcommand\FF{\mathcal{F}}

\newcommand\LL{\mathcal{L}}

\newcommand\Fol{\text{Fol}}

\newcommand\Aut{\text{Aut}}
\newcommand\Hom{\text{Hom}}

\DeclareMathOperator{\sing}{Sing}
\DeclareMathOperator{\exc}{Exc}
\DeclareMathOperator{\supp}{Supp}

\DeclareMathOperator{\vol}{\mathrm{vol}}

\newcommand{\setfoliation}[1]{\mathfrak F^{\rm adj, #1}_{N, v}}

\newcommand{\setgenpair}[1]{\mathfrak G^{#1}_{N,v}}

\newcommand{\setgpairsdcc}[1]{\mathfrak G^{\rm g#1}_{v, I, N}}

\newcommand{\loccit}{{\it loc.~cit.}}

\newtheorem{theorem}{Theorem}[section]
\newtheorem{proposition}[theorem]{Proposition}
\newtheorem{lemma}[theorem]{Lemma}
\newtheorem{corollary}[theorem]{Corollary}
\newtheorem{claim}[theorem]{Claim}
\newtheorem{conjecture}[theorem]{Conjecture}
\newtheorem*{question}{Question}

\theoremstyle{definition}

\newtheorem{construction}[theorem]{Construction}

\newtheorem{definition}[theorem]{Definition}

\newtheorem{remark}[theorem]{Remark}
\newtheorem{ej}[theorem]{Example}

\newtheoremstyle{named}{}{}{\itshape}{}{\bfseries}{.}{.5em}{\thmnote{#3 }#1}
\theoremstyle{named}

\title{On moduli of foliated surfaces}
\author{Calum Spicer, Roberto Svaldi and Sebasti\'{a}n Velazquez}
\date{}

\address{Department of Mathematics, King's College London, Strand,
London WC2R 2LS, UK}
\email{calum.spicer@kcl.ac.uk}

\address{Dipartimento di Matematica ``F. Enriques'', Universit\`a degli Studi di Milano, Via Saldini 50, 20133 Milano (MI), Italy}
\email{roberto.svaldi@unimi.it}

\address{Departamento de Matem\'atica, Facultad de Ciencias Exactas y Naturales, Universidad de Buenos Aires  and IMAS (UBA - CONICET), Buenos Aires, Argentina.}
\email{svelazquez@dm.uba.ar}

\begin{document}

\begin{abstract} 
The birational geometry of foliations exhibits phenomena that are fundamentally different than the case of varieties and that complicate the construction of a satisfactory moduli theory: 
the canonical class of a minimal foliation of general type need not determine a polarization; 
canonical singularities are not stable under deformation; 
and a family of foliations may degenerate to a non-saturated integrable distribution. 
These features imply that a new notion of stability and new techniques are required to control degenerations.
In this paper, we develop a theory of stable foliated surfaces and construct their moduli spaces. 

To achieve these goals, we introduce semi-log canonical foliations and identify a class of stable integrable distributions on surfaces which is preserved under degeneration. 
We also establish the corresponding boundedness, deformation, separatedness, and stable-reduction results. 
Some of the main technical ingredients are inversion of adjunction and deformation-openness for log canonical singularities of rank-one foliations.

Using these new results, we show that stable foliated surfaces with fixed Cartier index and adjoint volume form a Deligne--Mumford stack of finite type over $\mathbb C$. 
Allowing the numerical invariants to vary yields a Deligne--Mumford stack locally of finite type which satisfies the valuative criterion for properness for families with normal generic fibre. 
This provides a compactification framework for moduli of surface foliations and links their birational classification with their deformation theory.
\end{abstract}

\maketitle

\tableofcontents

\section{Introduction}
Recently, there has been considerable progress in the study of the birational geometry of foliations coming from the  development of the Minimal Model Program (MMP) for foliations.  The main aim of this paper is to develop this study further by applying techniques from the MMP for foliations to the problem of constructing moduli spaces of foliations.  Our main result is the existence of a moduli space for foliated surfaces.  This can be viewed as the final step in the classification of surface foliations begun in \cite{Brunella00, McQuillan08, Mendes00}.

The study of moduli of foliations goes back to at least Jouanolou, \cite{Jouanolou79}, who---motivated by earlier contributions of Jacobi and Darboux---constructed a parameter space for codimension one foliations on projective space.  The moduli space of foliations on $\mathbb P^n$ has been an intensively studied object, and most of the work on the subject has focused on understanding the various components of the moduli space, see for instance 
\cite{cerveau1996irreducible}. 

However, for a complete theory of moduli of foliated varieties it is crucial to allow both the foliation and the underlying variety to vary, see for instance \cite{pourcin1987deformations, gomez1988transverse}.  This is the point of view we adopt here.
Our main goal is to begin the study of moduli of foliations in analogy with the moduli of stable curves, or KSB stable varieties more generally.

\medskip

In this paper we propose a definition of stable family of foliations and prove two main results on the existence of moduli spaces for the moduli functor of stable foliations on surfaces.

\begin{theorem}[= Theorem \ref{thm:rep2}]
\label{thm_main_2}
    The moduli functor of stable foliated surfaces, denoted $\mathcal M^{2, 1}$
    is represented by a Deligne-Mumford stack locally of finite type over $\mathbb C$ and which satisfies the valuative criterion for properness for families with normal generic fibre over DVRs which are essentially finite type over $\mathbb C$.
\end{theorem}

We remark that in contrast to the case of stable varieties, stable limits will not always exist for families when the generic fibre is not normal, see Example \ref{non-normal-not-proper}.

$\mathcal M^{2, 1}$ will have infinitely many components and will not be of finite type, however by fixing two natural invariants we are able to show that the corresponding sub-functor is represented by a Deligne-Mumford stack of finite type.

\begin{theorem}[= Theorem \ref{thm:rep}]
\label{thm_main}
    Fix a positive integer $N$
    and a positive real number $v$.
The moduli functor of stable foliated surfaces of index $=N$ and adjoint volume $=v$,  denoted $\mathcal{M}^{2, 1}_{N, v}$,  is represented by a Deligne-Mumford stack of finite type over $\mathbb C$.
\end{theorem}

\medskip

\subsection{Stable families of foliations} 
The first step in developing a moduli theory for foliations is to choose the correct definition of stable family of foliations.  
To help fix ideas consider the following prototype example of a family of stable foliations:

Let $f\colon X \to \mathbb A^1$ be a smooth proper fibration over $\mathbb A^1$, let $\mathcal F$ be a smooth foliation on $X$ such that $T_{\mathcal F} \subset T_{X/\mathbb A^1}$ and such that $K_{\mathcal F}$ is $f$-ample.  Then $f\colon (X, \mathcal F) \to \mathbb A^1$ gives an example of a stable family of foliations.  
However, it will usually not be possible to extend $f$ to a family of smooth foliations with ample canonical class over $\mathbb P^1$. Typically, the foliation and underlying variety will both become singular in the limit.
It follows that smooth foliations with ample canonical class are too restrictive a class of foliations for achieving a proper moduli space.
The challenge in determining the correct class of stable foliations is therefore to 
enlarge the class of stable foliations ``just enough'' to ensure that the moduli functor is proper, without losing the representability of the functor by a separated space of finite type.

The study of moduli of curves (and higher dimensional varieties) indicates  that the definition of stable family of foliations should consist of four parts:
\begin{itemize}
    \item A naturally defined polarisation, e.g., the canonical bundle is ample;
    \item A condition on the singularities, e.g. nodal;
    \item A restriction on the kind of allowable families, e.g., flatness; and 
    \item A restriction on the numerical invariants of the objects considered, e.g. genus.
\end{itemize}
One should choose this definition in such a way which guarantees that the corresponding moduli functor is represented by a proper algebraic space.

\subsection{Definition of the moduli functor}
\label{s_intro_def_mod_functor}
We now define (roughly) the moduli functors that we will focus on in this paper.  We refer to Section \ref{s_precise_definition} for a precise definition.

We say that a flat family of integrable distributions\footnote[1]{See Section \ref{note_on_terminolgy} for an explanation of the use of the term ``distribution'' instead of ``foliation''.}
$f\colon (X, \mathcal F) \rightarrow T$
is stable 
provided 
\begin{enumerate}
    \item 
\label{intro_cond_Qfactorial}
    $K_{X/T}$ and $K_{\mathcal F}$
are $\mathbb Q$-Cartier; 

\item
\label{intro_cond_kollar_condition}
for all $j, k \in \mathbb Z$, $\mathcal O_X(jK_{\mathcal F}+kK_{X/T})$ commutes with base change;

\item 
\label{intro_cond_slc}
for all closed points $t \in T$, $X_t$ and $\mathcal F_t$
are semi-log canonical; and 
\item  
\label{intro_cond_polarization}
there exists $\eta \in \mathbb R_{>0}$ such that $K_{\mathcal F}+\epsilon K_{X/T}$ is $f$-ample for all $0<\epsilon<\eta$.
\end{enumerate}

We refer to Section \ref{s_locally_stable_def} for a precise definition of stable family and to Section \ref{s_singularities} for a precise definition of semi-log canonical, but we remark that the class of log canonical singularities is natural (and unavoidable, see Example \ref{easiest_example}) from the perspective of birational geometry and moduli theory.
We also remark that we will also need to work in the generality of the log setting, i.e., we will need to consider families of the form $(X, \Delta, \mathcal F) \rightarrow T$ where 
$(X, \Delta) \rightarrow T$ is a family of log varieties with $\Delta$
a reduced divisor.

We refer to Section \ref{s_precise_definition} for a precise definition of our moduli functor, but we briefly sketch its definition here.
Fix $d, r, N \in \mathbb Z_{>0}$ and $v \in \mathbb R_{>0}$.
We define $\mathcal M^{d, r}_{N, v}(T)$ to be the set of flat families of
integrable distributions over $T$ (modulo isomorphism), $f\colon (X, \Delta, \mathcal F) \rightarrow T$
such that the following hold:

  \begin{enumerate}[label=(\roman*)]
    \item
    \label{cond_bla}
    $f$ is flat and proper of relative dimension $d$;

\item $f\colon (X, \Delta, \mathcal F) \rightarrow T$ is a stable family of integrable distributions of rank $r$;

    \item $NK_{\mathcal F}$ is Cartier; and 

    \item ${\rm vol}((2d+1)NK_{\mathcal F_t}+K_{X_t}+\Delta_t
    ) = v$ for all $t \in T$.

\end{enumerate}

As examples show, see Section \ref{s_examples}, it is necessary to fix both $N$ and $v$ to ensure that the moduli space is of finite type.  Moreover, it does not suffice to fix $K_{\mathcal F_t}^2$ 
to ensure that the moduli space is of finite type, see \cite[Example 4.1]{passantino2024}.
We denote by $\mathcal M^{d, r}$ the moduli functor parametrising families of foliations only satisfying (i) and (ii) above.  Clearly $\mathcal M^{d, r}$ is not of finite type, but as we will see it is locally of finite type.

\subsection{Deformations of singularities}
In order to construct the moduli space it is essential to show that our stability condition is a locally closed condition in any family of foliations.  We therefore need to understand how singularities of foliations vary in families.

Our first result in this direction is that semi-log canonical singularities is an open condition in families.

\begin{theorem}[ = Theorem \ref{thm_lc_deform}]
\label{thm_intro_lc_deform}
        Let $(X, \mathcal F) \rightarrow T$ be a flat
    family of integrable distributions with $\dim (X/T) = 2$ and suppose that $K_{\mathcal F}$
    is $\mathbb Q$-Cartier.

 Suppose that $\mathcal F_t$ is semi-log canonical for some closed point $t \in T$.  Then there exists a Zariski open subset $t \in U \subset T$ such that for all $s \in U$,
    $\mathcal F_{s}$ is semi-log canonical.    
\end{theorem}

A key ingredient in the proof Theorem \ref{thm_intro_lc_deform} is the following result on the inversion of adjunction of log canonical singularities for foliations. 

\begin{theorem}[= Theorem \ref{thm_inversion_of_adjunction_full}]
Let $X$ be a normal variety, let $\mathcal F$ be a rank one foliation on $X$, let $S$ be a prime divisor on $X$ such that $K_{\mathcal F}+\varepsilon(S)S$ is $\mathbb Q$-Cartier.
Let $n\colon T \rightarrow S$ be the normalisation and write (via foliation adjunction) \begin{align*}n^*(K_{\mathcal F}+\varepsilon(S)S) = K_{\mathcal F_T}+B_T.\end{align*} 
If $(\mathcal F_T, B_T)$ is log canonical, then 
$(\mathcal F, \varepsilon(S)S)$ is log canonical in a neighbourhood of $S$.
\end{theorem}

In fact, we conjecture that these two theorems hold in much greater generality, see Conjectures \ref{conjecture sing} and \ref{conj_inversion_adjunction} below.  Both of these conjectures are known for varieties, but we expect that the study of these conjectures will be of interest for general foliations outside of their role in the study of moduli problems. 

We remark that Theorem \ref{thm_intro_lc_deform} is false if ``semi-log canonical'' is replaced by ``canonical'', see Example \ref{easiest_example}.  Indeed, canonical singularities only impose a very general condition in families and so there is no algebraically representable functor parametrising foliations with canonical singularities.

We also remark that adjunction on singularities of foliations fails when $\varepsilon(S) = 0$ (see \cite[Example 3.20]{CS23b}), and represents one key difference between the birational theory of foliations and that of varieties.  This presents several additional technical challenges, especially when proving Theorem \ref{thm_intro_lc_deform}.

\subsection{Versality}
We recall the following result due to Gomez-Mont, \cite[Theorem 3.4]{gomez1988transverse}, on the existence of certain versal deformation spaces of foliations.

\begin{theorem} Let $\FF_0$ be an integrable distribution on a compact connected manifold $X_0$. Then there exist a deformation of $(X,\FF)\to (D,0)$ such that $N_\FF$ is flat over $D$ and is miniversal among deformation with flat normal sheaf, i.e., if $(X',\FF')\to (S,s)$ is another deformation of $(X_0,\FF_0)$ such that $N_{\mathcal F'}$ is flat over $S$, then there exists an analytic neighborhood $V\ni s$ and a morphism $\varphi:(V,s)\to (D,0)$ such that $(X',\FF')\vert_V\simeq \varphi^*(X,\FF)$. Moreover, this yields an isomorphism between $T_0D$ and the tangent space to the corresponding deformation functor.
\end{theorem}

The condition on the flatness of $N_{\mathcal F}$ is not natural in many situations of interest, see Example \ref{ex:QHusk is necessary}. 
Our main theorem naturally implies, as an immediate corollary, the following existence result of versal deformation spaces for stable foliations.

\begin{theorem}[ = Theorem \ref{thm_versalilty_stable_def}]
    Let $(X_0,\Delta_0,\FF_0)$ be a stable foliated triple, where $X_0$ is a surface. 
    Then there exists a  stable deformation of $(X,\Delta,\FF)\to (D,0)$ satisfying the following miniversality property amongst stable deformations:
    
    For every stable deformation $(X',\Delta',\FF')\to (V,s)$ of $(X_0,\Delta_0,\FF_0)$ there exists an analytic neighbourhood $V\ni s$ and a morphism $\varphi: (V,s)\to (D,0)$ such that $(X',\Delta',\FF')\simeq \varphi^*(X,\Delta,\FF)$. 
    
    Moreover, the above property implies the existence of an isomorphism between $T_0D$ and the tangent space to the functor of stable deformations of $(X_0,\Delta_0,\FF_0)$.
\end{theorem}

\subsection{Boundedness (or lack thereof) of the moduli functor $\mathcal M^{2,1}$}

In Section~\ref{s_boundedness}, we show the boundedness of $\mathcal M^{2,1}_{N, v}$: namely, we show that there exists only finitely family of algebraic deformations parametrising stable foliated pairs appearing in 
$\mathcal M^{2,1}_{N, v}$.
Our result generalises the previous works on foliated boundedness in dimension two, see \cite{HL19, chen21, SS23}, finally including also the case of semi-log canonical foliated pairs that had not been previously discussed at all:
indeed, one of the important contributions of this manuscript is the definition of semi-log canonical foliated surfaces.
As an immediate consequence, the results in Section~\ref{s_boundedness} imply that 
$\mathcal M^{2,1}_{N, v}$ 
can be coarsely represented by an algebraic space of finite type.

As $\mathcal M^{2,1}_{N, v}$ is not necessarily proper, whereas the larger moduli space $\mathcal M^{2,1}$ is valuative proper, it is then natural to ask whether or not the closure 
$\overline{\mathcal M^{2,1}_{N, v}}$
of 
$\mathcal M^{2,1}_{N, v}$ inside 
$\mathcal M^{2, 1}$ 
is of finite type.
In this manuscript, we are not able to establish such a result;
on the other hand, we do not know of any examples in which 
$\overline{\mathcal M^{2,1}_{N, v}}$
fails to be of finite type.

In general, in order to prove the boundedness of 
$\overline{\mathcal M^{2,1}_{N, v}}$, 
for fixed $N \in \mathbb Z_{>0}$
and $v \in \mathbb Q_{>0}$,
we need to be able control the behaviour of the Cartier index of the canonical divisor of a foliation in families. 
While we are able to show the constancy of the Cartier index of the foliation in stable families for many types of singularities, cf. Section~\ref{s_index_families}, a similar result may not hold for strictly semi log terminal singularities, cf. Example~\ref{slt_index_jumps}.
Hence, it is natural to pose the following question, a positive answer to which would allow us to show that $\overline{\mathcal M^{2, 1}_{N, \nu}}$ is of finite type, at least for $N$ sufficiently large, and, moreover, it is contained in 
$\mathcal M^{2, 1}_{N', \nu'}$
for a suitable choice of an integer
$N' > N$ 
and 
$\nu' \in \mathbb Q$.

\begin{question}
Does there exist an integer $N$ depending only the Hilbert function $\chi(X, \mathcal O(mK_{\mathcal F}+n(K_X+\Delta))$ such that $NK_{\mathcal F}$ is Cartier at all the semi-log terminal points of $(X, \Delta, \mathcal F)$?
\end{question}

\subsection{Infinitesimal structure of the moduli functor}
We remark that in our construction we capture the non-reduced structure of the moduli space.  
Developing a theory of moduli of pairs $(X, \Delta)$ has proven to be very difficult when $\Delta$ is not a reduced divisor.  One key issue is defining families of pairs $(X, \Delta) \to S$ when $S$ is not a reduced scheme. Since we are interested in understanding the nilpotent structure of the moduli space, we therefore restrict our attention to the case 
where $\Delta$ is reduced.  

Many of our techniques apply more generally to the case where $\Delta$ has arbitrary coefficients, or we consider an integrable distribution together with a boundary with arbitrary coefficients, but in this setting we only capture the behaviour of families of foliations over normal bases.

\subsection{Comparison between moduli of stable foliations and of stable varieties.}
\label{s_intro_challenges}
Even though our approach to the construction of a moduli space for stable foliations is inspired by the moduli theory of stable varieties, 
there are several key ways in which the situation for stable foliations differs from the theory of moduli for stable curves and varieties.

As the analogue of the Abundance Conjecture fails already for rank one foliations on surfaces, the canonical divisor of the foliation does not in general define a polarisation (see Example \ref{not_ample_example}), and so a naturally defined substitute polarisation must be decided upon.
    As finite generation holds for sufficiently small (positive) perturbations of $K_{\mathcal F}$ in the $K_X$ direction, cf. 
    \cite{SS23}, we can handle this issue by considering foliated surfaces polarised by 
    $K_{\mathcal F}+\epsilon K_{X}$ for all $0<\epsilon\ll 1$ 
    (cf. condition (\ref{intro_cond_polarization}) in Section \ref{s_intro_def_mod_functor}).
    We remark that for such choice of polarisation many of the good properties of ample log divisors (e.g., effective generation, effective cohomological vanishing, cf. \cite{PS16}) hold and play an important role in the realization of our moduli construction, cf. Section \ref{s_intro_role_MMP}
 
 Flatness alone is not sufficient to guarantee that a family of foliations is ``reasonable'' from the perspective of moduli (see Example \ref{example_index}). Some additional conditions on our family of foliations are required.
    In the construction of moduli of stable varieties, this issue is overcome by introducing the so-called Koll\'ar condition, which is also an integral part of our definition of moduli functor (cf. conditions (\ref{intro_cond_Qfactorial})-(\ref{intro_cond_kollar_condition}) in Section \ref{s_intro_def_mod_functor}).  In our situation, though,  it is moreover necessary that we also fix the index of the canonical divisor of the foliation to ensure the local closedness of our Koll\'ar-type conditions.
    This is a remarkable difference with the moduli of stable varieties, where the index of the canonical divisor in stable families can be bounded in terms of the canonical volume of the fibres of the family.

Already when constructing the moduli spaces of stable surfaces (or higher dimensional stable varieties), it becomes evident that in order to obtain a proper functor the appearance of non-normal (semi-log canonical) surfaces must be allowed.
Moreover, the moduli picture diverges significantly from the curve case:
there is no longer an ``interior'' to the moduli space  consisting of smooth varieties, as is the case for $\overline{\mathcal M_g}$.  
In fact, there are already components of the moduli space of stable surfaces such that every parametrised variety is not even normal.
Similarly, for the moduli functor defined in Section \ref{s_intro_def_mod_functor}, we must consider semi-log canonical foliations on semi-log canonical surfaces (cf. condition 
(\ref{intro_cond_slc}) in Section \ref{s_intro_def_mod_functor}).
While this is not a very surprising aspect of our definition, at least for those familiar with the moduli theory of stable varieties, it necessitates our initiation of the development of the birational theory of semi-log canonical foliated surfaces.

\subsection{Foliation vs. integrable distribution}
\label{note_on_terminolgy}
Another important aspect of the moduli functor to keep in mind is the following.
On a normal variety a foliation is defined as a saturated subsheaf of the tangent sheaf which is closed under Lie bracket (see Section \ref{s_integrable_distribution}), easy examples show that a family of foliations can degenerate to a subsheaf of the tangent sheaf which is closed under Lie bracket but not saturated (see, e.g., Example \ref{easiest_example}).
In the literature such a subsheaf is usually referred to as an integrable distribution.
This phenomenon explains why in Section \ref{s_intro_def_mod_functor} we define stable families of integrable distributions.
In order to be consistent with the existing literature on the birational geometry of foliations, we will adopt this terminology and reserve the term foliation for saturated integrable distributions.  This terminology also seems preferable to us when working on non-normal varieties.  We refer to \cite{CS23b} or Section \ref{s_integrable_distribution} below for the definition of an integrable distribution on a non-normal variety.

\subsection{Role of the Minimal Model Program}
\label{s_intro_role_MMP}
Starting from  \cite{KSB88}, and through the contribution of many others, work on the moduli of higher dimensional stable varieties has highlighted the importance of techniques from the MMP in the study of moduli problems.  
Thanks to recent developments in the MMP for foliations on threefolds, see \cite{McQ05, Spicer20, CS20, CS21, SS22, 2025basepointfreenessrank}, we are in a position to successfully apply the ideas from the MMP in many of the steps involved in the construction of the moduli space of foliations on surfaces.

In order to prove that the moduli functor
$\mathcal{M}^{2, 1}$
satisfies the valuative criteria of properness and separatedness, we must show that stable families over 1-dimensional bases can be completed in a unique way using stable foliated surfaces with the same fixed invariants.
To achieve that, we follow the same principles that were used to construct the moduli of stable varieties, and we run a carefully devised sequence of runs of the MMP, both with respect to the canonical divisor of the foliation and also that of the total space of a semistable completion of the family.

Another essential use of the MMP in the proof of our main result is in demonstrating the boundedness of foliated surfaces parametrized by 
$\mathcal{M}^{2, 1}_{N, v}$.
Here, once again, our choice of polarisation via the perturbed log divisors of the form
$K_{\mathcal F}+\epsilon K_X$
proves to be crucial.
In fact, such choice allows us to turn the boundedness problem for foliated pairs 
$(X, \mathcal F)$ 
into a boundedness problem for generalized pairs
$(X, K_{\mathcal F})$
obtained by merely considering 
$K_\mathcal{F}$, which is a nef divisor, as the moduli part of the generalized pairs; such modified boundedness problem is set up and positively resolved in Section \ref{s_boundedness}.
To this end, it is important to notice that also the required bound of the Cartier index of $K_\mathcal{F}$ by an integer $N$, in addition to fixing the volume of $5NK_\mathcal{F}+K_X$, is necessary to proving the boundedness:
indeed, it is possible to construct unbounded collections of foliated pairs of fixed volume (and unbounded index), cf. Example~\ref{example_index} for a local example of this behaviour.

As the MMP is often quite technical and the literature can be difficult to navigate; for the reader's convenience we have summarised in Appendix \ref{mmp} the main results from the MMP that we will use.

We remark that the construction of proper moduli of varieties (or foliations) in dimension $=n$ requires the MMP and resolution of singularities in dimension $=n+1$.  We are therefore restricted by the current state of knowledge of both the MMP and resolution of singularities to considering moduli of foliated surfaces.

\subsection{Future directions}
We now list some connections with the work in this paper and other directions of study on foliations and related geometric structures.

{\it Projectivity of coarse moduli.}
If properness of the coarse moduli $\mathsf M^{2, 1}_{N, v}$ of $\mathcal M^{2, 1}_{N, v}$ can be decided upon, it is desirable to know whether $\mathsf M^{2, 1}_{N, v}$ is in fact a projective scheme (as opposed to a proper algebraic space).  
\begin{question}
If $\mathsf M^{2, 1}_{N, v}$ is proper, is $\mathsf M^{2, 1}_{N, v}$ a projective scheme?
\end{question}

In the case of moduli of stable varieties, it is known that the moduli space has projective coarse moduli thanks to \cite{KP17}.  These results rely on positivity properties of variations of Hodge structures associated to families of varieties.  It is unclear if there are analogous positivity statements for families of foliations, and so it is unclear if the approach taken in \cite{KP17} generalises to the setting considered here.

{\it Unfoldings of foliations.}
There is another notion of deformation of a foliations, namely, that of an unfolding.  
Roughly speaking, an unfolding of a family of foliations $f\colon (X, \mathcal F) \to T$ is a foliation $\widetilde{\mathcal F}$ on $X$ such that $T_{X/T} \cap T_{\widetilde{\mathcal F}} = T_{\mathcal F}$.  Not every family of foliations admits an unfolding.  Intuitively, a family of foliations which admits an unfolding is  a family where the solutions (or leaves) of the foliation vary smoothly in the family.  Unfoldings are the analogue for general foliations of isomonodromic deformation of connections on vector bundles.

While we do not study unfoldings here we expect that the moduli space of foliations is equipped with a foliation which represents the ``universal'' unfolding, cf. \cite{genzmer}.
In general, this foliation representing the universal unfolding will not have algebraic leaves.  This indicates that we cannot replace our notion of family of foliation with unfoldings 
and still hope to get a good moduli functor.

{\it (Complete) Moduli of fibred surfaces.}
There has been much interest in constructing and compactifying moduli spaces of fibred surfaces, 
see for instance \cite{Alexeev96, MR1862797, MR1849290, MR4621922, MR4137071, MR3961332, MR3985107}.
In general, a stable (resp. elliptic) fibration can easily be seen to degenerate to something other than a stable (resp. elliptic) fibration.  It is therefore an interesting question to determine the best way to compactify the moduli space of such fibrations.
This compactification have been previously studied either by considering the space of stable maps, or utilising KSB techniques.

Our main theorem provides an alternate compactification of the moduli of fibred surfaces, by reinterpreting a fibred surface as surface together with a foliation.
For instance, if  $S$ is a smooth projective curve and $X \to S$ is a family of a stable curves, if we denote by $\mathcal F$ the induced foliation, then it is easy to see that $(X, \mathcal F)$ gives an example of a stable foliation.
Investigating these alternate compactifications will be the subject of future work

{\it Alterations of foliations.}  In general, it is not possible to resolve a singular foliation to a smooth foliation using blow ups.  Rather, the best that one can hope for is to reduce foliation singularities to singularities in some distinguished class via a sequence of blow ups.  For co-rank one foliations such a distinguished class is given by the class of simple singularities, and when the dimension of the underlying variety is $\leq 3$ the existence of a reduction of singularities is known thanks to Seidenberg, \cite{Seidenberg68}, and Cano, \cite{Cano}.  However, in higher dimensions such a reduction statement is still unknown.   

The work of de Jong, \cite{Jong96}, suggests a different approach to reduction statements through alterations. The existence of the moduli space of stable curves plays a key role in de Jong's approach to alterations.  We expect that the moduli space we have produced here to play a similar role in a potential approach to reducing co-rank one foliation singularities in arbitrary dimension using alterations.

{\it Other notions of stability.}
We remark that there are likely to be several other reasonable candidates for ``stable'' foliation: these alternate definitions could depend on the context and kind of intended applications.  Studying these other notions of stability and their relation to the notion of stability considered here will be the topic of future work.

Of particular importance, in \cite[\S III.3]{McQuillan08} a definition of canonical model of a foliation is proposed which contracts elliptic Gorenstein leaves appearing on the minimal model of the foliation to points.  In general, the canonical model in the sense of \cite{McQuillan08} is not equipped with a natural polarisation (and so differs from the notion of stable foliation considered here), however it would be desirable to produce a moduli functor 
which parametrises these canonical models.

\subsection{Structure and sketch of proof}

In broad outline, the proof of our main theorem follows the proof of the existence of $\mathcal M_g$ and its compactification $\overline{\mathcal M_g}$.
However, owing to unique features in the birational geometry of foliations, there are considerable technical obstacles to realising our main goal.

Our first step is to show that (after fixing some appropriate numerical invariants) the set of stable surface foliations with this fixed set of invariants is bounded. This is done in Corollary \ref{boundedness.slc.triples.cor}. 

Once boundedness is established, it is then easy to see that the $\mathbb C$-valued points points our moduli functor all arise as $\mathbb C$-valued points of some fixed Hilbert scheme, $H$, of finite type (or more precisely a space parametrising quotient husks, ${\rm QHusk}$). 

A priori, there are more points in $H$ than the points corresponding 
to stable surface foliations, and so we next need to show that the set 
$H^{{\rm s}} \subset H$ corresponding to stable surface foliations is a locally closed subset.  This is done in Corollary \ref{cor_stab_rep}.  There are three main ingredients at play here:
\begin{itemize}
    \item We need to show that our singularity condition (i.e., semi-log canonical) is an open condition in families.  This was alluded to earlier, and is shown in Theorem \ref{thm_lc_deform}.

    \item We need to show that our positivity condition, namely $K_{\mathcal F}+tK_X$ is ample for all $0<t \ll 1$ is open in families.  This follows as a consequence of the $K_{\mathcal F}$-MMP.

    \item We need to show that the set of points corresponding to surface foliations where $K_X$ and $K_{\mathcal F}$ are $\mathbb Q$-Cartier is locally closed.  This is surprisingly subtle and points to why the bound on the Cartier index of $K_{\mathcal F}$ is essential, cf. Example \ref{example_index}.
\end{itemize}

Once this is done, our moduli space can be constructed as an appropriate quotient
of $H^{s}$.  Some care needs to be taken to ensure that this quotient represents any $T$-valued point of the moduli functor (and not just the $\mathbb C$-valued points).  The finiteness of stabilisers for this group action is handled in Theorem \ref{thm_separatedness}, and so we conclude that our moduli functor is represented by a Deligne-Mumford stack.

Finally, it remains to show that our moduli stack is proper.  This is handled in Theorem \ref{thm_compactify_to_family}.  The main ingredients here are the following.
\begin{itemize}
    \item A locally stable reduction theorem for families of foliations: this is proven in Theorem \ref{thm_locally_stable_reduction}.

    \item Existence of log canonical closures for families of stable foliations:  this follows from the existence of the $K_{\mathcal F}$-MMP and results on the existence of log canonical closures of generalised pairs.
\end{itemize}
These items show that if given a pointed smooth curve $c \in C$ and a family of stable foliations $(X^\circ, \mathcal F^\circ) \to C \setminus \{c\}$, then we can extend this family of a family of stable foliations $(X, \mathcal F) \to C$.

\subsection*{Acknowledgements}

We would like to especially thank M. McQuillan for many important discussions and suggestions on this work.

We would also like to thank F. Bernasconi, P. Cascini, B. Deroin, J. Liu, S. Filipazzi, Zs. Patakfalvi, J. V. Pereira and Q. Posva for many useful discussions. We thank S. Druel, G. Inchiostro and Q. Posva for comments on an earlier draft of this work. 

The first and third authors were partially supported by EPSRC.
The second author was supported by the “Programma per giovani ricercatori Rita Levi Montalcini” of MUR and by PSR 2022 – Linea 4 of the University of Milan. 
He is a member of the GNSAGA group of INDAM.

\section{Preliminaries}

We will work over $\mathbb C$ throughout this paper.  All schemes are assumed to be Noetherian, separated and excellent, unless otherwise stated.

\subsection{Divisorial sheaves, Mumford divisors and relative Mumford divisors} \label{Prelim. sheaves}
We recall the notion of (relative) Mumford divisors, 
see \cite[\S 4.8]{modbook} for a more complete discussion.

Let $X$ be a reduced scheme. By a {\bf Mumford divisor} $D$
we mean a Weil divisor on $X$ such that $X$ is regular at 
all the generic points of $D$.

Let $f\colon X \to S$ be a flat morphism with $S_2$ fibres.
A subscheme $D \subset X$ is a {\bf relative Cartier divisor}
if $D$ is flat over $S$ and $D|_{X_s}$ is a Cartier divisor on the fibre $X_s$.

Let $f\colon X \rightarrow S$ be a flat morphism 
with $S_2$ fibres. By an {\bf effective relative Mumford divisor} $D$ we mean a subscheme $D \subset X$ such that there exists a closed subset $Z \subset X$ satisfying the following properties:
\begin{enumerate}
\item ${\rm codim}_{X_s}(Z \cap X_s) \ge 2$ for all $s \in S$; 

\item $D|_{X \setminus Z}$ is a relative Cartier divisor;

\item $D$ is the closure of $D|_{X\setminus Z}$ in $X$; and 

\item $X_s$ is regular at all the generic points of $D\vert_{X_s}$. 
\end{enumerate}
By a {\bf relative Mumford divisor} we mean a $\mathbb Z$-linear combination of effective relative Mumford divisors.

Let $X$ be a scheme.  By a {\bf divisorial sheaf} $L$
we mean a coherent $S_2$ sheaf such that there exists a closed subset $Z \subset X$ of codimension $\ge 2$ such that $L\vert_{X \setminus Z}$ is a line bundle. 

Let $f\colon X \rightarrow S$ be a morphism. 
We say that a coherent sheaf $L$ on $X$ is a {\bf flat family of divisorial 
sheaves} if $L$ is flat over $S$ and $L\vert_{X_s}$ is a divisorial sheaf for all $s \in S$.

By \cite[Proposition 4]{Kollar2018}, if $X$ is reduced the group 
of Mumford divisors is isomorphic to the the group
of divisorial sheaves, and so we will occasionally use these notions interchangeably.
There is likewise a close connection between relative Mumford divisors and flat families of divisorial sheaves.  Indeed, given a flat family of divisorial sheaves $L$ and a global section $f$ of $L$ 
which does not vanish at the generic points of $X_s$ or the codimension one points of 
$\sing X_s$ for any $s \in S$, 
the equation $(f = 0)$ defines a relative effective Mumford divisor.

\subsection{Integrable distributions}
\label{s_integrable_distribution}
The following definition of an integrable distribution is found in \cite[\S 2.2]{CS23b}.

Let $X$ be an $S_2$ scheme.
A rank $r$ {\bf integrable distribution} $\mathcal F$ on $X$ is the data of  
\begin{enumerate}
\item a divisorial sheaf $L$; and
\item a Pfaff field, i.e., a morphism $\phi\colon \Omega^r_X \rightarrow L$,
satisfying the following integrability condition:
in some neighbourhood $U$
of the generic point of each irreducible component of 
$X$ there exists a coherent rank $r$ sheaf $E$
and a surjective morphism $q\colon \Omega^1_U \rightarrow E$ such that the $r$-th wedge
power of this morphism agrees with $\phi|_U$ 
and its dual $E^*\hookrightarrow T_U$ is closed under Lie bracket. 
\end{enumerate}

We define the {\bf canonical class} of the integrable distribution $\mathcal F$ to be any Mumford divisor
$K_{\mathcal F}$ on $X$
such that $\mathcal O_X(K_{\mathcal F}) \cong L$.
  A rank $r$ {\bf foliation} on $X$ is a rank $r$ integrable distribution on $X$ whose Pfaff field $\phi$ is such that ${\rm coker}~ \phi$ 
is supported in codimension at least two. 
 Given a rank $r$ integrable distribution $\mathcal F$ on a normal scheme $X$ 
we define the {\bf singular locus} of $\mathcal F$, denoted $\sing \mathcal F$, to be the co-support
of the ideal sheaf defined by the image of  the induced map 
$(\Omega^r_X\otimes \mathcal O_X(-K_{\mathcal F}))^{**} \rightarrow \mathcal O_X$
Given an irreducible Mumford divisor $D \subset X$ we define the indicator function $\varepsilon(D)$ by $\varepsilon(D) = 1$ if $D$ is not $\mathcal F$ invaiant and $=0$ otherwise.

Let $X$ be a normal variety and let $\mathcal F$ be an integrable distribution 
on $X$ such that $K_{\mathcal F}$ is $\mathbb Q$-Cartier.
Given a birational morphism $\pi\colon X' \rightarrow X$
we define the transform of $\mathcal F$, denoted $\pi^{-1}\mathcal F$
to be the unique integrable distribution $X'$ such that no $\pi$-exceptional 
divisor is contained in $\sing \pi^{-1}\mathcal F$ and such that 
$\pi^{-1}\mathcal F$ agrees with $\mathcal F$ on the dense open subset where $\pi$ is an isomorphism.

\begin{construction}
\label{fol_to_dist}
We recall that given a normal variety $X$ and an integrable distribution $\mathcal F$
there is an associated foliation $\tilde{\mathcal F}$ and a canonically defined Weil divisor $D \ge 0$ such that $K_{\mathcal F} = K_{\tilde{\mathcal F}}+D$, see \cite[Lemma 2.2]{CS23b}.  We refer to the
pair $(\tilde{\mathcal F}, D)$ as the {\bf foliated pair associated to} $\mathcal F$.

Conversely, given an integrable distribution $\mathcal F$ of rank $r$ 
with associated Pfaff field $\Omega^r_X \rightarrow \mathcal O(K_{\mathcal F})$
and an integral Mumford divisor $D \ge 0$
we can define an integrable distribution $\mathcal F'$ by taking the Pfaff field
$\Omega^r_X \rightarrow \mathcal O(K_{\mathcal F}+D)$ induced by the natural inclusion
$\mathcal O(K_{\mathcal F}) \rightarrow \mathcal O(K_{\mathcal F}+D)$.  It is easy to verify that these two constructions are inverses of each other.
\end{construction}

\begin{remark}
\label{rmk_log_dist_is_dist}
We could define a log integrable distribution on a normal variety 
$X$ 
to be the data
$(\mathcal F, D)$
of a integrable distribution
$\mathcal F$ on $X$ and an effective Mumford divisor $D$ on $X$.
However, by the above construction, we see that the category of log integrable distributions is naturally isomorphic to the category of integrable distributions.  We will therefore freely pass between the two notions.
\end{remark}

For a rank one foliation on normal variety it is often convenient to consider a slightly different definition of singular locus.

Let $X$ be a variety and let $\partial$ 
be a vector field on $X$.  We define the singular locus $\sing \partial$
to be the set of points 
\begin{align*}\{x \in X: \partial(\mathcal O_{X, x}) \subset \mathfrak m_x\}\end{align*}
where $\mathfrak m_x$ is the maximal ideal at the point $x \in X$.

Let $X$ be a normal variety, let $\mathcal F$ be a rank one foliation on $X$ such that 
$K_{\mathcal F}$ is $\mathbb Q$-Cartier.  Let $x \in X$ be a point and let $U$ be an open neighbourhood of $x$.  
Up to replacing $U$ by a smaller neighbourhood we may find an index one cover $\sigma\colon U' \rightarrow U$ associated to $K_{\mathcal F}$
and such that $\sigma^{-1}\mathcal F$ is generated by a vector field $\partial$.

We say that $\mathcal F$ is {\bf singular in the sense of McQuillan} at $x \in X$ provided that  $\sigma^{-1}(x) \subset \sing \partial$.
We denote by $\sing^+{\mathcal F}$ the locus of points $x \in X$ where $\mathcal F$ is singular in the sense 
of McQuillan. Note that $\sing^+{\mathcal F}$ does not depend on the choice of $U'$ and it is a closed subset of $X$. 

On a smooth variety it is easy to see that $\sing^+{\mathcal F} = \sing{\mathcal F}$, but in general we only have a containment
$\sing{\mathcal F} \subset \sing^+{\mathcal F}$, see \cite[Lemma 4.1]{CS23b}.

\subsection{Pull-back integrable distributions}
Given a normal schemes $X$ and $X'$, a foliation $\mathcal F$ 
on $X$, and a dominant rational map $p\colon X' \dashrightarrow X$,
there is a well defined pull-back foliation $p^{-1}\mathcal F$, cf.  \cite[\S 2.9]{Druel15b}

If we suppose that $\mathcal F$ is an integrable distribution on a normal variety $X$ and that $p\colon X' \rightarrow X$
is a birational morphism we define the pull-back integrable distribution $p^{-1}\mathcal F$ as follows.
Let $(\tilde{\mathcal F}, D)$ be the associated foliated pair, let $\mathcal G\coloneqq p^{-1}\tilde{\mathcal F}$ be 
the pull-back foliation and let $E \coloneqq p_*^{-1}D$ be the strict transform of $D$.
We then define $p^{-1}\mathcal F$ to be the integrable 
distribution associated to $(\mathcal G, E)$.

Let $X$ and $Y$ be $S_2$ schemes and let $\mathcal F$ be an integrable distribution on $X$. Let $s\colon Y \rightarrow X$ be a quasi-\'etale morphism and let $Z \subset X$ be a codimension two subset such that $s\colon Y \setminus s^{-1}(Z) \rightarrow X\setminus Z$
is \'etale and such that $K_{\mathcal F}\vert_{X\setminus Z}$ is Cartier.
Via the isomorphism $\Omega^1_{Y\setminus s^{-1}(Z)} \cong s^*\Omega^1_{X\setminus Z}$
we may pull-back the Pfaff field defining $\mathcal F$ to a Pfaff field 
\begin{align*}\Omega^1_{Y \setminus s^{-1}(Z)} \rightarrow s^*\mathcal O(K_{\mathcal F})|_{X \setminus Z}.\end{align*}
It is easy to check that this defines an integrable distribution on $Y\setminus s^{-1}(Z)$, and since 
$s^{-1}(Z)$ is codimension at least 2 in $Y$, it follows that this extends uniquely 
to an integrable distribution on all of $Y$ which we denote by $\sigma^{-1}\mathcal F$.
Note that when $X$ is normal this pull-back coincides with the pull-back defined above.

Let $X$ be an $S_2$ scheme and let $\mathcal F$ be an integrable distribution
on $X$. Let $n\colon X^n \rightarrow X$ be normalisation map. Let $Z \subset X$ be a codimension two subset such that $K_{\mathcal F}|_{X \setminus Z}$ is locally free.
By \cite[Lemma 4.3 and Proposition 4.5]{ADK08} (see also \cite{MR188247}) the Pfaff field $\Omega^r_{X \setminus Z} \rightarrow \mathcal O(K_{\mathcal F})|_{X \setminus Z}$ lifts to a Pfaff field 
$\Omega^r_{X^n \setminus n^{-1}(Z)} \rightarrow n^*\mathcal O(K_{\mathcal F})\vert_{X\setminus Z}$.  Again, it is easy to check that this defines an integrable distribution on 
$X^n \setminus n^{-1}(Z)$ and extends to an integrable distribution 
on $X^n$ which we denote by $n^{-1}\mathcal F$.

It is easy to verify that all these notions of pull-back coincide when mutually well defined.

\subsection{Flat families of integrable distributions}
\label{s_flat_families}
Let $f\colon X \rightarrow T$
be a flat morphism with $S_2$ fibres 
and let $L$ be a flat family of divisorial sheaves.  

A {\bf flat family of integrable distributions of rank $r$},
denoted $f\colon (X, \mathcal F) \rightarrow T$,
is the data of a
morphism \begin{align*}\phi\colon \Omega^r_{X/T} \rightarrow L\end{align*} such that the following hold.
\begin{enumerate}
    \item The restricted morphism \begin{align*}\phi_t\colon \Omega^r_{X_t} \rightarrow L_t\end{align*} defines an integrable distribution for all $t \in T$.
(Notice that $\phi$ is surjective at the generic points of $f^{-1}(t)$ for all $t \in T$).

\item In a neighbourhood of every generic point of $X$ the obvious morphism $\Omega^r_X \to L$ defines an integrable distribution.

\end{enumerate}

We remark that if either $r = 1$ or $T$ is reduced, then Condition (2) is redundant.

We will denote by $K_{\mathcal F}$ any relative Mumford divisor such that $\mathcal O(K_{\mathcal F}) = L$.

Notice that given a flat family of integrable distributions $(X, \mathcal F) \rightarrow T$
we have a natural morphism 
$\Omega^r_X \rightarrow L$.  If $X$ is $S_2$, then $L$ is a divisorial sheaf by \cite[Proposition 3.5]{MR2047697}.  Thus, this morphism defines an integrable distribution of rank $r$ on all of $X$.

Conversely, suppose that $X$ is $S_2$, $f\colon X \rightarrow T$
is a flat morphism with $S_2$ fibres and $\mathcal F$ is an integrable distribution on $X$ such that $\sing \mathcal F$ does not contain any components of fibres of $f$.  Then, if 
$\mathcal O(K_{\mathcal F})$ is a flat family of divisorial sheaves and the morphism $\Omega^r_X \to \mathcal O(K_{\mathcal F})$ factors 
as $\Omega^r_{X/T} \to \mathcal O(K_{\mathcal F})$, then $(X, \mathcal F) \to T$ is a flat family 
of integrable distributions.

As mentioned earlier we will often need to consider a Weil divisor $\Delta$ together with our underlying variety.  
In these cases we will use the notation $f\colon (X, \Delta, \mathcal F) \rightarrow T$ to denote a flat family of integrable distributions.

\subsection{Classes of singularities in the Minimal Model Program}
\label{s_singularities}

By a {\bf normal crossing singularity} we mean a singularity which is analytically isomorphic to 
$\{x_1x_2\cdots x_j = 0\} \subset \mathbb A^k$ 
for some integer 
$j \leq k$.

We say that a variety $X$ is {\bf deminormal} provided that is is $S_2$ and its codimension one points 
are either smooth or normal crossing singularities.

\begin{definition}
    \label{def.pair}
A {\bf couple} 
$(X, \Gamma)$ 
is the datum of a deminormal variety 
$X$ 
and a Mumford divisor 
$\Gamma$
on 
$X$.

A {\bf pair} 
$(X, \Theta)$ 
is the datum of a deminormal variety 
$X$ 
and an effective Weil 
$\mathbb R$-divisor 
$\Theta$ 
on 
$X$
with coefficients in 
$[0, 1]$
such that 
$(X, \lceil \Theta \rceil)$
is a couple.

A {\bf log pair} 
$(X, \Delta)$ 
is the datum of a pair such that 
$K_X+\Delta$ 
is 
$\mathbb R$-Cartier.

A {\bf foliated triple} 
$(X, \Delta, \mathcal F)$
is the datum of a pair
$(X, \Delta)$
together with an integrable distribution 
$\mathcal F$
on 
$X$ such that $K_{\mathcal F}$ is $\mathbb Q$-Cartier.
\end{definition}

We refer to 
\cite{KM98, Kollar13, CS20}
for complete definitions
on classes of singularities from the perspective of the MMP.
In \cite{CS20} these definitions are stated for foliations and not integrable distributions in general, for the readers convenience we explain these definitions in the case of an integrable distribution.

Given a normal variety $X$ and an integrable distribution $\mathcal F$ on $X$ 
we say that $\mathcal F$ has {\bf terminal (resp. canonical, resp. log terminal, resp. log canonical) singularities} 
provided for every birational morphism $\pi\colon X' \rightarrow X$, when we write 
\begin{align*}K_{\pi^{-1}\mathcal F} = \pi^*K_{\mathcal F}+\sum a(E_i, \mathcal F)E_i\end{align*}
then $a(E_i, \mathcal F) > 0$ (resp. $\geq 0$, resp. $>-\varepsilon(E_i)$, resp. $\ge -\varepsilon(E_i)$).

If 
$(\tilde{\mathcal F}, D)$ 
is the foliated pair associated to the integrable distribution 
$\mathcal F$, 
then 
$a(E, \mathcal F) = a(E, \tilde{\mathcal F}, D)$, 
and thus, 
$\mathcal F$ is log terminal (resp. log canonical) if and only if $(\tilde {\mathcal F}, D)$
is log terminal (resp. log canonical).

Let $P \in X$ be a, not necessarily closed, point of $X$.
We  say  that $(\mathcal F, \Delta)$ is {\bf terminal} (resp. {\bf canonical, log canonical})
{\bf at $P$} if
for any birational morphism $\pi\colon \widetilde X\to X$ and for any  $\pi$-exceptional  divisor $E$ on $\widetilde X$ whose centre in $X$ is the Zariski closure $\overline P$ of $P$,
we have that $a(E, \mathcal F)>0$ (resp. $\geq 0$, $\geq -\varepsilon(E)$).

We say that a pair 
$(X, 
\Delta)$ of deminormal scheme $X$
and a Mumford $\mathbb Q$-divisor is {\bf semi-log canonical (slc)} provided that 
the pair $(X^n, n^*\Delta+D)$ is log canonical where 
$n \colon X^n=\sqcup_{i=1}^{s} X_i^n \to X$ 
is the normalisation and
$D$ 
on 
$X^n$
is the divisorial (counted with multiplicity)
part of the conductor ideal.
As $(X, \Delta)$ is slc, 
$X$ has nodal singularities in codimension one along the image of $D$ on $X$.
We denote by 
$D^n$
the normalization of $D$ 
and by 
$\tau 
\colon 
D^n \to D^n$
the natural gluing involution on the double locus.

Let $X$ be a deminormal scheme and let $\mathcal F$ be an integrable distribution on $X$. We say that $\mathcal F$ is {\bf semi-log canonical (slc) (resp. semi-log terminal (slt))} provided
that $\mathcal F^n$ is log canonical (resp. log terminal)
where $\mathcal F^n:=n^\ast \mathcal F$ is the pullback to the normalisation $n \colon X^n \to X$.

These definitions are equally valid for more general schemes, for instance, in the setting of \cite[Definition 1.5]{Kollar13}.

\subsection{Well-formed integrable distributions}
As noted, we will need to work in the generality of integrable distributions on varieties equipped with a boundary divisor.
Frequently this boundary divisor will differ from the divisor associated 
to the integrable distribution, however, it is often necessary to impose 
a compatibility condition between these two divisors which we now describe.

Consider a triple $(X, \Delta, \mathcal F)$
where $X$ is a deminormal variety, $\Delta$ is a Mumford divisor and $\mathcal F$ is an integrable distribution on $X$.
Let $n\colon X^n \rightarrow X$ be the normalisation, let $(\tilde{\mathcal F^n}, \Gamma)$ be the associated foliated pair to $n^{-1}\mathcal F$ and let $\Delta^n = n^*\Delta+D$ where $D$ is the pre-image of the double locus of $X$.

We say that $(X, \Delta, \mathcal F)$ is {\bf well-formed} if 
$\Gamma \ge \Delta^n_{{\rm non-inv}}$,
where invariance is considered with respect to $\tilde{\mathcal F^n}$. We remark that the well-formed condition is equivalent to asking 
$T_{n^{-1}\mathcal F} \subset T_{X^n}(-\log \Delta^n)$.

\subsection{Locally stable families of integrable distributions}
\label{s_locally_stable_def}
We recall the definition of a locally stable family of pairs $(X, \Delta) \rightarrow T$ (see \cite[Definition-Theorem 4.7]{modbook} or \cite[\S 6.22]{modbook}). 
There are many issues surrounding the correct definition of a locally stable family when $T$ is not reduced and $\Delta$ is allowed to have arbitrary coefficients, but for our purposes it will be enough to consider the case where the coefficients of $\Delta$ are all $=1$ in which case the situation is simpler.

Consider the following set up.
Let 
\begin{enumerate}
    \item $f\colon X \rightarrow T$ be a flat morphism of finite type with deminormal fibres, and
    \item $\Delta = \sum D_i$ where $D_i$ is a reduced relative Mumford divisor and $D_i \to T$ is flat with divisorial subschemes as fibres.
\end{enumerate}

We say that
$f\colon (X, \Delta = \sum D_i) \rightarrow T$ is {\bf locally stable} provided
that $K_{X/T}+\Delta$ is $\mathbb Q$-Cartier
and $(X_t, \Delta_t)$ is semi-log canonical for all 
closed points $t \in T$ (we refer to \cite[\S 2.68]{modbook} for the definition of the relative dualising sheaf- by a slight abuse of notation we will write $K_{X/T}$ instead of $\omega_{X/T}$, cf. \cite[Definition 3.35]{modbook}).

Now, let $f\colon (X, \Delta, \mathcal F) \rightarrow T$ be a flat family of integrable distributions and fix an integer $N>0$.  
We say that $f\colon (X, \Delta, \mathcal F) \rightarrow T$
is {\bf locally stable} (resp. {\bf locally stable of index $=N$})
provided 
\begin{enumerate}
\item[(S1)] $(X, \Delta) \rightarrow T$
is locally stable (in particular $K_{X/T}+\Delta$ is $\mathbb Q$-Cartier);

    \item[(S2)] $K_{\mathcal F}$
is $\mathbb Q$-Cartier (resp. $NK_{\mathcal F}$ is Cartier); 

\item[(S3)] for all $j, k \in \mathbb Z$, $\mathcal O(jK_{\mathcal F}+k(K_{X/T}+\Delta))$ is flat over $T$ and commutes with base change, i.e., if $S \rightarrow T$ is a morphism then $\mathcal O_X(jK_{\mathcal F}+k(K_{X/T}+\Delta))_S \cong  \mathcal O_{X_S}(jK_{\mathcal F_S}+k(K_{X_S/S}+\Delta_S))$, equivalently, $\mathcal O(jK_{\mathcal F}+k(K_{X/T}+\Delta))$ is a flat family of divisorial sheaves; 

\item[(S4)] for all closed points $t \in T$, $\mathcal F_t$
is semi-log canonical; and 

\item[(S5)] for all closed points $t \in T$, $(X_t, \Delta_t, \mathcal F_t)$ is well-formed.
\end{enumerate} 

We say that a family of integrable distributions
is {\bf stable} (resp. {\bf stable of index} $=N$) if in addition the following holds:
\begin{enumerate}
    \item[(S6)] there exists $\eta \in \mathbb R_{>0}$ such that $K_{\mathcal F}+\epsilon(K_{X/T}+\Delta)$ is $f$-ample for all $0< \epsilon < \eta$.
\end{enumerate}

If $f\colon (X, \Delta, \mathcal F) \to T$ is locally stable (resp. stable) where $T$
is a point, then we will say that $(X, \Delta , \mathcal F)$ is a locally stable (resp. stable) foliated triple.

\begin{proposition}
\label{prop_threshold}
    Suppose that $(X, \Delta, \mathcal F)$ is stable of index $=N$.  Then 
    \begin{enumerate}
        \item $K_{\mathcal F}$ is nef; and 
        \item    $(2\dim X+1)NK_{\mathcal F}+(K_X+\Delta)$ is ample.
    \end{enumerate}
\end{proposition}
\begin{proof}
    Without loss of generality we may replace $X$ by its normalisation and so may assume that $X$ is normal.

Since $mK_{\mathcal F}+(K_X+\Delta)$
is ample for all $m \gg 0$ we deduce that 
\begin{itemize}
    \item $K_{\mathcal F}$ is nef (which proves item (1)); and
    \item if $\gamma \in \overline{NE}(X)$ is any non-zero curve class such that $(K_X+\Delta)\cdot \gamma\le 0$
    then $K_{\mathcal F}\cdot \gamma >0$.
\end{itemize}

   Now, suppose for sake of contradiction that $(2\dim X+1)N K_{\mathcal F}+(K_{X}+\Delta)$ is not ample.
   By Kleiman's criterion there exists an extremal ray $R$
   on which $(2\dim X+1)N K_{\mathcal F}+K_{X}+\Delta$
   is non-positive.  By our previous two points we see that this is only possible if $R$ is  $K_{X}$-negative and $K_{\mathcal F}$-positive.
   
   The Cone Theorem 
\cite[Theorem 5.6]{MR4298910} guarantees that there exists a curve $C \subset X$ such that 
   $R = \mathbb R_+[C]$ and $0<-(K_{X}+\Delta)\cdot C \leq 2\dim X$.  Since $K_{\mathcal F}\cdot C>0$ we have $(2\dim X+1)N K_{\mathcal F}\cdot C \geq 2\dim X+1$ and so $((2\dim X+1)N K_{\mathcal F}+(K_{X}+\Delta))\cdot C > 0$, a contradiction. 
\end{proof}

\subsection{Definition of moduli functor}
\label{s_precise_definition}
In this section we define precisely the moduli functors that will interest us.

We first define our notion of equivalence between two families of integrable distributions. 
Consider two families of integrable distributions of rank $r$ over a scheme $T$, 
$f_1\colon (X_1, \Delta_1, \mathcal F_1) \to T$ and 
$f_2\colon (X_2, \Delta_2, \mathcal F_2) \to T$.  Let $\phi_i\colon \Omega^r_{X_i/T} \to L_i$ be the Pfaff field defining the family of foliations $f_i\colon (X_i, \Delta_i, \mathcal F_i) \to T$.
We say that $f_1 \colon (X_1, \Delta_1, \mathcal F_1) \to T$ is equivalent to $f_2 \colon (X_2, \Delta_2, \mathcal F_2) \to T$,  provided there exists an isomorphism of $T$-schemes
$F\colon X_1 \to X_2$ such that 
\begin{itemize}
    \item $F^*(\Delta_2) = \Delta_1$; and 
    \item there is an isomorphism of divisorial sheaves $\psi\colon F^*L_2 \to L_1$
    making the following diagram commute
    \begin{center}
        \begin{tikzcd}
F^*\Omega^r_{X_2/T} \arrow[r, "F^*\phi_2"] \arrow[d, "F^*"] & F^*L_2 \arrow[d, "\psi"] \\
\Omega^r_{X_1/T} \arrow[r, "\phi_1"] & L_1
        \end{tikzcd}
    \end{center}
    where $F^*\colon F^*\Omega^r_{X_2/T} \to \Omega^r_{X_1/T}$ is the natural isomorphism.
\end{itemize}

Fix $d, r, N \in \mathbb Z_{>0}$ and $v \in \mathbb R_{>0}$. 
Let $T$ be a $\mathbb C$-scheme and let $f\colon (X, \Delta, \mathcal F) \to T$
be a flat family of integrable distributions over $T$.
Consider the following conditions:
\begin{enumerate}
    \item[(M1)] $f\colon (X, \Delta, \mathcal F) \to T$ is a stable family of integrable distributions of rank $r$; 
    \item[(M2)] $\dim (X/T) = d$;
    \item[(M3)] $\Delta$ is a reduced divisor;
    \item[(M4)] $NK_{\mathcal F}$ is Cartier; and 
    \item[(M5)] for all $t \in T$ 
    \begin{align*}{\rm vol}(K_{X_t}+\Delta_t+5NK_{\mathcal F_t})=v.\end{align*}
\end{enumerate}

We define
\begin{align*}
\mathcal M^{d, r}_{N, v}(T) = \{(X, \Delta, \mathcal F) \to T: 
\text{flat families of integrable } \\
\text{distributions satisfying conditions 
(M1) - (M5) above}\}/\sim
\end{align*}
where $\sim$ is the notion of equivalence established above.
Similarly, we define 
\begin{align*}
\mathcal M^{d, r}(T) = \{(X, \Delta, \mathcal F) \to T: 
\text{flat families of integrable } \\
\text{distributions satisfying conditions 
(M1) - (M3) above}\}/\sim
\end{align*}

The following proposition guarantees that our proposed moduli functor is, in fact, a functor.

\begin{proposition}
\label{prop_base_change}
Let \begin{align*}f\colon (X, \Delta, \mathcal F) \rightarrow T\end{align*}
be a locally stable (resp. stable) family of integrable distributions.  Let $S \rightarrow T$ be a morphism and let $X_S = X \times_TS$ with projection $p\colon X_S \to X$.

Then there is a well defined base change family of integrable distributions
\begin{align*}f_S\colon (X_S, \Delta_S, \mathcal F_S) \rightarrow S\end{align*}
which is a locally stable (resp. stable).
Moreover, if $(X, \Delta, \mathcal F) \to T$ satisfies conditions (M1)-(M5) above, 
then $(X_S, \Delta_S, \mathcal F_S) \to S$ satisfies conditions (M1)-(M5).
\end{proposition}
\begin{proof}
We first recall that the base change of a locally stable family of varieties is again locally stable, see \cite[Theorem 3.2]{modbook}.  

We next claim that 
 the base change of a flat family of integrable distributions is again a
flat family of integrable distributions.
Indeed, note that $L_S = p^*L$ is a flat family of divisorial sheaves
(see \cite[Definition 3.25]{modbook})
and $(\Omega^r_{X/T})_S \cong 
\Omega^r_{X_S/S}$.
Thus we have a natural morphism $\Omega^r_{X_S/S} \rightarrow L_S$
which defines our base change flat family of integrable distributions, 
$(X_S, \mathcal F_S) \rightarrow S$.

To see that $(X_S, \Delta_S, \mathcal F_S) \to S$ is locally stable, observe that 
Condition (S1) holds by what has already been observed, Condition (S3) holds by definition and Condition (S2) is implied by Condition (S3) and the fact that $K_{\mathcal F}$ is $\mathbb Q$-Cartier.  Conditions (S4) and (S5) (resp. Conditions (S4), (S5) and (S6)) hold because they may be checked fibrewise.

The moreover part follows because Conditions (M2)-(M5) may all be checked fibrewise.
\end{proof}

\section{Boundedness of foliated pairs in dimension two}
\label{s_boundedness}

The goal of this section is to prove some boundedness results for foliations that will be used in the construction of the moduli functor that we proposed in the previous section.
Our proof will utilise the notion and theory of generalised pairs; the reader interested in more details and results than those we provide here may want to refer to, for example, 
\cite{birkar-zhang}.

 \subsection{Notions of boundedness}
Let us recall the following basic definitions of boundedness for pairs and foliations.

\begin{definition}
\label{def.boundedness}
    A bounded family of proper two dimensional couples 
    (resp. pairs, log pairs) 
    is the datum of a couple 
    (resp. a pair, a log pair) 
    $(\mathcal X, \Delta_{\mathcal X})$
    together with a proper flat morphism 
    $h 
    \colon 
    \mathcal X 
    \to 
    T$
    of finite type varieties such that 
    $\Delta_{\mathcal X}$
    contains no fiber of 
    $h$, 
    and for all 
    $t \in T$, 
    $(\mathcal X_t, \Delta_{\mathcal X, t})$
    is a couple
    (resp. a pair, a log pair), 
    where 
    $\Delta_{\mathcal X, t} := \Delta_{\mathcal X}\vert_{\mathcal X_t}$.

    A bounded family of integrable distributions is the datum of a triple 
    $(\mathcal X, \Delta_{\mathcal X},   \mathcal G)$
    together with  a proper flat morphism with deminormal fibres
        $h \colon \mathcal X \to T$ such that
    \begin{itemize}  
        \item $(\mathcal X, \Delta_{\mathcal X}) \to T$
        is a bounded family of couples; and

        \item $(\mathcal X, \mathcal G) \to T$ is a flat family of integrable distributions of rank $r$.

    \end{itemize}

\end{definition}

\begin{definition}
    We say that a collection 
    $\mathfrak D$
    of two-dimensional proper foliated triples is 
    log bounded (resp. strongly log bounded)
    if there exists a bounded family of normal foliated triples 
    $(\mathcal X, \Delta_{\mathcal X}, \mathcal G)$
    such that for all foliated triples
    $(X, \Delta, \mathcal F) \in \mathfrak D$, 
    there exists 
    $t \in T$
    and an isomorphism
    $\psi \colon \mathcal X_t \to X$
    such that 
    $\psi^\ast \mathcal F=\mathcal G_t$, 
    and 
    $\psi_\ast {\rm Supp}(\Delta_{\mathcal X, t})={\rm Supp}(D)$
    (resp.
    $\psi_\ast \Delta_{\mathcal X, t}=D$).

\end{definition}

Let us introduce also the necessary definition of boundedness for generalized pairs.

\begin{definition}
\label{def.boundedness.gen.pair}
    We say that a collection 
    $\mathfrak E$
    of two-dimensional projective deminormal
    generalized pairs is log bounded (resp. strongly log bounded) if there exists a bounded family of deminormal couples
    $(\mathcal X, \Delta_{\mathcal X})$
    such that for all generalized pairs 
    $(X, \Delta, M) \in \mathfrak E$, 
    there exists 
    $t \in T$
    and an isomorphism
    $\psi \colon \mathcal X_t \to X$
    such that 
    $\psi_\ast {\rm Supp}(\Delta_{\mathcal X, t})={\rm Supp}(\Delta)$
    (resp.
    $\psi_\ast \Delta_{\mathcal X,t}=\Delta$). 
    Moreover, we require the existence of a relatively ample divisor $\mathcal H$ on $\mathcal X$ and a positive integer $C$, independent of $t$,  
    such that 
    $\mathcal H\vert_{\mathcal X_t} \cdot \psi^\ast M \leq C$.
\end{definition}

\subsection{Boundedness results}
We are now ready to prove the main boundedness results of this section.

We first state a general boundedness result for certain classes of generalized pairs on surfaces which we will prove at the end of this subsection.
We then explain how to deduce an associated boundedness result for minimal models of foliations that are polarised once we perturb the foliated canonical divisor in the direction of the canonical divisor of the underlying variety.

\begin{theorem}
\label{boundedness.glc.pairs.thm.finite.coeff}
Fix 
$v \in \mathbb Q_{>0}$, 
$N \in \mathbb Z_{>0}$, 
and 
$I \subset \mathbb Q \cap [0, 1]$
a finite set.
\\
The set
$\setgpairsdcc{lc}$
of generalized pairs
$(X,  \Delta, M)$ 
such that
\begin{enumerate}
\item 
\label{cond.normal.thm.bound}
$X$ 
is a normal projective surface, 
\item 
$\Delta$ 
is a Weil divisor on 
$X$ 
whose coefficients are in 
$I$,
\item 
$M$ is a nef divisor
 
\item
\label{cond.index.thm.bound}
$NM$ is Cartier, 
\item 
$(X, \Delta, M)$ is generalized log canonical,
where the moduli b-divisor coincides with the Cartier closure of $M$,  

\item 
\label{cond.nef.thm.bound}
there exists $\eta \in \mathbb R_{>0}$ such that 
$M+\epsilon (K_X+\Delta)$ is ample for all $0< \epsilon < \eta$, and
\item 
\label{cond.vol.thm.bound}
${\rm vol}(K_X+\Delta+5NM)=v$

\end{enumerate}
is bounded.
\end{theorem}

We postpone the proof of the above theorem until the end of this section, and we first explain its applications to the boundedness of foliations.
The first consequence is the following corollary which can be deduced using the above theorem and effective very ampleness results.

\begin{corollary}
\label{boundedness.lc.triples.cor.finite.coeff}
Fix 
$v \in \mathbb Q_{>0}$, 
$N \in \mathbb Z_{>0}$, 
and 
$I \subset \mathbb Q \cap [0, 1]$
a finite set.
The set
$\setfoliation{lc}$
of triples 
$(X, \Delta, \mathcal F)$ 
that satisfy the following properties
\begin{enumerate}
\item 
\label{cond.normal.cor.bound}
$X$ 
is a normal projective surface, 
\item 
$\Delta$ 
is a Weil divisor on 
$X$ 
whose coefficients are in 
$I$,
\item 
$\mathcal F$ is an integrable distribution on $X$,

\item 
$(X, \Delta)$ is log canonical, 
\item 
$(\mathcal F, 0)$ is log canonical,
\item
\label{cond.index.cor.bound}
$NK_{\mathcal F}$ is Cartier, 
\item 
\label{cond.nef.cor.bound}
there exists $\eta \in \mathbb R_{>0}$ such that 
$K_\mathcal{F}+\epsilon (K_X+\Delta)$ is ample for all $0< \epsilon < \eta$, and
\item 
\label{cond.vol.cor.bound}
${\rm vol}(K_X+\Delta+5NK_{\mathcal F})=v$

\end{enumerate}
is strongly log bounded.
In particular, there exists a 
$m=m(N) \in \mathbb Z_{>0}$, 
such that for any triple 
$(X,\Delta , \mathcal{F}) \in \mathfrak{F}_{v, N}^{\rm adj, lc}$, 
$m(K_X+\Delta+5NK_\mathcal{F})$ is a very ample Cartier divisor.
\end{corollary}

\begin{proof}
		By Theorem~\ref{boundedness.glc.pairs.thm.finite.coeff}, 
		$(X, \Delta, K_\mathcal{F}) 
		\in \setfoliation{lc}$
		which is bounded.
		In particular, by the proof of Theorem~\ref{boundedness.glc.pairs.thm.finite.coeff}, 
		cf. Step 4-5, there exists 
		$k \in \mathbb Z_{>0}$,
		independent of the chosen foliated triple in 
		$\setfoliation{lc}$,
		such that  
		$k(K_X+\Delta)$ is Cartier.
		Thus, the generalized log divisor
		$k(K_X+\Delta+5NK_\mathcal{F})$ 
		is ample and Cartier.
		By
		\cite[Theorem~1.1]{fujino-2016}, 
		there exists a positive integer 
		$m=m(N)$, 
		independent of the choice of triple in 
		$\setfoliation{lc}$, 
		such that 
		$H_X=mk (K_X+\Delta+5NK_\mathcal{F})$
		is very ample of volume 
		$(mk)^2v$.
		That implies that the collection of varieties 
		\begin{align*}
		\mathcal S :=
		\left \{
		X \;
		\middle \vert \;
		X \in \setfoliation{lc}
		\right \}
		\end{align*}
		is bounded.
		
		There exists another positive integer 
		$m_1=m_1(N)$ 
		such that 		
		$\Omega^1_X(m_1H_X)$
		is globally generated:
 		indeed, if 
 		$\mathcal X \to T$ 
 		is a projective family showing the boundedness of 
		$\mathcal S$, 
		we can assume that there exists a relatively ample Cartier divisor 
		$\mathcal H$
		on 
		$\mathcal X$
		such that for any 
		$t \in T$ with 
		$\mathcal X_t \simeq X \in \mathcal S$, 
		$\mathcal H\vert_{\mathcal X_t}$
		is linearly equivalent to $H_X$.
		By Serre's property
		\cite[Theorem II.5.17]{Hartshorne77}, 
		there exists 
		$m_1=m_1(\mathcal X)$ 
		such that 
		$\Omega^1_{\mathcal X/T}(m_1\mathcal H)$
		is relatively globally generated, which proves our claim.
		
		By the previous paragraph 
        and the fact that we have a generically surjective morphism 
        \[\Omega^{1}_X(m_1H_X) \to \mathcal O_X(K_{\mathcal F}+m_1H_X)\] 
       
	\noindent	then we obtain that 
		$H^0(\mathcal O_X(K_\mathcal{F}+m_1 H_X)) \neq 0$.
		At this point we can argue as in the proof of 
		\cite[Theorem 6.1]{SS23}:
		as $\mathcal S$ is bounded, there exists integers 
		$m_2, m_3$ and a positive integer
		$m_4$ such that 
		\begin{align*}
		m_2 \leq K_X \cdot H_X \leq m_3
		\quad 
		\text{and}
		\quad 
		0< \Delta \cdot H_X \leq m_4.
		\end{align*}
		In turn, that implies the existence of a positive integer $m_5$
		such that 
		$K_\mathcal{F} \cdot H_X \leq m_5$.
		Hence, there exists an effective Weil divisor
		$D \in 
		\vert K_\mathcal{F}+m_1 H_X)\vert$
		such that 
		$0< D \cdot H_X \leq m_5+m_1(mk)^2v$.
		Thus, the set of couples
		\begin{align*}
		\left \{
		(X, D) \
		\middle \vert \
		X \in \mathcal S, \ 
		D \in 
		\vert 
        K_\mathcal{F}+m_1 H_X
        \vert
		\right \}
		\end{align*}
		is bounded. 
		All the ingredients are now in place to directly apply
		\cite[Lemma 2.22]{SS23}
		to obtain the strongly log boundedness of  
		$\setfoliation{lc}$
		--
		let us note that while 
		\loccit\
		is stated for families of 	
		foliations, it holds equally well for families of integrable distributions).		
	\end{proof}

In turn, 
Corollary 
\ref{boundedness.lc.triples.cor.finite.coeff}
can be used to prove a boundedness statement that includes also slc foliated two-dimensional triples.

\begin{corollary}
\label{boundedness.slc.triples.cor}
Fix 
$v \in \mathbb Q_{>0}$, 
$N \in \mathbb Z_{>0}$, 
and 
$I \subset \mathbb Q \cap [0, 1]$
a finite set.
The set
$\setfoliation{slc}$
of triples 
$(X,  \Delta, \mathcal F)$ 
such that
\begin{enumerate}
\item 
$X$ is a deminormal projective surface; 
\item 
$\mathcal F$ is an integrable distribution on $X$;
\item 
$(X, \Delta)$ is semi log canonical;
\item 
$\mathcal F$ is semi log canonical;
\item 
$NK_{\mathcal F}$ is Cartier;
\item 
there exists $\eta \in \mathbb R_{>0}$ such that 
$K_\mathcal{F}+\epsilon K_X$ is ample for all $0< \epsilon < \eta$; and

\item 
\label{cond.vol.thm.slc.bound}
${\rm vol}(K_X+\Delta+5NK_{\mathcal F})=v$
\end{enumerate}
is strongly log bounded.
\end{corollary}

	\begin{proof}
		By Corollary~\ref{boundedness.lc.triples.cor.finite.coeff},
		it suffices to show that the collection of triples
		$(X,  \Delta, \mathcal F) \in 
		\mathfrak{F}^{\rm adj, slc}_{l, v} 
		\setminus
		\mathfrak{F}^{\rm adj, lc}_{l, v}$ is bounded. 
		We write
		$X=\cup_{i=1}^{s_X} X_i$
		as a union its irreducible components, and denote by
		$n \colon X^n \to X$
		the normalization of 
		$X$, 
		$X^n = \sqcup_{i=1}^{s_X} X_i^n$.
		
		Then,
		\begin{align}
			\label{eqn.vol.slc}
			v=
			\vol(K_X+\Delta+5NK_\mathcal{F})=
			\sum_{i=1}^{s_X}
			\vol(K_{X_i^n }+\Delta\vert_{X^n_i}+D_i+5NK_\mathcal{F}\vert_{X_i^n}), 
		\end{align}
		where 
		$D_i$
		denotes the reduced divisor on 
		$X_i^n$
		supported on the divisorial part of the  double locus.
		It follows immediately from the hypotheses of the corollary that 
		$(X_i^n, \Delta\vert_{X^n_i}+D_i, K_\mathcal{F}\vert_{X^n_i})$
		is a stable foliated triple. 
		We also consider 
		$(X_i^n, \Delta\vert_{X^n_i}+D_i, 5NK_\mathcal{F}\vert_{X^n_i})$ 
		as a generalized log canonical pair whose moduli part is the Cartier closure of the nef divisor
		$5NK_\mathcal{F}\vert_{X^n_i}$, cf. Proposition \ref{prop_threshold}.(1).
		
		We now divide the proof into steps, for the reader's convenience.
		
		\medskip
		
		{\bf Step 1}.
		{\it In this step we show that the collection
			\begin{align*}
			\left \{
			(\hat{X}, \hat{\Delta}, \hat{\mathcal{F}}) \
			\middle \vert \
			\substack{
			(\hat{X}, \hat{\Delta}, \hat{\mathcal{F}})
			\text{ is obtained as a component in the normalization}
			\\
			\text{of a triple }
			(X, \Delta, \mathcal F) \in \setfoliation{slc}}
			\right \}
		\end{align*}
		is strongly log bounded.}
		
		By 
		\cite[Theorem 1.3]{bir.2103}, 
		$\vol(K_{X_i^n}+\Delta\vert_{X^n_i}+D_i+5NK_\mathcal{F}\vert_{X_i^n})$
		belongs to a DCC set 
		$V=V(N) \subset \mathbb Q_{>0}$
		which is independent of 
		$i$ 
		and the choice of triple in 
		$\setfoliation{slc}$.
		In view of that,
		\eqref{eqn.vol.slc} 
		implies the existence of a positive integer 
		$C_1=C_1(N, v)$,
		and a finite set 
		$J= J(N, v) \subset \mathbb Q_{>0}$
		such that for all 
		$(X,  \mathcal F, \Delta) \in \setfoliation{slc}$, 
		$s_X \leq C_1$, 
		and
		for all 
		$i=1, \dots, s_X$,
		$K_{X_i^n}+\Delta\vert_{X_i^n}+D_i+5NK_\mathcal{F}\vert_{X_i^n}$
		is ample with
		$\vol(K_{X_i^n}+\Delta\vert_{X_i^n}+D_i+5NK_\mathcal{F}\vert_{X_i^n}) \in J$.
		Hence, by 
		Theorem~\ref{boundedness.glc.pairs.thm.finite.coeff}
		and 
		Corollary 
		\ref{boundedness.lc.triples.cor.finite.coeff}, 
		we can conclude that 
		for all 
		$(X,  \Delta, \mathcal F) \in \setfoliation{slc}$
		and all
		$i=1, \dots, s_X$, 
		$(X^n_i, \Delta\vert_{X^n_i}+D_i+5NK_{\mathcal F}\vert_{X^n_i})$
		belongs to the bounded collection
		\[
		\bigsqcup_{v' \in J}
		\mathfrak{F}_{v', N}^{\rm adj, lc}.
		\]
		Moreover, there exists a positive integer 
		$\bar k$
		such that
		$    
		\bar k (K_{X^n_i}+\Delta\vert_{X^n_i}+D_i+5NK_{\mathcal F}\vert_{X^n_i})
		$
		is Cartier and very ample on 
		$X^n_i$
		--
		cf. Steps 4-5 in the proof of Theorem~\ref{boundedness.glc.pairs.thm.finite.coeff} for the latter.
		Thus, for any 
		$\mathfrak{F}_{v', N}^{\rm adj, lc}$,
		$v' \in J$,
		we can fix a bounding family 
		$(\mathcal X_{v', N}, \mathcal R_{v', N}) \to T_{v', N}$
		together with a relatively very ample line bundle 
		$\mathcal H_{v', N}$
		whose existence is guaranteed by 
		the definition of boundedness applied to 
		$\mathfrak{F}_{v', N}^{\rm adj, lc}$;
		in particular, we may freely assume that
		$\mathcal H_{v', N}$
		restricts to a Cartier divisor linearly equivalent to 
		$\bar k (K_{Y}+\Gamma+5NK_{\mathcal G})$
		at all those fibers that are isomorphic to a triple 
		$(Y, \Gamma, \mathcal G) \in \mathfrak{F}_{v', N}^{\rm adj, lc}$.
		
		\medskip
		
		{\bf Step 2}.
		{\it In this step we show that there exists a positive integer 
		$\bar k'$ such that 
		for all 
		$(X, \Delta, \mathcal F) \in \setfoliation{slc}$, 
		$\bar k'(K_X+\Delta)$ 
		is Cartier. 
		In particular, we also have that 
		$\bar k'(K_X+\Delta+5NK_{\mathcal F})$ is Cartier, 
		by assumption.}
		
		As being Cartier is a local condition, it suffices to consider the possible singularities appearing on triples in 
		$\setfoliation{slc}$, 
		cf. \cite[Theorems 4.24]{KSB88}.
		First of all, as 
		$I \subset \mathbb Q$ 
		is finite, we can find a positive 
		$k_1$ 
		such that for all
		$i \in I$, 
		$k_1i \in \mathbb Z$.
		Given a point 
		$x \in X$, 
		we now show that the Cartier index of 
		$K_X+\Delta$ 
		is bounded by a positive integer 
		$\bar k'$, 
		divisible by 
		$k_1$, 
		and  independent of the choice of
		$(X,\Delta, \mathcal F) \in \mathfrak{F}^{\rm adj, slc}_{l, v}$ and of 
		$x \in X$. 
		By the argument in the previous paragraph, we can take $\bar k'=\bar k$ whenever $X$ is normal at $x$.
		Hence, we only need to analyze the different possible cases in the classification of strictly slc singularities according to the classification contained in \loccit.
		\begin{itemize}
			\item 
		If  $x \in X$ is a 
		$\mathbb Z_2$ 
		quotient of a degenerate cusp (case (iv) of \loccit), 
		it suffices to take 
		$\bar k'=2$, 
		since in that case
		$x \not \in \Delta$
		and a degenerate cusp is Gorenstein.
		To see that $x \not \in \Delta$, it suffices to note that the minimal semi-resolution $\pi \colon X' \to X$ extracts a cycle $E$ of rational curves and $K_{X'}+E=\pi^\ast K_X$.
		Thus, the strict transform of $\Delta$ cannot intersect $E$, as that would violate the semi log canonicity of $K_X+\Delta$.
		
			\item 
		If $x \in X$ is Gorenstein (case (ii) of \loccit), we only need to consider the cases contained in 
		\cite[Theorem 4.21(i))]{KSB88}, as the others have already been addressed.
		\\
		If $x \in X$ is a pinch point, then it is a Gorenstein singularity,  
		and
		$2k_1\Delta$ is Cartier, since the local Picard group at $x \in X$ is cyclic of order $2$, cf. \cite[Example 4.8]{MR3549172}.
		\\
		If $x \in X$ is a normal crossing point, then it is a Gorenstein singularity, and
		$k_1\Delta$ is Cartier, since the local class group at $x \in X$ is torsion-free, cf. \cite[Example 4.6]{MR3549172}.
	
			\item 
		Finally, we consider the case where $x \in X$ is a semi-log-terminal singularity, see \cite[Theorem 4.23(iii)-(v)]{KSB88}.
		In those cases, the germ of singularity is a finite quotient of a pinch point or of a normal corssing point. 
		The cardinality of the quotient is bounded from above by a positive integer $k_2$, since the normalizations of the components of $X$ are bounded, cf. Step 1: in particular, in a bounded family of klt surface singularities, the order of the possible quotients are bounded.
		Hence, $2k_2k_1(K_X+\Delta)$ is Cartier, by the previous case.
		\end{itemize}
		
To conclude this step, it suffices to define 
$\bar k' := 2k_2k_1$.

\medskip

{\bf Step 3}.
{\it In this step we conclude the proof, showing the strongly log boundedness of the desired collection.}

As for all 
$(X, \Delta, \mathcal F) \in \setfoliation{slc} \setminus \setfoliation{lc}$, 
$\bar k(K_X+\Delta +5NK_{\mathcal F})$
is Cartier, 
\cite[Corollary 1.4]{fujino-2016}
implies that there exists a positive integer
$m'=m'(N, I)$ 
such that 
$m'\bar k(K_X+\Delta +5NK_{\mathcal F})$
is very ample.
Since 
$\vol(m'\bar k(K_X+\Delta +5NK_{\mathcal F})=(m'\bar k)^2v$,  
then
\begin{align*}
	\mathcal S' :=
	\left \{
	X \
	\middle \vert \
	\exists (X, \Delta, \mathcal F) \in \setfoliation{slc}
	\right \}
\end{align*}
is bounded: in fact, 
\[
h^0(\mathcal O_X(m'\bar k(K_X+\Delta +5NK_{\mathcal F})))
\leq 
\sum_{i} 
h^0(\mathcal O_{X_i}(m'\bar k(K_{X_i}+\Delta_i +5NK_{\mathcal F_i})))
\]
and the summation is bounded from above in view of the boundedness of the triples 
$(X_i, \Delta_i, \mathcal F_i)$, cf. Step 1.
To conclude, we can then follow verbatim the last paragraph of the proof of Corollary \ref{boundedness.lc.triples.cor.finite.coeff}.
\end{proof}

We finally present the proof of Theorem~\ref{boundedness.glc.pairs.thm.finite.coeff}.

\begin{proof}[Proof of Theorem~\ref{boundedness.glc.pairs.thm.finite.coeff}]

We divide the proof into steps for the reader's convenience.

\medskip

{\bf Step 0}.
{\it In this step we make standard reductions}.

Note that condition~\eqref{cond.nef.thm.bound} implies that 
$M$ 
is nef and that 
$K_X+\Delta+sM$
is ample for
$s \gg 1$.
As 
$(X, \Delta)$ 
is log canonical, the Cone Theorem 
\cite[Theorem 5.6]{MR4298910} 
and condition~\eqref{cond.index.thm.bound} together imply that 
$K_X+\Delta+ 5NM$ 
is ample.
We denote 
$\widetilde{M}_X:=NM$, which we think of as a nef Cartier b-divisor via its Cartier closure.
Thus, we will indicate by 
$\widetilde{M}_{X'}$ the trace of the Cartier closure of 
$\widetilde{M}_X$
on any birational model 
$X'\dashrightarrow X$.

\medskip

{\bf Step 1.}
{\it In this step we show that the set of elements of 
$\setgpairsdcc{lc}$
is birationally bounded by a family satisfying some additional properties}.

Applying \cite[Proposition 5.2]{bir.2103},
as the volume of 
$K_X+\Delta+\widetilde{M}_X$ 
is fixed, 
there exists a bounded set of couples 
$\mathcal{P}$ 
such that for each 
$(X, \Delta, \widetilde{M}_X) \in \setgpairsdcc{lc}$ 
we can find a 
log smooth log pair 
$(X',\Sigma')\in \mathcal{P}$,
such that  
$\Sigma'$ is reduced,
and there exists a birational map 
$\psi \colon X' \dashrightarrow X$ 
satisfying the following properties:
\begin{itemize}
\item 
$\Sigma' \ge \Delta'$, 
where 
$\Delta' := {\rm Exc}(\psi) + \psi_\ast^{-1} \Delta$;

\item  
the Cartier closure of 
$\widetilde{M}_X$  
descends 
on
$X'$; and, 
\item
$\vol(K_{X'}+\Delta'+5\widetilde{M}_{X'})=v$.
\end{itemize}

\medskip

{\bf Step 2}.
{\it In this step we show that we can assume that 
$\psi$
is a morphism.}

As 
$K_{X}+ \Delta + \widetilde{M}_X$ 
is ample by hypothesis,
taking a resolution of indeterminacies
\begin{align*}
\xymatrix{
& X'' \ar[dr]^{p} \ar[dl]_{q} & 
\\
X' \ar@{-->}[rr]^{\psi} & & X,}
\end{align*}
it follows that 
\begin{align}
\label{ineq1}
p^\ast (K_{X}+ \Delta+ 5\widetilde{M}_X)
\leq 
q^\ast(K_{X'}+\Delta'+5\widetilde{M}_{X'})
\end{align}
Indeed, 
\begin{align*}
    p^\ast (K_{X}+ \Delta+ 5\widetilde{M}_X) & = 
    K_{X''}+ p_\ast^{-1}\Delta + p^\ast 5\widetilde{M}_X + {\rm Exc}(p)-G_1 
    & [G_1 \geq 0, \text{$p$-exceptional}]
    \\
    & = q^\ast(K_{X'}+\Delta'+5\widetilde{M}_{X'}-q_\ast G_1) - G_2  
    & [G_2 \geq 0, \text{$q$-exceptional}]
\end{align*} 
where the first equality is simply the log pullback formula, whereas the second follows by the negativity lemma, \cite[Lemma 3.39]{KM98},
as 
$q_\ast {\rm Exc}(p) = {\rm Exc}(\psi)$, 
$\Delta'=\psi_\ast^{-1} \Delta +{\rm Exc}(\psi)$, 
and
\begin{align*}
q_\ast (K_{X''}+ p_\ast^{-1}\Delta + p^\ast 5\widetilde{M}_X + {\rm Exc}(p)-G_1)= 
K_{X'}+\Delta'+5\widetilde{M}_{X'}-q_\ast G_1
\end{align*}
together with the fact that 
$K_{X}+ \Delta+ 5\widetilde{M}_{X}$ 
is ample.
By Step 1 we know that 
\begin{align*}
\vol(K_{X'}+\Delta'+5\widetilde{M}_{X'}) = \vol(K_{X}+\Delta+5\widetilde{M}_{X})=v,
\end{align*}
and we have also that  
$a(E, X, \Delta+5\widetilde{M}_X) \ge a(E, X', \Delta'+5\widetilde{M}_{X'})$
for any divisor 
$E$ 
over 
$X$, 
see (\ref{ineq1}).
Hence, we may apply
\cite[Lemma 2.17]{bir.2103}
to show that
$\psi^{-1} \colon 
X \dashrightarrow X'$ 
does not contract any divisor, which in turn implies that 
$(X, \Delta, 5\widetilde{M}_X)$
is the generalised log canonical model of 
$(X', \Delta', 5\widetilde{M}_{X'})$, 
and, thus, that
\begin{align}
\label{eqn.section.rings}
H^0(X, m(K_X+\Delta+5\widetilde{M}_X))=
H^0(X', m(K_{X'}+\Delta'+5\widetilde{M}_{X'})), 
\quad \forall m \in \mathbb Z_{>0}.
\end{align}

\medskip

{\bf Step 3}.
{\it In this step we show that there exists a factorization
\begin{align*}
    \xymatrix{
   X'
    \ar[r]_{s'}
    \ar@/^1.5pc/[rr]^{\psi}
    &
   \overline X'
    \ar[r]_{\phi'}
    &
    X
    }
\end{align*}
and a generalized dlt log pair
$(\overline X', \overline{\Delta}', 5\widetilde{M}_{\overline X'})$, 
where
$\overline \Delta:=s'_\ast \overline \Delta$, 
$\widetilde{M}_{\overline X}:=s'_\ast \widetilde{M}_{\overline X}$
such that 
$\phi'$
is a generalized dlt modification of the generalized log pair 
$(X, \Delta, M)$.}

Running the 
$(K_{X'}+\Delta'+5\widetilde{M}_{X'})$-MMP
starting on 
$X'$ will terminate with a $(K_{X'}+\Delta'+5\widetilde{M}_{X'})$-minimal model, 
i.e., there exists a finite sequence of birational morphisms
\begin{align*}
 \xymatrix{
X' =: X_0 \ar[r]
\ar @/^1.5pc/ [rrrr]^s
&
X_1 \ar[r]
&
\dots \ar[r]
&
X_{n-1} \ar[r]
&
X_n =: \overline X
}   
\end{align*}
that are 
$(K_{X'}+\Delta'+5\widetilde{M}_{X'})$-negative;
moreover, the strict transform 
$K_{\overline X}+ \overline \Delta+ 5\widetilde{M}_{\overline X}$
of 
$K_{X'}+\Delta'+ 5\widetilde{M}_{X'}$
is big and nef.
As each step of the above run of the MMP
is 
$(K_{X'}+\Delta'+ 5\widetilde{M}_{X'})$-negative 
and
$\widetilde{M}_{X'}$
is nef Cartier, then
$\widetilde{M}_{X'}$ descends along each step of the above MMP, i.e., 
\begin{align}
\label{eqn:descent.M}
\widetilde{M}_{X'} = s^\ast \widetilde{M}_{\overline X}, 
\qquad 
\text{where }
\widetilde{M}_{\overline X} := s_\ast \widetilde{M}_{X'}
\end{align}
and 
$\widetilde{M}_{\overline X}$
is nef and Cartier on 
$\overline{X}$.
Moreover, for all non-negative integers 
$m, t$,
\begin{align}
\label{eqn.section.rings2}
H^0(X, 
m(K_X+\Delta+M))=
H^0(\overline X, 
m(K_{\overline X}+ \overline \Delta +\widetilde{M}_{\overline X})).
\end{align}
The Iitaka fibration of 
$K_{\overline X}+ \overline \Delta +s\widetilde{M}_{\overline X}$
is given by a morphism
\begin{align*}
\phi \colon \overline X \to X=
{\rm Proj} 
\left (
\bigoplus_{m \geq 0} 
H^0 \left (
\overline X, 
\mathcal O_{\overline X}
\left (
n(K_{\overline X}+ \overline D+ 5\widetilde{M}_{\overline X})
\right )
\right)
\right ),
\end{align*} 
by 
\eqref{eqn.section.rings} and 
\eqref{eqn.section.rings2}, and 
since $K_{X}+ D+5 \widetilde{M}_X$ is ample.
Finally, as 
$(\overline X,  \overline D+ 5\widetilde{M}_{\overline X})$
is the outcome of a run of the 
$(K_{X'}+D'+5\widetilde{M}_{X'})$-MMP, 
cf. \eqref{eqn.section.rings}, 
then 
$\phi$
is a generalized dlt modification of the generalized log pair 
$(X, D+ 5\widetilde{M}_X)$, and 
\begin{align}
\label{eqn.dlt.mod.pullback.formula}
K_{\overline X}+ \overline D+ 5\widetilde{M}_{\overline X}
=
\phi^\ast
(K_X+D+5\widetilde{M}_X), 
\quad 
\widetilde{M}_{\overline X}=\phi^\ast \widetilde{M}_X.
\end{align}
We can write 
$\overline D=\phi_\ast^{-1}D + G$ 
for an effective divisor 
$G$
which we decompose as
$G=\lfloor G \rfloor + G_1$.
As for any prime component 
$E_1$
of 
$G_1$, 
$E_1^2 < 0$,
then  
$(K_{\overline X}+ \overline D+ 5\widetilde{M}_{\overline X} + t E_1) \cdot E_1 < 0 $
for all
$t>0$
As 
$(\overline X,  \overline \Delta+ 5\widetilde{M}_{\overline X} + \epsilon_1 E_1)$
is dlt for all 
$0< \epsilon_1 \ll 1$, 
then there exists a birational morphism 
$\tau_1 \colon \overline X \to \overline X_1$
such that 
$\mathrm{exc}(\tau_1)=E_1$ 
and 
$(\overline X_1,  \overline \Delta_1+ 5\widetilde{M}_{\overline X_1})$ is dlt, 
where 
$\overline \Delta_1 := (\tau_1)_\ast \overline \Delta$, 
$\widetilde{M}_{\overline X_1} := (\tau_1)_\ast \widetilde{M}_{\overline X}$.
Arguing by induction on the number of components of $G_1$, 
we can then construct a factorization 
\begin{align}
\label{diag.gdlt.mod}
    \xymatrix{
 \overline X
    \ar[r]_{\tau}
    \ar@/^1.5pc/[rr]^{\phi}
    &
   \overline X'
    \ar[r]_{\phi'}
    &
    X
    }
\end{align}
contracting 
$G_1$,
such that 
$(\overline X', \overline \Delta' + 5\widetilde{M}_{\overline X'})$
is generalized dlt, where 
$\overline \Delta' := \tau_\ast \overline D$, 
$\widetilde{M}_{\overline X'} := \tau_\ast \widetilde{M}_{\overline X'}$, 
and 
\begin{align}
\label{eqn.dlt.mod.pullback.formula2}
 K_{\overline X'}+ \overline \Delta'+ 5\widetilde{M}_{\overline X'}
=
(\phi')^\ast
(K_X+\Delta+5\widetilde{M}_X), 
\quad 
\widetilde{M}_{\overline X'}=(\phi')^\ast \widetilde{M}_X.   
\end{align}

\medskip

{\bf Step 4}.
{\it 
In this step we show that in order to show the boundedness of the log pairs contained in 
$\setgpairsdcc{lc}$
it suffices to show that there exist 
a positive integer 
$m$
such that 
$m(K_{\overline X}+\overline \Delta)$
is Cartier.}

If
$m$ 
exists and 
$m(K_{\overline X}+\overline \Delta)$ 
is Cartier, then 
also 
$m(K_{\overline X}+\overline \Delta+5\widetilde{M}_{\overline X})$ 
is Cartier by the construction performed in the previous step; 
moreover, 
\eqref{eqn.dlt.mod.pullback.formula}
together with 
\cite[\S 4.17, Lemma]{reid.chapter}
implies that also 
$m(K_{\overline X'}+\overline \Delta')$
and
$m(K_{\overline X'}+\overline \Delta'+5\widetilde{M}_{\overline X'})$ 
are Cartier, 
since 
$\overline X'$
has rational singularities.
In turn, that and the fact that the Cartier index of two-dimensional strictly log canonical singularities is $\leq 6$, 
cf. 
\cite{shok.complements},
together imply that 
$m(K_X+\Delta)$
and 
$m(K_X+\Delta+\widetilde{M}_X)$ are Cartier.

The claimed conclusion then follows from 
\cite[Theorem~1.1]{fujino-2016},
which ensures the existence of a positive integer 
$m_1=m_1(m)$
such that 
$\vert m_1m(K_{X}+\Delta+5\widetilde{M}_X)\vert$
is very ample, together with the fact that 
\begin{align*}
{\rm vol}(2m_1m(K_{X}+\Delta+5\widetilde{M}_X))= 
(m_1m)^2v
\end{align*}
is then independent of the choice of pair in
$\setgpairsdcc{lc}$.

\medskip

In order to conclude the existence of the positive integer 
$m$ 
and the finite set 
$I_0$
satisfying the property stated in Step 4, we will show that 
$(\overline X, \overline \Delta)$
is strongly log bounded.

\medskip

{\bf Step 5. }{\it In this step we show the existence of the positive integer $m$ satisfying the properties stated in Step 4.}

Let  
$f \colon (\mathcal X, \mathcal S) \to T$
be a family of reduced log pairs such that any log pair in
$\mathcal P$, see Step 1, appears as a fibre of $f$.
As the bounded collection pairs $\mathcal P$ is made of snc pairs, up to passing to a locally closed (finite) stratification in the Zariski topology of $T$, 
we can assume that 
$(\mathcal X, \mathcal S)$
is snc over 
$T$.
Moreover, up to substituting 
$T$ with the Zariski closure of the set
\begin{align*}
T_{\mathcal P} :=
    \left \{
t \in T \ 
\middle \vert \
(\mathcal X_t, \mathcal S_t) \in \mathcal P
    \right \}, 
\end{align*}
we can assume that 
$T_{\mathcal P}$ 
is Zariski dense in 
$T$.
Moreover, up to passing to a further finite locally closed stratification of 
$T$ 
in the Zariski topology, we can assume that each irreducible component of 
$T$
is also a connected component.
Let us observe that the last operation does not affect the denseness in each connected of $T$ of points supporting pairs in $\mathcal P$.
Finally, up to passing to a finite stratification of $T$ into locally closed subsets, we can assume that each irreducible (and connected) component of 
$T$
is smooth.
Again, by Noetherian induction, it is easy to see that such operation can be done ensuring the Zariski denseness of 
$T_{\mathcal P}$ 
in 
$T$.
Moreover, the final outcome 
$(\mathcal X, \mathcal S) \to T$ 
is still relatively snc.

Let us fix an irreducible (equivalently,  connected) component
$W_h$
of 
$T$.
We denote by 
$(\mathcal X^h, \mathcal S^h)$
the pullback of 
$(\mathcal X, \mathcal S)$
to 
$W_h$.
By our construction, for all 
$w \in W_h$
the number 
$k_{h}$ 
of irreducible components of 
$\mathcal S^h_{w}$
on 
$\mathcal X^h_{w}$
is independent of 
$w$;
moreover, also the intersection matrix 
$(m_{ij}:=
\mathcal S^1_{w, i} \cdot \mathcal S^1_{w,  j})_
{1\leq i, j \leq k_{h}}$
of the irreducible components of 
$\mathcal S^h_w$ 
is independent of 
$w \in W_h$.

Now for any $\overline{X}$ there exists $w \in W_h$ such that we have a contraction $\mathcal X^h_w \to \overline{X}$ whose exceptional locus is contained 
in $\mathcal S^h_w$.
Since $\mathcal X^h_w$ is smooth and $\overline{X}$ has rational singularities (since it is dlt) we see that the Cartier index of any Weil divisor on $\overline{X}$
divides the determinant of the intersection matrix of the exceptional locus of $\mathcal X^h_w \to \overline{X}$, cf. \cite[Theorem 4.6]{MR1805816}.  As noted, the intersection matrix of the exceptional locus takes on only finitely many possible values, and so the determinant of the intersection matrix takes on only finitely many values, call them $m_1, \dots, m_{k_h}$.  We can conclude this step by taking $m$ to be the least common multiple of $m_1, \dots, m_{k_h}$.

\medskip

From Step 5, we can conclude that there exists a positive integer 
$m$ such that 
$m(K_X+D+M)$ 
is Cartier and very ample
for all 
$(X+D+M) \in \setgpairsdcc{lc}$.
Moreover, 
$m$ 
is independent of the triple chosen in 
$\setgpairsdcc{lc}$ and instead only depends 
on the set of coefficients $I$.

\medskip

{\bf Step 6}.
{\it In this step we show that also 
$\widetilde{M}_X$ is bounded.}

Since $(X, D)$ is strongly log bounded, 
then there exists 
a positive integer 
$\theta$, 
independent of the chosen pair in 
$\setgpairsdcc{lc}$, 
such that 
$K_X+D+5\widetilde{M}_X
\pm
\frac 1 \theta (K_X+D)
$
is ample.
But then 
\begin{align*}
\left [ 
\left (
1\pm \frac 1 \theta 
\right ) 
\left (
K_X+D+5M 
\right ) 
\right ]^2 
& = 
\left [ 
\left ( 
\left (
1\pm \frac 1 \theta 
\right ) 
\left (
K_X+D 
\right )
+5M 
\right )
\pm \frac{5}{\theta} M 
\right ]^2
\\
& =
 \left ( 
 K_X+D +5M
 \pm
 \left (
\frac 1 \theta
 \right )
 \left (
K_X+D
 \right)
 \right )^2 \\
& \pm
\frac{10}{\theta}  
\left ( 
\left (
1 \pm \frac 1 \theta 
\right ) 
\left (
K_X+D 
\right )
+5M 
\right )
\cdot 
M
\\
& +  
\left (
\frac{5}{\theta} 
\right )^2 
M^2.
\end{align*}
On one hand, 
\begin{align*}
    \left [ 
    \left (
    1\pm \frac 1 \theta 
    \right ) 
    \left (
K_X+D+5lM 
\right ) 
\right ]^2
    = 
    \left (
    1\pm \frac 1 \theta 
    \right ) ^2 v;
\end{align*}
on the other hand, 
both 
$ \left ( 
 K_X+D +5lM
 \pm \left (
\frac 1 \theta
 \right )
 \left (
K_X+D
 \right)
 \right )^2$
 and 
 $M^2$
 are non-negative.
 Hence, 
 \begin{align}
 \label{ineq.final.bound.KF}
   0 
   & <
   \left (
   K_X+D 
   +5lM 
   \right )
   \cdot 
   M
   \leq 
   v \mp \frac 1 \theta (K_X+D)\cdot M.
 \end{align}
 By the strongly log boundedness of 
 $(X, D)$, 
 $(K_X+D)\cdot M$
 can only achieve finitely many possible values, when 
 $(X, D, M)$
 varies in 
 $\setgenpair{lc}$.
 Hence, there exists a positive integer 
 $C_4$ such that 
 for any 
 $(X, D, M)
 \in 
 \setgenpair{lc}$, 
 $\vert(K_X+D)\cdot M \vert \leq C_4$.
 Hence, that observation and \eqref{ineq.final.bound.KF} 
 together imply that 
 \begin{align*}
     0 
   & <
   \left (
   K_X+D 
   +5lM 
   \right )
   \cdot 
   M
   \leq 
   v +C_4,
 \end{align*}
 as desired.
\end{proof}

\section{Deformations of log canonical singularities}

\subsection{Two conjectures on singularities of foliations}

We present two closely related conjectures on the behaviour of log canonical
foliation singularities.  Studying these conjectures is a critical part of our 
study of moduli, but we expect these conjectures to be of broad interest.

\begin{conjecture}[Semi-log canonical singularities deform] \label{conjecture sing}
        Let $f\colon (X, \mathcal F) \rightarrow T$ be a flat
    family of integrable distributions such that $K_{\mathcal F}$ is $\mathbb Q$-Cartier and the fibres of $f$ are  deminormal.

 Suppose that $\mathcal F_t$
    is semi-log canonical for some closed point $t \in T$.  Then there exists a Zariski open subset $t \in U \subset T$ such that for all $s \in U$,
    $\mathcal F_{s}$ is semi-log canonical.    
\end{conjecture}

We remark that if we replace ``semi-log canonical'' in the above conjecture by ``canonical'' the conjecture is false as can be seen in Example \ref{easiest_example}.

This conjecture is closely related  to the next conjecture

\begin{conjecture}[Inversion of adjunction]
\label{conj_inversion_adjunction}
Let $X$ be a normal variety, let $\mathcal F$ be a foliation on $X$, let $S$ be a prime divisor on $X$ and let $B \ge 0$
be a $\mathbb Q$-divisor such that $m_SB = \varepsilon(S)$
and such that $K_{\mathcal F}+B$ is $\mathbb Q$-Cartier.
Let $n\colon T \rightarrow S$ be the normalisation and write (via foliation adjunction, cf. \cite[Proposition-Definition 3.7]{CS23b}) \begin{align*}n^*(K_{\mathcal F}+B) = K_{\mathcal F_T}+B_T.\end{align*} 

If $(\mathcal F_T, B_T)$ is log canonical, then 
$(\mathcal F, B)$ is log canonical in a neighbourhood of $S$.
\end{conjecture}

In the case of varieties both of these conjectures are known, see for instance, \cite{Kawakita05}, and in fact 
Conjecture \ref{conj_inversion_adjunction} easily implies Conjecture \ref{conjecture sing}.
For foliations, however, the relation between these two conjectures is far more subtle owing to the failure of adjunction on foliation singularities when $\varepsilon(S) = 0$, see \cite[Example 3.20]{CS23b}.

\subsection{Adjunction results for rank one foliations}
Consider a vector field $\partial$ on a scheme $X$ and a closed point $x \in \sing \partial$.
In this case, $\partial(\mathfrak m)\subset \mathfrak m$ where $\mathfrak m$ is the ideal of $x$ and we get an induced linear map
\begin{align*}\partial_1\colon \mathfrak m/\mathfrak m^2
\to \mathfrak m/\mathfrak m^2.\end{align*}
We call $\partial_1$ the linear part of $\partial$.

We recall the following results regarding log canonical singularities of rank one integrable distributions (although the results in \cite{MP13} are stated for rank one foliations we note that they hold equally well for rank one integrable distributions).

\begin{proposition}
\label{prop_lc_nonnilp}
    Let $X$ be a normal variety and let $\mathcal F$
    be an integrable distribution defined by a vector field 
    $\partial$ (equivalently $\mathcal F$ is rank one and $K_{\mathcal F}$ is Cartier).

    Let $x \in \sing \partial$ be a closed point.
    Then $\mathcal F$ is log canonical at $x$ if and only 
    if $\partial_1$ is non-nilpotent
\end{proposition}
\begin{proof}
This is \cite[Fact I.ii.4]{MP13}
\end{proof}

\begin{corollary}
\label{cor_lc_is_open}
    Let $X$ be a normal variety and let $\mathcal F$
    be a rank one integrable distribution such that $K_{\mathcal F}$
    is Cartier.  Then the set 
    $\{ x \in X: \mathcal F \text{ is log canonical at } x\} \subset X$ is Zariski open.
\end{corollary}
\begin{proof}
    This is \cite[Definition-Summary I.ii.6]{MP13}
\end{proof}

\begin{lemma}
    \label{lem_nonnilp_normal}
    Let $X = {\rm Spec}(R)$ be a reduced affine scheme, let $n\colon Y = {\rm Spec}(S) \rightarrow X$
    be the normalisation and let
    $\partial$
    be a vector field on $X$.
    Then the following hold:
    
    \begin{enumerate}
\item[(0)] there is a lift of $\partial$
to a vector field $\tilde{\partial}$
on $Y$,
    
\item[(1)] if $I \subset R$ is a $\partial$-invariant ideal, then 
$IS$ is a $\tilde{\partial}$-invariant ideal;
    
    \item[(2)] if $x \in X$ is a closed point such that $x \in \sing \partial$,
    then $n^{-1}(x) \subset \sing \tilde{\partial}$; and 

    \item[(3)] if $x \in \sing \partial$ and for any $y \in n^{-1}(x)$, if the linear part of $\tilde{\partial}$
    \begin{align*}\tilde{\partial}_1\colon \mathfrak n/\mathfrak n^2 \to \mathfrak n/\mathfrak n^2\end{align*}
    is non-nilpotent, where $\mathfrak n$ is the maximal ideal of $y$, 
    then the linear part of $\partial$
    \begin{align*}\partial_1\colon \mathfrak m/\mathfrak m^2 \to \mathfrak m/\mathfrak m^2\end{align*}
    is non-nilpotent.

    \end{enumerate}
\end{lemma}
\begin{proof}
All the claims are local about any closed point of $x \in X$, so without loss of generality we may replace $R$ by its completion along $\mathfrak m$, and so we may assume that $R$ is a complete local ring.

It suffices to check the  claims of the lemma on each irreducible component of $X$, so by replacing $X$ by one of its irreducible components we may assume that $X$ is an integral scheme. 

We first remark that (0) is a result of Seidenberg, see \cite{MR188247}.  Note that 
if we consider $\tilde{\partial}$ as a derivation on $S$, then the restriction of $\tilde{\partial}$
to $R \subset S$
is precisely the derivation given by $\partial$.

We now prove (1). By definition, $IS$ is generated by elements of the form 
$fs$ where $f \in I$ and $s \in S$.  Note that $\tilde{\partial}(f) = \partial(f) \in I$, thus 
\begin{align*}\tilde{\partial}(fs) = \tilde{\partial}(f)s +f\tilde{\partial}(s) \in IS,\end{align*}
as required.

 We now prove (2).  By (1), $\mathfrak mS$ is $\tilde{\partial}$-invariant, and therefore its radical is invariant as well
 (cf. \cite[Lemma 1.8]{MR460303}), i.e.,   $\tilde{\partial}(\sqrt{\mathfrak mS}) \subset \sqrt{\mathfrak mS}$.  Since $R$ is a domain, we see
    that $\mathfrak n \coloneqq \sqrt{\mathfrak mS}$
    is a maximal ideal, cf. \cite[\href{https://stacks.math.columbia.edu/tag/0C37}{Tag 0C37}]{stacks-project}.   
    For any $f \in S$, if we let $c = f \mod \mathfrak n \in k$ then we see that \begin{align*}\tilde{\partial}(f) = \tilde{\partial}(f-c) \in \tilde{\partial}(\mathfrak n) \subset \mathfrak n, \end{align*}
    i.e., $\tilde{\partial}(S) \subset \mathfrak n$, as required.

We finally prove (3).  Supposing that $\tilde{\partial}_1$ is non-nilpotent we may find 
$f \in \mathfrak n$ such that $\tilde{\partial}_1(f_1) = \lambda f_1$
where $0 \neq f_1 = f \mod \mathfrak n^2$ and $\lambda \neq 0$.
To show that the linear map
\begin{align*}\partial_1\colon \mathfrak m/\mathfrak m^2 \to \mathfrak m/\mathfrak m^2 \end{align*}
is non-nilpotent, it suffices to show that 
 the linear map
 \begin{align*}L\colon (\mathfrak mS)/\mathfrak n(\mathfrak mS)\to (\mathfrak mS)/\mathfrak n(\mathfrak mS)\end{align*}
 induced by $\partial$ is non-nilpotent.
Since $\mathfrak n = \sqrt{\mathfrak mS}$, there exists a positive integer $l$ such that 
$f^l \in \mathfrak mS$, but $f^l \notin \mathfrak n(\mathfrak mS)$.  The non-nilpotence of $L$ then follows by observing that $\partial(f^l) = l\lambda f^l \mod \mathfrak n^l$ and so $l\lambda \neq 0$
is an eigenvalue of $L$ as required.
\end{proof}

\begin{proposition}
\label{prop_locally_stable_adjunction}
Consider the following set up.
\begin{enumerate}
    \item Let $X$ be an $S_2$ variety with $\dim X = 3$. 
    \item Let $f\colon X \to T$ be a flat surjective morphism to a smooth curve.
    \item Let $\mathcal F$ be a rank one integrable distribution on $X$ such that
    \begin{enumerate} 
    \item $T_{\mathcal F} \subset T_{X/T}$ in a neighbourhood of every generic point of $X$; 
    \item $\sing \mathcal F$ does not 
contain any components of fibres of $f$; and
\item $K_{\mathcal F}$
is $\mathbb Q$-Cartier and $\mathcal F$ has semi-log canonical singularities.
\end{enumerate}
\end{enumerate}

Let $0 \in T$ be a closed point, let $n\colon S \rightarrow X_0 = f^{-1}(0)$ be the normalisation of an irreducible component and let $(\mathcal G, \Theta)$ be the foliated pair on $S$ associated to the restricted integrable distribution, cf. \cite[Proposition-Definition 3.12]{CS23b}.  
Then $(\mathcal G, \Theta)$ is log canonical.    
\end{proposition}
\begin{proof}
Without loss of generality we may freely replace $X$ by its normalisation and so we may assume that $X$ is normal and that $\mathcal F$ is log canonical.

Let $F$ be a prime divisor on some birational model $S'$ of $S$.  Since $F$ is arbitrary, in order to prove the Proposition it suffices to show that 
$a(F, \mathcal G, \Theta) \ge -\varepsilon(F)$.  

We may find a sufficiently high log resolution of $(X, X_0)$, call it  $p\colon Y \rightarrow X$, so that we have a morphism 
$\tilde{S} \rightarrow S'$ where $\tilde{S}$
is the strict transform of $n(S)$.
Let $E_1, \dots, E_r$ be the $p$-exceptional divisors.
Perhaps replacing $Y$ by a higher model and relabelling the $E_i$ we may assume that 
$F = E_1 \vert_{\tilde{S}}$ and that $E_1$ is contained in the 
central fibre of $Y \rightarrow T$.

By \cite[III.iii.4]{MP13}
 up to replacing $Y$ by a higher model 
we may assume that $\mathcal F_Y$ has log canonical 
singularities in a neighbourhood of the generic point of $F$.
By foliation adjunction, cf. \cite[Proposition-Definition 3.12]{CS23b}, we may write $K_{\mathcal F_Y}\vert_{\tilde{S}} = 
K_{\mathcal G_{\tilde{S}}}+\Gamma$ where $\Gamma \ge 0$ and 
$f_*\Gamma = \Theta$, where $f\colon \tilde{S} \rightarrow S$
is the natural morphism.
We then have an equality \begin{align*}K_{\mathcal G_{\tilde{S}}}+\Gamma+\sum_{i = 1}^r a(E_i, \mathcal F)E_i\vert_{\tilde{S}} = f^*(K_{\mathcal G}+\Theta)\end{align*}
It follows that $a(F, \mathcal G, \Theta) = -m_F\Gamma+a(E_1, \mathcal F)$. 

Thus, to show that $a(F, \mathcal G, \Theta) \ge -\varepsilon(F)$ it suffices to show that 
(i) $m_F\Gamma \leq \varepsilon(F)$ and (ii) $a(E_1, \mathcal F) \ge 0$.  
 Item   (i) follows by noting $\mathcal F_Y$ has log canonical singularities in a 
neighbourhood of the generic point of $F$ and so
Lemma \ref{lem_sing_adjunction_families} 
guarantees that $m_F\Gamma \leq \varepsilon(F)$.
Item (ii) follows by first noting that since $E_1$ is contained in the central fibre, 
$E_1$ is $\mathcal F_Y$-invariant.  Since $\mathcal F$ has log canonical singularities we have
$a(E_1, \mathcal F) \ge \varepsilon(\mathcal F_Y, E_1) = 0$.
\end{proof}

\begin{remark}
    As mentioned, adjunction on singularities for invariant divisors 
is false in general, and so in Proposition \ref{prop_locally_stable_adjunction} it is crucial that our foliation is tangent to a fibration.
\end{remark}

\begin{lemma}
	\label{lem_sing_adjunction_families}

Consider the following set up.
\begin{enumerate}
    \item Let $X$ be a smooth variety with $\dim X = 3$. 
    \item Let $f\colon X \to T$ be a surjective morphism to a smooth curve with simple normal crossings fibres.
    \item Let $\mathcal F$ be a rank one integrable distribution on $X$ such that
    \begin{enumerate} 
    \item $T_{\mathcal F} \subset T_{X/T}$; 
    \item $\sing \mathcal F$ does not 
contain any components of fibres of $f$; and
\item  $\mathcal F$ has log canonical singularities.
\end{enumerate}
\end{enumerate}

	Let $0 \in T$ be a closed point
	and let $S_0 \subset X_0 = f^{-1}(0)$ be an irreducible component. 
	Let $(\mathcal F_0, D_0)$ be the restricted foliated pair on $S_0$. 
Then for any component $C$ of $D_0$ we have $m_CD_0 \leq \varepsilon(\mathcal F_0, C)$.
\end{lemma}
\begin{proof}
	Let $C$ be an irreducible component of $D_0$.  The question is local about a general point of  $C$, so we may freely choose a system of coordinates $(t, x, y)$ on $X$ so that
	 $C = \{x = t = 0\}$, $S_0 = \{t = 0\}$ and 
     \begin{itemize}
         \item in the case that $X_0$ is normal, $f(t, x, y) = t^s$; and 
         \item  in the case that $X_0$ is not normal in a neighbourhood of $C$, $f(t, x, y) = t^sx^r$, so $(X_0)_{{\rm red}} = \{xt = 0\}$.
     \end{itemize}
 
	We may assume that $\mathcal F$ is defined by a vector field  $\partial = a\partial_x+b\partial_t+c\partial_y$
    where $a, b, c \in \mathcal O_X$.  Since $S_0$ is invariant by $\partial$ let us note that $t|b$.

Let us suppose for sake of contradiction that $m_CD_0 = k >\varepsilon(\mathcal F_0, C)$.
We remark that any exceptional divisor centred over $C$ is mapped to $0 \in T$, and is therefore invariant by the transform of $\mathcal F$.  Thus to derive a contradiction it suffices to find an exceptional divisor $E$ dominating $C$ with discrepancy $a(E, \mathcal F) \leq -1 < 0 = \varepsilon(E)$.

Note that $m_CD_0$ is precisely the order of vanishing of the restricted vector field $\partial|_{\{t = 0\}}$ along $C$, in particular, $C \subset \sing \partial$.  Thus,
$m_CD_0 = k$
implies that $a, b, c \in (x^k, t)$.
So, let us write $a = A_1x^k+A_2t, b = B_1x^k+B_2t$
and $c = C_1x^k+C_2t$ where $A_i, B_i, C_i \in \mathcal O_X$.
As noted earlier, $t|b$, so in fact we may assume that $B_1 = 0$.

\begin{claim}
\label{claim_order2}
$A_1x^k \in (x, t)^2$.
\end{claim}
\begin{proof}[Proof of Claim]
The restriction of $\partial$ to $S_0$ is 
$x^k\alpha_1\partial_x+x^k\gamma_1\partial_y$ where $\alpha_1$ (resp. $\gamma_1$) is the restriction of $A_1$ (resp. $C_1$) to $S_0$. This vector field has a codimension $=1$ component of its zero locus.  After dividing by an appropriate function we may find a generator of $\mathcal F_0$.  Indeed, we may write $x^k\alpha_1 = x^kq\tilde{\alpha}_1$ and 
$x^k\gamma_1 = x^kq\tilde{\gamma}_1$
where $\tilde{\alpha}_1$ and $\tilde{\gamma}_1$ have no common factors.  In particular, $\tilde{\alpha}_1\partial_x+\tilde{\gamma}_1\partial_y$
is a vector field on $S_0$ without codimension one zeros which defines $\mathcal F_0$.

We argue in cases based on $\varepsilon(\mathcal F_0, C)$.

If $\varepsilon(\mathcal F_0, C) = 1$, then $k \ge 2$, and so $A_1x^k \in (x, t)^2$.  

If $\varepsilon(\mathcal F_0, C) = 0$ then $k \ge 1$.   Since $\{x = 0\}$ is $\mathcal F_0$-invariant, it is invariant by the vector field 
$\tilde{\alpha}_1\partial_x+\tilde{\gamma}_1\partial_y$.  It follows that $x|\tilde{\alpha}_1$, hence $x|\alpha_1$, which implies that $A_1 = xA'_1+tA''_1$ where $A'_1, A''_1 \in \mathcal O_X$, and so $A_1x^k \in (x, t)^2$.
\end{proof}

We next make the following observation.

\begin{claim}
\label{claim_easy_discrep_calc}
    If
$a, b\in (x, t)^2$, 
then the blow up in $C$ will extract a divisor of discrepancy $\leq -1$.
\end{claim}
\begin{proof}[Proof of Claim]
Consider the coordinate chart for the blow up $p\colon X' \rightarrow X$ in $C$ 
given by $(u, v, y) \mapsto (uv, u, y) = (t, x, y)$.
In these coordinates we can verify that 
$p^*\partial_x = \partial_u - \frac{v}{u}\partial_v, p^*\partial_t = \frac{1}{u}\partial_v$ and $p^*\partial_y = \partial_y$ and so 
\begin{align*}p^*((A_1x^k+A_2t)\partial_x+B_2t\partial_t+(C_1x^k+C_2t)\partial_y) =& \\
(\tilde A_1u^k+\tilde A_2uv)\partial_u +(\tilde B_2v-\tilde A_1vu^{k-1}-\tilde A_2v^2)&\partial_v+(\tilde C_1u^k+\tilde C_2uv)\partial_y,
\end{align*}
where we denote by $\tilde{A}_i$ (resp. $\tilde{B}_i$, resp. $\tilde{C}_i$) the pullback of $A_i$ (resp. $B_i$, resp. $C_i$)
by $p$).
Since $a = A_1x^k+A_2t, b = B_2t \in (x, t)^2$ and $A_1x^k \in (x, t)^2$ (by Claim \ref{claim_order2}), it follows that
$A_2, B_2 \in (x, t)$, hence $u|\tilde A_2$ and 
$u|\tilde B_2$. Since $A_1x^k \in (x, t)^2$ we see that $u^2|\tilde A_1u^k$ and so $u| (\tilde B_2v-\tilde A_1vu^{k-1}-\tilde A_2v^2)$.  It follows that $p^*\partial$ vanishes along $\{u = 0\}$, i.e., $p$ extracts a divisor of discrepancy $\le-1$.
\end{proof}

We argue in cases based on whether $X_0$ is normal at the generic point of $C$ or not.

{\bf Case I: }{\it  Suppose that $X_0$ is normal at the generic point of $C$.}

In this case we must have $\partial(t) = 0$ 
and so $b = 0$.
Consider the coordinate chart for the blow up $p\colon X' \rightarrow X$ in $C$ 
given by $(u, v, y) \mapsto (uv, u, y) = (t, x, y)$.
In these coordinates we can verify that 
\begin{align*}p^*((A_1x^k+A_2t)\partial_x+(C_1x^k+C_2t)\partial_y)) =& \\
(\tilde A_1u^k+\tilde A_2uv)\partial_u - &(\tilde A_1vu^{k-1}+\tilde A_2v^2)\partial_v+(\tilde C_1u^k+\tilde C_2uv)\partial_y
\end{align*}
in particular this blow up extracts an invariant divisor 
of discrepancy $\leq 0$.
By Claim \ref{claim_order2} $A_1x^k \in (x, t)^2$ and so 
we see that 
\begin{align*}(\tilde A_1u^k+\tilde A_2uv), (\tilde A_1vu^{k-1}+\tilde A_2v^2)
\in (u, v)^2.\end{align*}  By Claim \ref{claim_easy_discrep_calc} a blow up centred in $\{u = v = 0\}$
will extract a divisor of discrepancy $\leq -1$, a contradiction.

\medskip 

{\bf Case II: }{\it  Suppose that $X_0$ is not normal at the generic point of $C$.}

In this case $X_0 = \{x^rt^s = 0\}$ where $r, s \in \mathbb N$ and so 
\begin{align*}0 = \partial(x^rt^s) = rx^{r-1}t^sa+sx^rt^{s-1}b.\end{align*}
In particular, we see that $x|a$ and $t|b$.
Since $x|a$ we see that $x|A_2$ and since $A_1x^k \in (x, t)^2$ by Claim \ref{claim_order2} we see that $a\in (x, t)^2$.
Since $rta+sxb = 0$ we then deduce that $b \in (x, t)^2$ and so 
Claim \ref{claim_easy_discrep_calc} shows that
the blow up in $\{x = t = 0\}$ extracts a divisor of discrepancy $\le-1$, a contradiction.
\end{proof}

\begin{remark}
\label{rem_very_general_canonical}
    Let $X$ be a normal variety, let $f\colon X \to T$ be a fibration over a curve and let $\mathcal F$ be a foliation on $X$ such that $T_{\mathcal F} \subset T_{X/T}$.  We may restrict $\mathcal F$
 to a foliation $\mathcal F_\eta$ on the generic fibre.  It is clear that if $\mathcal F$ has canonical singularities, then $\mathcal F_\eta$ has canonical singularities.  However, as Example \ref{easiest_example} shows, it may not be the case that $\mathcal F_t$ has canonical singularities for all closed points in some Zariski open subset of $T$.  However, \cite[Fact III.i.3]{MP13} shows that 
 $\mathcal F_t$ will have canonical singularities for a very general point of $T$.
 \end{remark}

\subsection{Proof of Conjectures \ref{conjecture sing} and \ref{conj_inversion_adjunction} for rank one foliations}

\begin{theorem}
\label{thm_inversion_of_adjunction}
    Let $X$ be a normal variety, let $\mathcal F$ be a rank one integrable 
    distribution such that $K_{\mathcal F}$ is $\mathbb Q$-Cartier.  Let $V \subset X$
be a $\mathcal F$ invariant subvariety not contained in $\sing^+ \mathcal F$, 
let $n\colon W \rightarrow V$ be the normalisation and let $\mathcal F_W$
be the restricted integrable distribution, cf. \cite[Proposition-Defintion 3.12]{CS23b}.

Suppose that $\mathcal F_W$ has canonical (resp. log canonical) singularities, then $\mathcal F$ has canonical (resp. log canonical) singularities in a neighbourhood of $V$.
 \end{theorem}
 \begin{proof}
The problem is local about any closed point $x \in V$, so 
we may freely replace $X$ by a neighbourhood 
of $x$.
By \cite[Lemma 2.20]{SS23} it suffices to check that $\mathcal F$ is log canonical 
on any quasi-\'etale cover, and so by replacing $X$ by the index one
cover associated to $K_{\mathcal F}$ we may assume that 
$K_{\mathcal F}$ is Cartier.  Thus $\mathcal F$
is defined by a vector field $\partial$.

If $x \notin \sing \partial$, 
then $\mathcal F$ is terminal at $x \in X$, see \cite[Fact III.i.1]{MP13},
and there is nothing more to show.  So we may
assume that $x \in \sing \partial$.

Let $\mathfrak m$ be the ideal of $x \in X$ and let 
$\mathfrak n$ be the ideal of $x \in V$. 
Let $\partial_V$ denote the restriction of $\partial$
to a vector field on $V$, and let $\partial_W$ denote the lift of 
$\partial_V$ to $W$, which exists by Lemma \ref{lem_nonnilp_normal}.(0).

For any closed point $y \in n^{-1}(x)$, since $\mathcal F_W$ has log canonical singularities, Proposition \ref{prop_lc_nonnilp} implies that
the linear map 
$\partial_{W, 1} \colon \mathfrak m_y/\mathfrak m_y^2 \to \mathfrak m_y/\mathfrak m_y^2$ is non-nilpotent, where
$\mathfrak m_y$ is the ideal of $y$.
  We then apply Lemma \ref{lem_nonnilp_normal}
to deduce that $\partial_{V, 1}\colon \mathfrak n/\mathfrak n^2 \rightarrow \mathfrak n/\mathfrak n^2$ is non-nilpotent.
Since we have a surjection $\mathfrak m/\mathfrak m^2 \rightarrow \mathfrak n/\mathfrak n^2$
we deduce that $\partial_1 \colon \mathfrak m/\mathfrak m^2 \rightarrow \mathfrak m/\mathfrak m^2$ is non-nilpotent.
We conclude that $\mathcal F$ has log canonical singularities by applying Proposition \ref{prop_lc_nonnilp} and Corollary \ref{cor_lc_is_open}.

Suppose that $\mathcal F_W$ has canonical singularities.
By \cite[Fact III.i.3]{MP13} $\mathcal F_W$ is defined 
by a non-radial vector field, i.e., a vector field not of the form $u(\sum a_i x_i \partial_{x_i})$ where $a_i \in \mathbb Z_{\ge 0}$ and $u$ is a unit.  It follows that $\partial$ is not radial and therefore we may apply \cite[Fact III.i.3]{MP13}
to conclude that $\mathcal F$ has canonical singularities.
 \end{proof}

\begin{theorem}
\label{thm_inversion_of_adjunction_full}
Let $X$ be a normal variety, let $\mathcal F$ be a rank one foliation on $X$, let $S$ be a prime divisor on $X$ such that $K_{\mathcal F}+\varepsilon(S)S$ is $\mathbb Q$-Cartier.
Let $n\colon T \rightarrow S$ be the normalisation and write (via foliation adjunction, cf. \cite[Proposition-Definition 3.7]{CS23b}) \begin{align*}n^*(K_{\mathcal F}+\varepsilon(S)S) = K_{\mathcal F_T}+B_T.\end{align*} 
If $(\mathcal F_T, B_T)$ is log canonical, then 
$(\mathcal F, \varepsilon(S)S)$ is log canonical in a neighbourhood of $S$.
\end{theorem}
\begin{proof}
    If $\varepsilon(S) = 0$, then we may apply Theorem \ref{thm_inversion_of_adjunction} to conclude.
    If $\varepsilon(S) = 1$, then we may apply \cite[Theorem 3.16]{CS23b} to conclude.
\end{proof}

 \begin{theorem}
 \label{thm_lc_deform}
     Let $f\colon (X, \mathcal F) \rightarrow T$
     be a flat family of integrable distributions of rank $=1$ where $T$ is scheme of finite type over $\mathbb C$, $\dim (X/T) = 2$, $K_{\mathcal F}$ is $\mathbb Q$-Cartier and such that the fibres of $f$ are deminormal.

Fix $t \in T$, and suppose that $\mathcal F_t$ is semi-log canonical.  Then there exists a Zariski open neighbourhood $t \in U \subset T$ such that 
$\mathcal F_s$ is semi-log canonical for all $s \in U$.
 \end{theorem}
\begin{proof}
We may freely replace $T$ by a reduced curve $\Sigma \subset T$ and so we may assume that $\dim X = 3$
and that $X$ is a variety.
We are also free to replace $X$ by its normalisation and so without loss of generality we may assume that $X$ is normal.  The theorem is then an immediate consequence of Theorem \ref{thm_inversion_of_adjunction} and Proposition \ref{prop_locally_stable_adjunction}.
\end{proof}

\subsection{A criterion for local stability}
In this section we explain a criterion for verifying that 
a family of rank one foliations on surfaces is locally stable.

\begin{lemma}
\label{lem_picard_node}
    Let $x \in X = X_1 \cup X_2$ be a semi-log canonical surface where $X_i$ is irreducible and normal and suppose that $X_1 \cap X_2$ is a smooth curve.

Let \begin{align*}\phi\colon {\rm Pic}^{\rm loc}(x \in X)  \rightarrow {\rm Pic}^{\rm loc}(x \in X_1)\oplus {\rm Pic}^{\rm loc}(x \in X_2)\end{align*}
be the map given by $D \mapsto (D|_{X_1}, D|_{X_2})$.

Then $\ker \phi \cong \mathbb Z$.  In particular, if $D$ is $\mathbb Q$-Cartier and, 
$D|_{X_1}$ and $D|_{X_2}$ are Cartier, then $D$ is Cartier.
\end{lemma}
\begin{proof}
This follows from essentially the same arguments as \cite[Example 4.5]{MR3549172}:
 if $D|_{X_1}$ and $D|_{X_2}$
are both Cartier (and hence trivial in the local Picard group), then the obstruction to gluing them to a Cartier divisor on $X$ is an element of ${\rm Pic}^{\rm loc}(x \in X_1 \cap X_2) \cong \mathbb Z$.
\end{proof}

\begin{lemma}
\label{lem_Cartier_criterion}
Consider the following set up.
\begin{enumerate}
\item Let $f\colon (X, \Delta) \rightarrow S$ be a locally stable family where $\dim (X/S) = 2$ and $S$ is a smooth curve. 
\item Let $\mathcal F$ be a rank one integrable distribution
 on 
$X$ such that 
\begin{enumerate}
    \item  $T_{\mathcal F} \subset T_{X/S}$ in a neighbourhood of every generic point of $X$; and 
    \item $K_{\mathcal F}$ is $\mathbb Q$-Cartier and $\mathcal F$ is 
semi-log canonical.
\end{enumerate}
\end{enumerate}

Then for any $s \in S$ and any codimension two point $W \in X$ such that $W \in X_s$, $K_{\mathcal F}$
is Cartier in a neighbourhood of $W$.
\end{lemma}
\begin{proof}
Either $X$ is normal in a neighbourhood of the generic point of $W$ or it is not, in which case 
$X$ is normal crossings at the generic point of $W$.

{\bf Case I: }{\it Suppose that $X$ is normal 
in a neighbourhood of $W$.}
Since $f\colon (X, \Delta) \rightarrow S$ is locally stable, we know that $X_s$ is semi-log canonical.  In particular, either 
$X_s$ is smooth at the generic point of $W$ or it is normal crossings at the generic point of $W$.  We will handle these cases separately.

{\bf Case I.a: }{\it  Suppose $X_s$ is smooth at the generic point of $W$.}
Since $X_s$ is a Cartier divisor this implies that $X$ is smooth at the generic point of $W$
and so $K_{\mathcal F}$ is Cartier at $W$.

{\bf Case I.b: }{\it  Suppose that $X_s$ is normal crossings 
at the generic point of $W$.}
In this case, in a neighbourhood of general point  $x \in \overline{\{W\}}$, $X \rightarrow S$ gives a smoothing
of a normal crossing point and so in appropriate local coordinates about $x \in \overline{\{W\}}$,
we have $(x \in X) \cong (0 \in   \{xy - t^k = 0\} \subset \mathbb A^4)$
and the map to $S$ is given by $(x, y, z, t) \mapsto t$.
In particular, $X$ is a quotient of $\mathbb A^3$
by $\mathbb Z/k\mathbb Z$ via the action 
$\zeta \cdot (u, v, w) = (\zeta u, \zeta^{-1}v, w)$
where $\zeta$ is a primitive $k$-th root of unity.

Let $q\colon \mathbb A^3 \rightarrow X \subset \mathbb A^4$ be the quotient map. In coordinates this map
is given by $x = u^k, y = v^k, t = uv, z = w$.
Let $\partial = a\partial_u+b\partial_v+c\partial_z$ be a vector field generating $q^{-1}\mathcal F$ where $a, b, c \in \mathcal O_{\mathbb A^3, 0}$.
The condition that $T_{\mathcal F} \subset T_{X/S}$
implies that $\partial(q^*t) = \partial(uv) = 0$, and
so $va+ub = 0$.  This in turn implies that we may write
$a = va_0$ and $b = -ua_0$ where $a_0 \in \mathcal O_{\mathbb A^3, 0}$.
Up to rescaling $\partial$ by a unit, we may also assume that the $\mathbb Z/k\mathbb Z$-action
on $X$ acts on $\partial$ as $\partial \mapsto \zeta^r \partial$
where $r \in \{0, \dots, k-1\}$.  

We argue in cases based on whether $0 \in \sing \partial$ or $0 \not\in \sing \partial$.

If $0 \not\in \sing \partial$, then we must have that $c$ is a unit, so (up to rescaling by a constant) we may write $\partial = va_0\partial_v-ua_0\partial_u+\partial_w+\delta$ where $\delta \in (u, v, w)T_{\mathbb A^3, 0}$. Since $\partial_w$ is fixed by the $\mathbb Z/k\mathbb Z$-action, it follows that $\zeta^r = 1$, i.e., $\partial$
is fixed by by the $\mathbb Z/k\mathbb Z$-action. Thus $\partial$ descends to a vector field on $X$ and so $K_{\mathcal F}$
is Cartier.

If $0 \in \sing \partial$, then 
since $x \in \overline{\{W\}}$ is a general point, we may freely assume that $q^{-1}(W) = \{u = v = 0\} =  \sing \partial$, and so 
$c \in (u, v)$.
Since $\mathcal F$ is log canonical, $q^{-1}\mathcal F$ is log canonical as well, cf. \cite[Lemma 2.20]{SS23}, so by Proposition \ref{prop_lc_nonnilp} the linear part of $\partial$ at $0$ is non-nilpotent.
It follows that $a_0$ must be a unit, and so (up to rescaling by a constant) we may write 
$\partial = u\partial_u-v\partial_v+c\partial_w+\delta$ where $\delta \in (u, v)^2T_{\mathbb A^3, 0}$.
Since $u\partial_u-v\partial_v$ is fixed by the $\mathbb Z/k\mathbb Z$-action, we again conclude that $\zeta^r = 1$
and so
$\partial$
is fixed by $\mathbb Z/k\mathbb Z$-action.
Thus $\partial$ descends to a vector field on $X$
and so $K_{\mathcal F}$ is Cartier.

\medskip
{\bf Case II: }{\it  Suppose that $X$ is not normal 
in a neighbourhood of $W$.}
Since $X_s$ is semi-log canonical and not normal, it follows that $X_s$ is normal crossings at a general pont of the closure of $W$.
Since $X_s \subset X$ is a Cartier divisor and $X$ is not normal 
we deduce that $X$ is normal crossings at a general point of the closure of 
$W$ as well.
We may then conclude by applying Lemma \ref{lem_picard_node} at a general point of the closure of $W$.
\end{proof}

We make the following easy observation:

\begin{lemma}
\label{lem_rel_mumford_crit}
    Let $f\colon X \rightarrow S$ be a flat morphism with $S_2$ fibres, and let $D$ be a Mumford divisor
    on $X$ such that $D$ is Cartier at 
    every codimension two point of $X$ contained in a fibre of $f$.  
    Then there exists a relative Mumford divisor $D'$ such that $D \sim D'$.
\end{lemma}
\begin{proof}
This is a consequence of prime avoidance.
\end{proof}

\begin{proposition}
\label{prop_family_criterion}
Consider the following set up.
\begin{enumerate}
\item Let $f\colon (X, \Delta) \rightarrow T$ be a locally stable family where $\dim (X/T) = 2$ and $T$ is a smooth curve.
\item Let $\mathcal F$ be a rank one integrable distribution
 on 
$X$ such that 
\begin{enumerate}
  \item $T_{\mathcal F} \subset T_{X/T}$ in a neighbourhood of every generic point of $X$; \item $K_{\mathcal F}$ is $\mathbb Q$-Cartier and  $\mathcal F$ is 
semi-log canonical; and 
\item $(X_\eta, \mathcal F_\eta, \Delta_\eta)$ is well-formed where $\eta \in T$ is the generic point.
\end{enumerate}
\end{enumerate}

Then $(X, \Delta, \mathcal F) \rightarrow T$ is a locally stable family of integrable distributions.
\end{proposition}
\begin{proof}
We first check that $(X, \Delta, \mathcal F) \to T$ is a flat family of integrable distributions. 
 To do this it suffices to show that $\mathcal O(K_{\mathcal F})$ is a flat family of divisorial sheaves.
This may be checked locally, so we are free to replace 
$X$ by a neighbourhood of any point.
By Lemma \ref{lem_Cartier_criterion} and Lemma \ref{lem_rel_mumford_crit}
we see that there exists a divisor $D \sim K_{\mathcal F}$
such that $D$ is a relative Mumford divisor.
We may then apply \cite[Proposition 2.79]{modbook} (cf. \cite[Theorem 7.20]{Kollar13}) to $\mathcal O(D)$ to see that $\mathcal O(K_{\mathcal F})$
is a flat family of divisorial sheaves.

It now remains to show that the conditions for being locally stable are satisfied.

(S1) and (S2) both hold by assumption. 

(S3) is seen to hold by another application of  \cite[Proposition 2.79]{modbook} (cf. \cite[Theorem 7.20]{Kollar13}) which shows that 
$\mathcal O(iD+j(K_{X/T}+\Delta))$ is a flat family of divisorial sheaves for all $i, j \in \mathbb Z$
(we remark that $K_{X/T}+\Delta$ is a relative Mumford divisor whenever $(X, \Delta) \rightarrow T$ is locally stable).

(S4) holds by
Proposition \ref{prop_locally_stable_adjunction}. 

Finally, we check that (S5) holds. To check well-formedness we may freely replace $X$ by its normalisation, and so we may assume that $X$ is normal.  We may also freely replace $X$ by a neighbourhood of any codimension one subvariety of $X_t$, 
and so by Lemma \ref{lem_Cartier_criterion} we may assume that $K_{\mathcal F}$
is Cartier, in particular $\mathcal F$ is defined by a vector field $\partial$.

Let $t \in T$ be a closed point. Let $n\colon X^n_t \to X_t$ be the normalisation, let $\mathcal F^n_t$
be the pulled back integrable distribution, and let $(\tilde{\mathcal F}^n_t, \Gamma^n_t)$ be the associated foliated pair.  By definition, $\supp \Gamma^n_t$
is the union of the codimension one components of $\sing \partial^n$, where $\partial^n$ is the lift of $\partial$ to $X^n_t$, which exists by Lemma \ref{lem_nonnilp_normal}.(0).
Let $(\tilde{\mathcal F}, \Gamma)$ be the foliated pair associated to $\mathcal F$.
Observe that if $D$ is any component of $\Delta$ and if $D_t$ is not
$\mathcal F_t$-invariant, then $D$ is not $\mathcal F$-invariant.
It follows that 
\begin{align*}(\Delta_t)_{{\rm non-inv}} \leq (\Delta_{{\rm non-inv}})_t.\end{align*}
Since $(X_t, \Delta_t)$ is slc we see that  $\Delta_t$ is reduced for all $t$, moreover, by assumption $\Delta_{{\rm non-inv}} \leq \Gamma$, and it is easy to see that $n^*\Gamma \leq \Gamma^n_t$. So to check 
that $(X_t, \Delta_t, \mathcal F_t)$ is well-formed it suffices to check for
any divisor $D \subset n^{-1}(\sing X_t)$ we have $m_D\Gamma^n_t \ge \varepsilon(\tilde{\mathcal F}^n_t, D)$.  

By \cite[Theorem 5]{MR0212027} $\sing X_t$
is $\partial$-invariant, and so by Lemma \ref{lem_nonnilp_normal}.(1) $D$ is $\partial^n$-invariant.
We argue in cases based on whether 
$n(D) \subset \sing \partial$ or not.

If $n(D) \subset \sing \partial$, then by Lemma \ref{lem_nonnilp_normal}.(2) $D \subset \sing \partial^n$ and so $D \subset \supp \Gamma^n_t$, hence $m_D\Gamma^n_t \ge 1$ (recall that $\Gamma^n_t$ is an integral divisor by construction).  

If $n(D) \not\subset \sing \partial$, then Lemma \ref{lem_nonnilp_normal}.(2) implies that 
$D \not\subset \sing \partial^n$.  However, since $D$ is $\partial^n$-invariant we conclude that $D$ is $\tilde{\mathcal F}^n_t$-invariant, i.e, $\varepsilon(\tilde{\mathcal F}^n_t, D) = 0$
and so clearly $m_D\Gamma_n^t \ge \varepsilon(\tilde{\mathcal F}^n_t, D)$, as required.

Since (S1)-(S5) all hold,  $(X, \Delta, \mathcal F) \rightarrow T$ is a 
locally stable family of integrable distributions.
\end{proof}

\begin{lemma}
    \label{lem_norm_loc_stable}
    Let $f\colon (X, \Delta, \mathcal F) \rightarrow T$
    be a locally stable family of integrable distributions where $\dim (X/T) = 2$ and $T$ is a smooth curve.
    Let $n\colon X^n \rightarrow X$ the normalisation, 
    let $\Delta^n = n^*\Delta+D$ where $D$ is the pre-image of the double locus and
    let $\mathcal F^n$ be the pulled back integrable distribution.

    Then $(X^n, \Delta^n, \mathcal F^n) \rightarrow S$
    is a locally stable family of integrable distributions.
\end{lemma}
\begin{proof}
By \cite[Proposition 2.12]{modbook} $(X^n, \Delta^n) \rightarrow T$ is locally stable, and so $K_{X^n/T}+\Delta^n$ is $\mathbb Q$-Cartier and $(X^n_t, \Delta^n_t)$ is semi-log canonical for all $t \in T$.

By construction $K_{\mathcal F^n} = n^*K_{\mathcal F}$ and so $K_{\mathcal F^n}$ is $\mathbb Q$-Cartier.

Notice that for any $t \in T$, the normalisation of $X^n_t$ is naturally isomorphic to the normalisation of $X_t$.
Thus it is clear that $(X^n_t, \Delta^n_t, \mathcal F^n_t)$ is well-formed and $\mathcal F^n_t$ is semi-log canonical for all $t$.

We may then apply Proposition \ref{prop_family_criterion} to conclude.
\end{proof}

\begin{lemma}
\label{lem_cover_s2}
    Let $f\colon (X, \Delta, \mathcal F) \rightarrow T$ be a locally stable family of rank one integrable distributions where 
    $\dim (X/T) = 2$ and $T$ is a smooth curve.
    Let $\sigma\colon Y \rightarrow X$ be quasi-\'etale morphism such that $\sigma_*\mathcal O_Y$ is an $S_2$-sheaf. 
    Then, \begin{align*}g= f \circ \sigma \colon (Y, \sigma^*\Delta, \sigma^{-1}\mathcal F) \rightarrow T\end{align*} is a locally stable family.
\end{lemma}
\begin{proof}
We first show that $(Y, \sigma^*\Delta) \rightarrow T$
is a locally stable family of pairs.

    We claim that $Y$ is deminormal.
    Since $\sigma$ is \'etale in codimension one and $X$ is deminormal it is immediate that $Y$ has at worst nodal singularities in codimension $1$.
    We now show that $Y$ is $S_2$.  We apply \cite[Proposition 5.4]{KM98}
    to deduce that $\mathcal O_Y$ is $S_2$, i.e., $Y$ is $S_2$.
    We may then apply \cite[Lemma 2.9]{modbook} to conclude
    that $(Y, \sigma^*\Delta) \rightarrow T$ is locally stable.

\medskip

Let $\sigma^n\colon Y^n \rightarrow X^n$ be the induced map between the normalisations of $Y$ and $X$.  Note that $\sigma^n$
is quasi-\'etale and so $K_{\mathcal G^n} = \sigma^*K_{\mathcal F^n}$.  Since $\mathcal F^n$
is log canonical we deduce that $\mathcal G^n$
is log canonical (the proof of \cite[Proposition 5.20]{KM98} works equally well here).  We apply 
Proposition \ref{prop_family_criterion} to conclude.
\end{proof}

\begin{remark}
    The condition in Lemma \ref{lem_cover_s2} that 
    $\sigma_*\mathcal O_Y$ is $S_2$ holds in the case that 
    $Y$ is the index one cover associated to a relative Mumford divisor $D$ such that $mD \sim 0$ where $m>0$.

    Indeed, by construction, we have $\sigma_*\mathcal O_Y = \bigoplus_{j = 0}^{m-1}\mathcal O(-jD))$. Since $D$ is a Mumford divisor,  $\mathcal O(kD)$ is $S_2$ for all $k \in \mathbb Z$ (see for instance \cite[p. 136]{modbook})
    and so $\sigma_*\mathcal O_Y$ is an $S_2$-sheaf.
\end{remark}

    \section{Separatedness and finiteness of automorphisms}

\begin{theorem} 
\label{thm_separatedness}
For $i = 1, 2$, let $f_i\colon (X_i, \Delta_i, \mathcal F_i) \rightarrow S$
be stable families of integrable distributions of rank one
where $S$ is a smooth curve.
Let $s \in S$ be a closed point and set $S^\circ = S \setminus \{s\}$.

Suppose that there is an isomorphism
\begin{align*}\phi^\circ\colon (X_1, \Delta_1, \mathcal F_1)\times_SS^\circ \rightarrow (X_2, \Delta_2, \mathcal F_2)\times_SS^\circ\end{align*}
over $S^\circ$.
Then $\phi^\circ$ extends to a unique isomorphism
\begin{align*}\phi\colon (X_1, \Delta_1, \mathcal F_1) \rightarrow (X_2, \Delta_2, \mathcal F_2)\end{align*}
over $S$.
\end{theorem}

\begin{proof}
We first consider the case where the generic fibre of $X_i \rightarrow S$ is normal. By assumption, we have a birational map $\phi\colon X_1 \dashrightarrow X_2$ which is an isomorphism 
over $S \setminus \{s\}$.  We may resolve
$\phi$ by maps $p_i\colon W \rightarrow X_i$  such that $\exc p_i$ 
is contained in $g^{-1}(s)$ where $g = p_i\circ f_1$.  Let $\mathcal F_W$ be the transform 
of $\mathcal F_1$ (equivalently the transform of $\mathcal F_2$) on $W$ and let $\Delta_W$ be the strict transform of $\Delta_1$ (equivalently $\Delta_2$) on $W$.

By Theorem \ref{thm_inversion_of_adjunction},
we know that $\mathcal F_i$ has log canonical singularities
and so if we write \begin{align*}K_{\mathcal F_W} = p_i^*K_{\mathcal F_i}+\sum a(E_j, \mathcal F_i)E_j\end{align*} then $a(E_j, \mathcal F_i) \geq -\varepsilon(E_j)$.  Next, notice that  
    $\exc p_i$ is $\mathcal F_W$ invariant:
indeed, $g^{-1}(s)$ is $\mathcal F_W$ invariant and since $\exc p_i$
is a union of irreducible components of $g^{-1}(s)$, $\exc p_i$ is likewise invariant. Since $\exc p_i$ is invariant we
have $\varepsilon(E_j) = 0$ and so in 
fact $a(E_j, \mathcal F_i) \geq 0$.

By inversion of adjunction for varieties, \cite{Kawakita05}, we know
that the pair $(X_i, \Delta_i+X_{i, s})$ is log canonical
and so $a(E_j, X_i, \Delta_i+X_{i, s}) \geq -1$.
Since $X_{i, s}$ is Cartier it follows that 
$m_{E_j}(p_i^*X_{i, s}) \geq 1$.  Thus $a(E_j, X_i, \Delta_i) \geq 0$.

It follows that for all $m, n \in \mathbb Z_{\ge 0}$
we have that \begin{align*}mK_{\mathcal F_W}+n(K_{W/S}+\Delta_W) = p_i^*(mK_{\mathcal F_i}+n(K_{X_i/S}+\Delta_i)+F_i\end{align*} 
where $F_i \geq 0$.

For some fixed integers $m \gg n >0$ we have that $mK_{\mathcal F_i}+n(K_{X_i/S}+\Delta_i)$ is $f_i$-ample,
and so
\begin{align*}F_1-F_2 = -p_1^*(mK_{\mathcal F_1}+n(K_{X_1/S}+\Delta_1)+p_2^*(mK_{\mathcal F_2}+n(K_{X_2/S}+\Delta_2))\end{align*} is $p_1$-nef
and $p_{1*}(F_1-F_2) = -p_{1*}F_2 \leq 0$.
We apply the negativity lemma, \cite[Lemma 3.39]{KM98} to conclude that $F_2-F_1 \ge 0$. 
Exchanging the roles of $X_1$ and $X_2$ we conclude 
that $F_1-F_2 \ge 0$ and so $F_1= F_2$.
In particular, 
\begin{align}
\label{eq_crepant}
   p_1^*(mK_{\mathcal F_1}+n(K_{X_1/S}+\Delta_1)) = p_2^*(mK_{\mathcal F_2}+n(K_{X_2/S}+\Delta_2)). 
\end{align}

Since $mK_{\mathcal F_i}+n(K_{X_i}+\Delta_i)$ is $f_i$-ample, 
we also have
$X_i \cong {\rm Proj}_S\bigoplus_{j \in \mathbb Z_{\ge 0}} f_{i*}\mathcal O(j(mK_{\mathcal F_i}+n(K_{X_i/S}+\Delta_i)))$.
The equality (\ref{eq_crepant}) implies that 
\begin{align*}\bigoplus_{j \in \mathbb Z_{\ge 0}} f_{1*}\mathcal O(j(mK_{\mathcal F_1}+n(K_{X_1/S}+\Delta_1))) \cong \bigoplus_{j \in \mathbb Z_{\ge 0}} f_{2*}\mathcal O(j(mK_{\mathcal F_2}+n(K_{X_2/S}+\Delta_2)))\end{align*} as $\mathcal O_S$-algebras and we may conclude.    

\medskip

We now consider the case where the generic fibre 
of $X_i \rightarrow S$ is not necessarily normal.
By Lemma \ref{lem_norm_loc_stable}, the normalised families
$(X_i^n, \Delta^n_i, \mathcal F^n_i) \rightarrow T$
are locally stable.  By the previous case we see 
that $(X_1^n, \Delta^n_1, \mathcal F^n_1) \cong (X_2^n, \Delta^n_2, \mathcal F^n_2)$.

If $W_i = \sing X_{i, s}$, we therefore see that 
$(X_1\setminus W_1, \Delta_1|_{X_1\setminus W_1}, \mathcal F_1|_{X_1\setminus W_1}) \cong (X_2 \setminus W_2, \Delta_2|_{X_2 \setminus W_2}, \mathcal F_2|_{X_2 \setminus W_2})$.
Since $X_i$ is $S_2$ and ${\rm codim}_{X_i}(W_i) \ge 2$
we may apply \cite[Theorem 11.39]{modbook} to deduce that 
$(X_1, \Delta_1, \mathcal F_1) \cong (X_2, \Delta_2, \mathcal F_2)$.
\end{proof}

Let $(X, \Delta, \mathcal F)$ be 
a stable integrable distribution.
We define $\Aut (X, \Delta, \mathcal F)$ to be the group of automorphisms $(X, \Delta, \mathcal F)$.

\begin{corollary} \label{cor_finite_automorphisms}
    Let $(X, \Delta, \mathcal F)$
    be a stable integrable distribution where $\mathcal F$
is rank one.  Then $\Aut (X, \Delta, \mathcal F)$ is a finite group.
\end{corollary}
\begin{proof}
    This is a formal consequence of Theorem \ref{thm_separatedness}: indeed the proof of \cite[Proposition 8.64]{modbook} applies equally well here.
\end{proof}

\section{Local representability}

The goal of this section is to show that stability is a representable condition
in families.  This is a key technical condition which will be crucial in showing that our moduli functor is representable by a Deligne-Mumford stack.
To be precise we will prove the following:

\begin{theorem}[= Corollary \ref{cor_stab_rep}]
Fix an integer $N>0$.  Let $(X, \Delta, \mathcal F) \rightarrow T$ be a family
    of rank one integrable distributions where $T$ is a scheme of finite type over $\mathbb C$ and $\dim (X/T) = 2$.

    Then, there exists a locally closed partial decomposition $T' \rightarrow T$ such that 
    a morphism $W \rightarrow T$ factors as $W \rightarrow T' \rightarrow T$
    if and only if $(X_W, \Delta_W, \mathcal F_W) \rightarrow W$ is a stable family of integrable distributions of index $=N$.
\end{theorem}

\subsection{Representability of positivity}

\begin{lemma}
\label{lem_mmp_implies_nef_open}
    Suppose that $X \rightarrow T$ is a proper morphism, and let $D$ be a $\mathbb Q$-Cartier divisor on $X$. Suppose that there exists a  $D$-minimal model over $T$, i.e., a $D$-negative rational contraction $\phi\colon X \dashrightarrow X'/T$ such that $D' = \phi_*D$ is nef over $T$.
    Then the set 
    \begin{align*}\{t \in T: D\vert_{X_t} \text{ is nef }\} \subset T\end{align*} is open.\end{lemma}
    \begin{proof}
   To prove the Lemma we may freely replace $X$ by its normalisation.   
        Let $W = \exc \phi$.  We claim that if $t \in T$ is a point such that $D\vert_{X_t}$ is nef then 
        $X_t \cap W = \emptyset$.  Supposing the claim, then the Lemma follows by Noetherian induction on $T$.

        We now prove the claim. Let $W$ be the normalisation of the closure of the graph of $\phi$ inside $X \times_T X'$, and let $p\colon W \rightarrow X$ and $q\colon W \rightarrow X'$ be the projections.  On one hand,  since $\phi$ is $D$-negative, $p^*D-q^*D' = E \ge 0$.  On the other hand, if $n\colon \tilde{W}_t \rightarrow W_t$ is the normalisation and $\mu\coloneqq q\circ n$, then $n^*(p^*D-q^*D')$ is $\mu$-nef.

        Since $q_*E = 0$ it follows that $\mu_*n^*E = 0$, 
        and so the Negativity Lemma, \cite[Lemma 3.39]{KM98} implies that $-n^*E \ge 0$.
        Thus $n^*E = 0$ and so $X_t \cap E = \emptyset$, and so $X_t \cap W = \emptyset$.
    \end{proof}

\begin{proposition}
\label{prop_positivity_is_open}
    Let $(X, \Delta, \mathcal F) \rightarrow T$ be a proper
    locally stable family of rank $1$ integrable distributions where $T$ is a scheme of finite type over $\mathbb C$ and such that $\dim (X/T) = 2$.
    Fix $m, n \in \mathbb Z$.
    Then, the subsets of $T$
    \begin{enumerate}
        \item $T_{m, n}  := \{t \in T : mK_{\mathcal F_t}+nK_{X_t} \text{ is ample}\}$; and 
        \item $T_+ := \{t \in T : rK_{\mathcal F_t}+K_{X_t} \text{ is ample for all } r \gg 0\}$
    \end{enumerate}
are Zariski open.
\end{proposition}
\begin{proof}
We may freely replace $T$ by an open affine around an arbitrary point and by its reduction, so we may assume that $T$ is a reduced affine scheme.

    The openness of $T_{m, n}$ for fixed $m, n \in \mathbb N$ is an immediate consequence of the fact that for any line bundle $L$ on $X$ 
    the subset $\{t \in T: L_t \text{ is ample}\}\subset T$ 
    is open.

    We now show that $T_+$ is open. 
  We define 
    \begin{align*}V := \{t \in T: K_{\mathcal F_t} \text{ is nef }\} \subset T.\end{align*}
We claim that $V$ is open.
To see this, we are free to base change over the normalisation of a 
curve $C \to B \subset T$, and so we may freely assume that $\dim X = 3$.  Moreover, we may freely replace $X$ by its normalisation
    and so we may assume that $X$ is normal.

Let $(\tilde{\mathcal F}, \Gamma)$ be the foliated pair associated to $\mathcal F$.
Since $X \rightarrow T$ is locally stable it follows that 
$X$ has klt singularities and therefore by Theorem \ref{thm_mmp} 
a minimal model over $T$ for $(\tilde{\mathcal F}, \Gamma)$ exists. 
We may then apply Lemma \ref{lem_mmp_implies_nef_open} to
conclude.

    \begin{claim}
     Let $N$ denote the Cartier index of $K_{\mathcal F}$, then
    $T_+ = T_{5N, 1} \cap V$.
    \end{claim}
    \begin{proof}

   We first show that $T_+ \subset T_{5N, 1}\cap V$.
Let $t \in T_+$. Proposition \ref{prop_threshold}
shows that $K_{\mathcal F_t}$ is 
nef and $5K_{\mathcal F_t}+(K_{X_t}+\Delta_t)$ is ample, hence $t \in T_{5N, 1} \cap V$.

   The reverse inclusion follows from observing that 
   $T_{5N, 1} \cap V = T_{m, 1} \cap V$
   for $m \geq 5N$ and that $T_+ = \cap_{r \gg 0} T_{r, 1}$.       
    \end{proof}

The above claim and the openness of $T_{5N, 1}$ and $V$ imply that $T_+$ is open as required.
 \end{proof}

\subsection{Representability of local stability}

We recall the following result found in \cite[Proposition 3.31]{modbook}.
We refer to \cite[Defintion 3.27]{modbook} for the definition 
of the hull pull-back $F^H_W$ of a mostly flat family of divisorial sheaves and
  we refer to \cite[Definition 10.83]{modbook} for the definition of a locally closed partial decomposition.

\begin{theorem}
\label{t_prop3.31}
    Let $f:X\to S$ be a flat, projective morphism with $S_2$ fibers and $N_1,\dots,N_s,L_1,\dots,L_r$ mostly flat families of divisorial sheaves. Then there is a locally closed partial decomposition $S^{NL} \to S$ such that, a morphism $q: T\to S$ factors through $S^{NL}$ if and only if the following hold.
    \begin{enumerate}
        \item The hull pull-backs $(L_j)^H_T$ are invertible, and
        \item the $(N_i[\otimes] L_1^{[m_1]}[\otimes]\dots[\otimes]L_r^{[m_r]})^H_T$ are flat families of divisorial sheaves for every $m_i\in \ZZ$.
    \end{enumerate}
\end{theorem}

We now show that local stability of a family of foliations is a representable condition.

\begin{theorem}
\label{thm_local_stab_rep}
Fix an integer $N>0$.  Let $(X, \Delta, \mathcal F) \rightarrow T$ be a flat family
    of rank one integrable distributions where $T$ is a scheme of finite type over $\mathbb C$
    and suppose that $\dim (X/T) = 2$.

    Then, there exists a locally closed partial decomposition $T' \rightarrow T$ such that 
    a morphism $W \rightarrow T$ factors as $W \rightarrow T' \rightarrow T$
    if and only if $(X_W, \Delta_W, \mathcal F_W) \rightarrow W$ is a locally stable family of integrable distributions of index $=N$.
\end{theorem}
\begin{proof}
    By \cite[Theorem 3.2 (see also the proof of Theorem 6.24)]{modbook}
    we know that $(X, \Delta) \rightarrow T$ being locally stable is a representable condition and so we may freely assume that $(X, \Delta) \rightarrow T$ is locally stable and so Condition (S1) in the definition of local stability holds. Since $(X, \Delta) \to T$ is locally stable, $K_{X/T}+\Delta$ is $\mathbb Q$-Cartier and so there exists $M>0$ such that $M(K_{X/T} +\Delta)$ is Cartier. We then apply Theorem \ref{t_prop3.31} with $L_1 = \mathcal O(NK_{\mathcal F})$, $L_2 = \mathcal O(M(K_{X/T}+\Delta))$ and \begin{align*}\{N_1, \dots, N_s\} = \{\mathcal O(jK_{\mathcal F}+k(K_{X/T}+\Delta)): j \in \{0, \dots, N\}, k \in \{0, \dots, M\}\}\end{align*} to give us a locally closed partial decomposition $T_K \to T$
 such that a morphism
 morphism $W \to T$ factors as $W \to T_K \to T$
    if and only if $(X_W, \Delta_W, \mathcal F_W) \to W$
    satisfies conditions (S1)-(S3) of the definition of local stability of index $=N$.

    We now show that after replacing $T_K$ by a locally closed subset that we can ensure that conditions (S4) and (S5) hold as well.
    Indeed, by Theorem \ref{thm_lc_deform} the subset \begin{align*}\{t \in T_K: \mathcal F_t \text{ is semi log canonical}\} \subset T_K\end{align*} is open.
Arguing as in the proof of Proposition \ref{prop_family_criterion} we see that the set of 
    points \begin{align*}\{t \in T_K: (X_t, \Delta_t, \mathcal F_t) \text{ is well-formed}\} \subset T_K\end{align*}
    is closed.

    Thus, taking $T' \subset T_K$ to be the locally closed subset of points corresponding to well-formed semi-log canonical integrable distributions gives our desired partial decomposition.
\end{proof}

\begin{corollary}
\label{cor_stab_rep}
Fix an integer $N>0$.  Let $(X, \Delta, \mathcal F) \rightarrow T$ be a flat family
    of rank one integrable distributions where $T$ is a scheme of finite type over $\mathbb C$ and suppose that $\dim (X/T) = 2$.

    Then, there exists a locally closed partial decomposition $T' \rightarrow T$ such that 
    a morphism $W \rightarrow T$ factors as $W \rightarrow T' \rightarrow T$
    if and only if $(X_W, \Delta_W, \mathcal F_W) \rightarrow W$ is a stable family of integrable distributions of index $=N$.
\end{corollary}
\begin{proof}
   This follows from Theorem \ref{thm_local_stab_rep} 
   and Proposition \ref{prop_positivity_is_open}.
\end{proof}

\section{Locally stable reduction}

In this section we prove a locally stable reduction statement for families of foliations over curves.  This will form a key ingredient in verifying the valuative criterion of properness for our moduli functor.

\begin{theorem}
\label{thm_locally_stable_reduction}
Let $f\colon X \rightarrow S$ be a projective morphism where $X$ is normal, 
$\dim X = 3$ and $S$ is a smooth curve.  Let $\mathcal F$ be a foliation on $X$ such that
$T_{\mathcal F} \subset T_{X/S}$.  Let $D$ be a reduced divisor.

Then there exists a finite morphism $T \rightarrow S$ and a birational modification 
$\mu\colon Y \rightarrow X\times_ST$
such that 
\begin{enumerate}
\item $Y$ is smooth, $(Y, B) \rightarrow S$
is semi-stable in the sense of \cite{KKMS} 
where $B = \mu_*^{-1}D_T+{\rm Exc}(\mu)$, in particular, $(Y, B) \to T$ is toroidal with reduced fibres;
and 

\item  $(\mathcal F_Y, B_{{\rm non-inv}})$ has canonical singularities, 
where $\mathcal F_Y$ is the pull-back of $\mathcal F$ along $Y \rightarrow X$.
\end{enumerate}
In particular, $(Y, B, \mathcal G) \rightarrow T$ is a locally stable family of integrable distributions
where $\mathcal G$ is the integrable distribution associated to
$(\mathcal F_Y, B_{{\rm non-inv}})$.
\end{theorem}
\begin{proof}
The problem is local about any point $0 \in S$, so without loss of generality we may shrink $X$
about a neighbourhood of $f^{-1}(0)$.

{\bf Step 1: }{\it We perform a reduction of singularities of $K_{\mathcal F}$.}
After replacing $(X, D)$ by a resolution of singularities, we may apply \cite[III.iii.4.bis]{MP13} to find a reduction of singularities $\nu\colon \tilde{X} \rightarrow X$ of 
$(\mathcal F, D_{{\rm non-inv}})$, 
so that $(\tilde{\mathcal F}, E_{{\rm non-inv}})$ has canonical singularities where 
$\tilde{\mathcal F} \coloneqq \nu^{-1}\mathcal F$ and 
$\tilde{D} = {\rm exc}(\nu)+\nu_*^{-1}D$.
We remark that a priori \cite{MP13} requires the use of weighted blow ups, 
but since $\mathcal F$ is tangent to a fibration only smooth blow ups are required, cf. \cite{CR14}, and so 
we have that $\tilde{X}$ is smooth and that $K_{\tilde{\mathcal F}}+\tilde{D}_{{\rm non-inv}}$ is Cartier.
Set $\tilde{f}:= f \circ \nu \colon \tilde{X} \to S$.

{\bf Step 2: }{\it  We resolve the singularities of $(\tilde{X}, \tilde{D})$.}
Since $K_{\tilde{\mathcal F}}+\tilde{D}_{{\rm non-inv}}$ is Cartier we see that \begin{align*}T_{\tilde{\mathcal F}}(-\log \tilde{D}_{{\rm non-inv}}) := T_{\tilde{\mathcal F}} \cap T_X(-\log \tilde{D}_{{\rm non-inv}}) = T_{\tilde{\mathcal F}} \cap T_X(-\log \tilde{D})\end{align*}
is locally free.  
If $\partial$ is any local generator of $T_{\tilde{\mathcal F}}(-\log \tilde{D}_{{\rm non-inv}})$ we see that $\tilde{D}+\tilde{f}^{-1}(0)$ is invariant by $\partial$.  
 Let $p\colon \tilde{\tilde{X}} \rightarrow \tilde{X}$ be a functorial log resolution of $(\tilde{X}, \tilde{D}+\tilde{f}^{-1}(0))$, cf. \cite[ Theorems 3.35 and 3.45]{kollar07b}.  Set $\tilde{\tilde{D}} = \exc p+ p_*^{-1}\tilde{D}$ and $\tilde{\tilde{\mathcal F}} = p^{-1}\tilde{\mathcal F}$. By \cite[Corollary 4.7]{MR2581247} there exists a vector field 
$\partial'$ on $\tilde{\tilde{X}}$ such that $p_*\partial' = \partial$.
Since $\tilde{\mathcal F}$ has canonical singularities it is non-dicritical, \cite[Corollary III.i.4]{MP13}, and so every $p$-exceptional divisor is $\tilde{\tilde{\mathcal F}}:=p^{-1}\tilde{\mathcal F}$ invariant.  It follows that 
\begin{align*}K_{\tilde{\tilde{\mathcal F}}}+\tilde{\tilde{D}}_{{\rm non-inv}} = 
p^*(K_{\tilde{\mathcal F}}+\tilde{D}_{{\rm non-inv}})\end{align*}
and so $(\tilde{\tilde{\mathcal F}},  \tilde{\tilde{D}}_{{\rm non-inv}})$ 
has canonical singularities.

So, up to replacing $X$ by $\tilde{\tilde{X}}$ we may freely assume that 
\begin{itemize}
    \item $(X, D+f^{-1}(0)_{\rm red})$ is an snc pair; and
    \item $(\mathcal F, D_{{\rm non-inv}})$ has canonical singularities.
\end{itemize}

{\bf Step 3: }{
\it We now perform a semi-stable reduction of 
$(X, D) \rightarrow S$.} Following \cite[Chapter 2]{KKMS} (cf. also \cite[Theorem 2.58]{modbook})
we may find a base change $S' \rightarrow S$ 
and a modification $\pi\colon Y \rightarrow W$ where $W$ is the normalisation of 
$X\times_SS'$ such that 
$(Y, B\coloneqq {\rm exc}(\pi)+\pi_*^{-1}(D\times_SS')) \rightarrow S'$ is semi-stable.
Following \cite[Chapter 2]{KKMS} we see that $\pi$ is a sequence of blow ups
which are toroidal with respect to the embedding $W \setminus (D_W \cup f_W^{-1}(0')) \subset W$
where $D_W$ is the pull back of $D$ to $W$, $f_W \colon W \to S'$ is the natural map and $0'$ is the pre-image of $0$ under $S' \to S$.

{\bf Step 4: }{\it 
We will show that if $\mathcal G$ is the pull-back of $\mathcal F$ to $Y$, 
then $(\mathcal G, B_{{\rm non-inv}})$ has canonical singularities.}

Let $\mathcal F_{W}$ be the pull-back of $\mathcal F$ 
to $W$.
Since the ramification of $r\colon W \rightarrow X$ is $\mathcal F_W$ invariant,
we have by \cite[Lemma 3.4]{Druel21} (cf. also \cite[Proposition 2.2]{CS20})
that $(K_{\mathcal F_W}+D_{W, {\rm non-inv}}) = r^*(K_{\mathcal F}+D_{{\rm non-inv}})$.
The proof of \cite[Proposition 5.20]{KM98} 
applies here to show that $(K_{\mathcal F_{W}}, D_{W, {\rm non-inv}})$ has canonical singularities.
We also remark that $K_{\mathcal F_W}+D_{W, {\rm non-inv}}$ is Cartier.

As noted, $\pi$ is a sequence of blow ups
which are toroidal with respect to the embedding $W \setminus (D_W \cup f_W^{-1}(0')) \subset W$
where $f_W \colon W \to S'$ is the natural map and $0'$ is the pre-image of $0$ under $S' \to S$.
Since every component of $D_W+W_0$ is invariant by any local section of  $T_{\mathcal F_W}(-\log D_{W, {\rm non-inv}})$ 
it follows
that every toroidal strata of $D_W+W_0$ is invariant by $\tilde{\mathcal F}_W$, cf. \cite[Theorem 5]{MR0212027}.  
In particular, it follows that $\pi$ is a sequence of blow ups in $\mathcal F_W$ invariant centres.
We then apply \cite[Proposition I.2.4]{bm16} to deduce
 \begin{align*}(K_{\mathcal G}+B_{{\rm non-inv}})+F= 
 \pi^*(K_{\mathcal F_W}+D_{W, {\rm non-inv}})\end{align*} where $F \ge 0$.
Since $\mathcal F_W$ has canonical singularities, it is non-dicritical, cf. \cite[Corollary III.i.4]{MP13}, and so every $\pi$-exceptional divisors is invariant. It follows that $B_{{\rm non-inv}} 
= \pi_*^{-1}E_{W, {\rm non-inv}}$ and hence $F = 0$.
  Since
 \begin{align*}(K_{\mathcal G}+B_{{\rm non-inv}})= 
 \pi^*(K_{\tilde{\mathcal F}_W}+E_{W, {\rm non-inv}})\end{align*}
 it follows immediately that $(\mathcal G, B_{{\rm non-inv}})$ has canonical singularities
 and we are done.
\end{proof}

\begin{remark}
    Observe that Theorem \ref{thm_locally_stable_reduction} holds equally well in the case where $\mathcal F$ is only assumed to be 
    an integrable distribution.  Indeed, if $(\tilde{\mathcal F}, E)$ is the foliated pair associated to $\mathcal F$, then we can apply Theorem \ref{thm_locally_stable_reduction} to $\tilde{\mathcal F}$ and $D = E$.
    \end{remark}

\section{Index in families}
\label{s_index_families}
The goal of this section is to study how the Cartier index of the canonical divisor of an integrable distribution varies in families.
As explained earlier, in order to get a finite type moduli functor for foliations it is necessary to fix the Cartier index of the canonical divisor, moreover, as seen in Example \ref{example_index} it is not possible to bound the Cartier index of the canonical divisor of a semi-log canonical integrable distribution in terms of other invariants of the foliation.  

While we are not in general able to control the Cartier index of the canonical divisor of the foliation on the central fibre of a family in terms of the Cartier index of the canonical divisor on the generic fibre, see Example \ref{slt_index_jumps} below, we are nevertheless able to do so in a wide range of cases, see
Theorem \ref{thm_index_is_stable}.
We remark that Theorem \ref{thm_index_is_stable}
is not needed in the construction of our moduli space, nor in the proof of properness of our moduli functor, however, we do expect that it will be necessary in future investigations of the boundedness properties of the moduli functor.

\begin{proposition}
\label{prop_lt_index}
    Let $f\colon (x \in X, \Delta, \mathcal F) \rightarrow T$ be a germ of a locally stable family of rank one integrable distributions where $T$ is a smooth curve and $\dim (X/T) =2$.  Let $0 \in T$ be a closed point.  
    Suppose that $\mathcal F_0$ is log terminal and that $NK_{\mathcal F_t}$
    is Cartier for all $t \in T \setminus \{0\}$.  
    Then $NK_{\mathcal F_{0}}$ is Cartier.
\end{proposition}

\begin{proof}
By definition $K_{\mathcal F}$ is $\mathbb Q$-Cartier, so up to replacing $X$ by a neighbourhood of $x$ there exists an integer $m >0$ such that $mK_{\mathcal F} \sim 0$. 
We fix $m$ to be the minimal positive integer for which $mK_{\mathcal F} \sim 0$ around $x \in X$.

Let $\sigma\colon Y \rightarrow X$ be the index one cover associated to $K_{\mathcal F}$ with Galois group $G = \mathbb Z/m\mathbb Z$. Notice that $\sigma$ is quasi-\'etale.
Without loss of generality we may assume that $y = \sigma^{-1}(x)$ is a single point.
Since $K_{\mathcal F}$ is a relative Mumford divisor (see Lemma \ref{lem_rel_mumford_crit}), by Lemma \ref{lem_cover_s2}, $g = f\circ \sigma \colon  (Y, \sigma^*\Delta, \mathcal G \coloneqq \sigma^{-1}\mathcal F) \rightarrow T$ is a locally stable family.

\begin{claim}
\label{cl_lt}
    $\mathcal G_0$ is a log terminal integrable distribution.
\end{claim}
\begin{proof}[Proof of Claim]
Since $\sigma$ is quasi-\'etale we have that $K_{\mathcal G} = \sigma^*K_{\mathcal F}$, and 
therefore $K_{\mathcal G_0} = \sigma_0^*K_{\mathcal F_0}$ where $\sigma_0\colon Y_0 \rightarrow X_0$ is the restricted map.  
By assumption $\mathcal F_0$ is log terminal, so arguing along the lines of the proof of \cite[Proposition 5.30]{KM98}
we see that $\mathcal G^0$ is log terminal.
\end{proof}

By \cite[Fact III.i.1]{MP13}, $\mathcal G_0$ in fact has terminal singularities and so $Y_0$ is a smooth surface, hence $Y = Y_0\times \mathbb A^1$ and in appropriate local coordinates $(x, y, t)$ on $Y$ we have that $\mathcal G$ is defined by $\partial_x$ and the action of $\mathcal G$ on $Y$ is given by $\zeta \cdot (x, y, t) = (\zeta x, \zeta^a y, t)$.  For $t\neq 0$, we then deduce that the Cartier index of $K_{\mathcal F_t}$ is in fact $=m$, which concludes the proof.
\end{proof}

The above proposition does not hold if ``log terminal'' is replaced by ``semi-log terminal" as the following example shows.

\begin{ej}
\label{slt_index_jumps}
    Let $X = \{xy+t = 0\} \subset \mathbb A^4$ with coordinates $(x,y,z, t)$ and let $\mathcal F$ be the foliation generated by $\partial_z$.
    Consider the $\mathbb Z/k\mathbb Z$-action on $X$ given by $\zeta \cdot (x, y, z, t) = (\zeta^{a}x, \zeta^{-a}y, \zeta z, t)$.  Let $Y$ (resp. $\mathcal G$) be the quotient of $X$ (resp. $\mathcal F$) by this action.
    Note that we have a fibration $(Y, \mathcal G) \to \mathbb A^1_t$ which is a locally stable family of foliations.
    Note that $K_{\mathcal G_0}$ has Cartier index $=k$, but for $t \neq 0$, $Y_t$ is smooth and so $K_{\mathcal G_t}$ is Cartier.
\end{ej}

\begin{proposition}
\label{prop_index_dlt_case}
        Let $f\colon (X, \Delta, \mathcal F) \rightarrow T$
    be a family of locally stable rank one integrable distributions where $\dim (X/T) = 2$ and $T$ is a smooth curve. Fix a point $0 \in T$.  Suppose that 
    \begin{enumerate}
        \item $(X, \Delta+X_0)$ is dlt; and 
  
                \item $\mathcal F_0$ is not semi-log terminal.
    \end{enumerate}
     Then $2K_{\mathcal F_0}$ is Cartier. 
\end{proposition}
\begin{proof}
The problem is local about a point $x \in X_0$, 
so we may freely replace $X$ by a neighbourhood of this point 
whenever necessary.

If $x$ is a log canonical centre of $X_0$ then by inversion of adjunction, \cite{Kawakita05}, it is a log canonical centre of $(X, X_0)$.
Since $(X, X_0)$ is dlt it follows that $X$ is smooth at $x$, 
in particular $K_{\mathcal F}$, and hence $K_{\mathcal F_0}$,
 is Cartier.

So, we may assume that $X_0$ is semi-log terminal at 
$x$.  By \cite[Theorem 4.23]{KSB88}
it follows that $x \in X_0$ is one of the following
where $\sigma$ is a primitive $r$-th root of unity and 
$a, b \in \mathbb Z$:
\begin{enumerate}
    \item a quotient of $\mathbb C^2$;
    \item a normal crossing point or a pinch point;
    \item the quotient of a normal crossing point $\{xy = 0\}$
    under the action $x \mapsto \sigma^a x, y\mapsto \sigma^b y, z \mapsto \sigma z$;
    \item the quotient of a normal crossing point $\{xy = 0\}$
    under the action $x \mapsto \sigma^a y, y \mapsto x, z \mapsto \sigma z$;
    \item the quotient of the pinch point $x^2 = zy^2$ under the action $x\mapsto \sigma^{a+1}x, y\mapsto \sigma^a y, z\mapsto \sigma^2 z$.
\end{enumerate}

In Case (1) either $\mathcal F_0$ is canonical or log canonical and not canonical.  We see by the classification in \cite[Fact I.2.4]{McQuillan08} 
if $\mathcal F_0$ has canonical singularities, then $2K_{\mathcal F_0}$
is Cartier.  By \cite[Lemma 2.11]{SS23} if $\mathcal F_0$ is log canonical 
and not canonical then $K_{\mathcal F_0}$ is Cartier. 
In either case $2K_{\mathcal F_0}$ is Cartier.

\medskip

In Case (2) if $X_0$ is a normal crossing point, then Lemma \ref {lem_picard_node} implies that $K_{\mathcal F_0}$ is Cartier.  So assume that $X_0$ is a pinch point.  Recall that $X_0$ is a $\mathbb Z/4\mathbb Z$ quotient of
 a normal crossings point.
 Indeed, the pinch point  $\{x^2=yz^2\}$ can be realised as a quotient of $\{\xi\eta = 0\} \subset \mathbb A_{(\xi, \eta, \zeta)}^3$ under the action of $\mathbb Z/4\mathbb Z$ given by
$(\xi, \eta, \zeta) \mapsto (\eta, -\xi, \tau \eta)$ where $\tau$ is a primitive fourth root of unity.
Let $q\colon \{\xi\eta = 0\} \rightarrow X_0$ be the quotient map.

The pull back of  $\mathcal F_0$
to $\{\xi\eta = 0\}$ is defined by a vector field 
$\partial$ on $\mathbb A^3$ which leaves 
$\{\xi \eta = 0\}$ invariant and such that the $\mathbb Z/4\mathbb Z$ action on $\partial$ is given by $g \cdot \partial = \chi(g)\partial$ for $g \in \mathbb Z/4\mathbb Z$ for some character $\chi\colon \mathbb Z/4\mathbb Z \to \mathbb C^*$.  

Suppose for sake of contradiction that $\partial$ is non-singular.  Then, up to scaling by a unit, $\partial = \partial_\zeta+\delta$ where $\delta \in \mathfrak mT_{\mathbb C^3}$ and defines a foliation 
which descends to a foliation $\mathcal F_0$ on $X_0$
such that $K_{\mathcal F_0}$ is not Cartier but $4K_{\mathcal F_0}$ is Cartier.  However, it is easy to check (arguing along the lines of Claim \ref{cl_lt})
that in this case $\mathcal F_0$ is semi-log terminal, 
contrary to assumption.

Since $\mathcal F_0$ is semi-log canonical, and $q$ is quasi-\'etale, it follows that $q^{-1}\mathcal F_0$ is semi-log canonical.  So, by Proposition \ref{prop_lc_nonnilp} and Lemma \ref{lem_nonnilp_normal} $\partial$ has non-nilpotent linear part. The space of linear vector fields on $\mathbb A^3$ which vanish at the origin and leave $\{\xi\eta = 0\}$ invariant is spanned by
$\xi\partial_\xi-\eta\partial_\eta, \xi\partial_\xi+\eta\partial_\eta$ and $\zeta\partial_\zeta$, each of which are eigenvectors of the $\mathbb Z/4\mathbb Z$-action.
Both $\zeta\partial_\zeta$ and $\xi\partial_\xi+\eta\partial_\eta$
are invariant by the $\mathbb Z/4\mathbb Z$ action, and 
$(\xi\partial_\xi-\eta\partial_\eta)\otimes( \xi\partial_\xi-\eta\partial_\eta)$ is invariant by the $\mathbb Z/4\mathbb Z$-action.

Since $\partial$ is an eigenvector of the $\mathbb Z/4\mathbb Z$ action, to compute the eigenvalue of $\partial$ it suffices to compute the eigenvalue of the linear part of $\partial$ under the $\mathbb Z/4\mathbb Z$ action.  By the previous paragraph we know that this eigenvalue is $\pm 1$ so it follows that 
 $\partial\otimes \partial$ is invariant by the $\mathbb Z/4\mathbb Z$ action and so $2K_{\mathcal F_0}$ is Cartier.

\medskip

 In Case (3) we see that $X_0$ has two irreducible components, each 
 of which is a quotient singularity.  It follows that if a $\mathbb Q$-Cartier divisor $D$ is Cartier on each component of $X_0$, then in fact $D$ is Cartier.  By case (1), we know that 
 $2K_{\mathcal F_0}$ is Cartier on each irreducible component, and we may apply Lemma \ref{lem_picard_node} to conclude that $2K_{\mathcal F_0}$ is Cartier.
\medskip

 In Case (4), let $q\colon \{xy = 0\} \rightarrow X_0$
 be the quotient map, and let $\partial$ be a vector field which 
 generates $q^{-1}\mathcal F_0$ and which is an eigenvector of the $\mathbb Z/r\mathbb Z$ action.
 Since $\mathcal F_0$ is semi-log canonical and not semi-log terminal, the same holds true for $q^{-1}\mathcal F_0$.  Thus, $\partial$ must vanish at $(0, 0, 0)$ and have non-nilpotent linear part, see Lemma \ref{lem_nonnilp_normal}.
 The action of $\mathbb Z/r\mathbb Z$
 on the linear parts of vector fields which vanish at the origin
 part has three eigenvectors:
 $x\partial_x - y\partial_y, x\partial_x+y\partial_y$ and $z\partial_z$.  These eigenvectors 
 all have eigenvalue $\pm 1$
 and so, 
 as in Case (2), we conclude that  $\partial\otimes \partial$
 is invariant by the $\mathbb Z/r\mathbb Z$
 action and so $2K_{\mathcal F_0}$
 is Cartier.

\medskip

We now consider Case (5).  Let $Y$ be the pinch point, let $r\colon Y \rightarrow X_0$ be the quotient morphism and 
let $\mathcal G = r^{-1}\mathcal F_0$.
There are two cases: either $K_{\mathcal G}$ is Cartier or it is not.  If $K_{\mathcal G}$ is Cartier then it is generated by a vector field $\partial$ which leaves $\{x^2+zy^2 = 0\}$ invariant.
The space of linear parts of such vector fields is spanned by 
$y\partial_y-2z\partial_z$ and $x\partial_x+y\partial_y$.
Both of these vector fields are invariant by the 
$\mathbb Z/r\mathbb Z$-action, and hence 
$\partial$ is invariant by this action.
It follows that $K_{\mathcal F_0}$ is Cartier.

Now suppose that $K_{\mathcal G}$ is not Cartier.
As noted earlier, the pinch point  $Y = \{x^2=zy^2\}$ can be realised as a quotient of $\{\xi \eta = 0\} \subset \mathbb A_{(\xi, \eta, \zeta)}^3$ under the action of $\mathbb Z/4\mathbb Z$ given by
$(\xi, \eta, \zeta) \mapsto (\eta, -\xi, \tau \zeta)$ where $\tau$ is a primitive fourth root of unity.
The invariant functions under this action are
$x = \zeta^2(\xi^2-\eta^2), z = \zeta^4$ and $y = \xi^2+\eta^2$.

The pull back of  $\mathcal G$
to $\{\xi\eta = 0\}$ is defined by a vector field 
$\partial$ on $\mathbb A^3$ which leaves 
$\{\xi \eta = 0\}$ invariant.  
As in Case (2)
we see that the pull-back of $\mathcal G$ is defined by a vector field with linear part equal to either 
$\xi\partial_\xi-\eta\partial_\eta, \xi\partial_\xi+\eta\partial_\eta$ or $\zeta\partial_\zeta$.
Both $\zeta\partial_\zeta$ and $\xi\partial_\xi+\eta\partial_\eta$
are invariant by the $\mathbb Z/4\mathbb Z$ action, 
so they will induce foliations on $Y$ such that $K_{\mathcal G}$ is Cartier, in which case by our above analysis $K_{\mathcal F_0}$ is Cartier

It remains to consider the case 
where the linear part of $\partial$ is 
$\xi\partial_\xi-\eta\partial_\eta$. 
We claim that $\partial \otimes \partial$ descends to a section of $T^{\otimes 2}_{\mathcal G}$ which is invariant by the $\mathbb Z/r\mathbb Z$ action on the pinch point.  As in the previous cases, to verify this claim we may freely replace $\partial$ by its linear part, and $r^{-1}\mathcal G$ by foliation generated by the linear part of $\partial$.

The local generators of $T_{\mathcal G}$ are given by  
$\mathbb Z/4\mathbb Z$-invariant sections of $T_{r^{-1}\mathcal G}$.  This submodule is generated by $\partial_1 = \zeta^2(\xi\partial_\xi-\eta\partial_\eta)$ and $\partial_2 = (\xi^2-\eta^2)(\xi\partial_\xi-\eta\partial_\eta)$.
We can calculate that 
\begin{align*}
&\partial_1(x) = \partial_1(\zeta^2(\xi^2-\eta^2)) = 2\zeta^4(\xi^2+\eta^2) = 2zy \\
&\partial_1(y) = \partial_1(\xi^2+\eta^2) = 2\zeta^2(\xi^2-\eta^2) = 2x \\
&\partial_1(z) = \partial_1(\zeta^4)  = 0
\end{align*}
and so $\partial_1$ descends to the vector field $\delta_1 = 2x\partial_y + 2zy\partial_x$.
A similar calculation shows that  $\partial_2$ descends to the vector field 
$\delta_2 = \frac{x^2}{z}\partial_y+yx\partial_x$.

Recall from above that $(\xi\partial_\xi-\eta\partial_\eta)\otimes (\xi\partial_\xi-\eta\partial_\eta)$ descends to a generator of $T_{\mathcal G}^{\otimes 2}$ and we may write
\begin{align*}(\xi\partial_\xi-\eta\partial_\eta)\otimes (\xi\partial_\xi-\eta\partial_\eta) = \frac{1}{\zeta^2(\xi^2-\eta^2)}\partial_1\otimes \partial_2.\end{align*} 
Since $x = \zeta^2(\xi^2-\eta^2)$ we deduce that $\frac{1}{x}\delta_1\otimes\delta_2$ is a generator of  $T_{\mathcal G}^{\otimes 2}$.
The action of $\mathbb Z/r\mathbb Z$ 
sends $\delta_1 \mapsto \sigma \delta_1$ and $\delta_2 \mapsto \sigma^a \delta_2$.  In particular, $\frac{1}{x}\delta_1\otimes\delta_2$ is invariant by the $\mathbb Z/r\mathbb Z$-action on the pinch point.
It follows that $\frac{1}{x}\delta_1\otimes\delta_2$
descends to a generator of $\mathcal O(2K_{\mathcal F_0})$
and so $2K_{\mathcal F_0}$ is Cartier, as required.
\end{proof}

\begin{theorem}
\label{thm_dlt_mod_fam}
       Let $f\colon (X, \Delta, \mathcal F) \rightarrow T$
    be a locally stable family of rank one integrable distributions where $X$ is normal,  $\dim (X/T) = 2$, the generic fibre $X_\eta$ is klt and $T$ is a smooth curve. Fix a point $0 \in T$.

    Then there exists a finite morphism $r\colon S \rightarrow T$ birational modification $p\colon X' \rightarrow X\times_TS$
    such that if $\pi\colon X' \rightarrow X$ is the natural map and 
    $\mathcal F' = \pi^{-1}\mathcal F$
is the pulled back foliation then,     \begin{enumerate}
\item $(X', \mathcal F') \rightarrow S$ is a locally stable family of integrable distributions;
        \item $\pi^*K_{\mathcal F} = K_{\mathcal F'}$;
        \item for any $n \in \mathbb Z_{>0}$, if $nK_{\mathcal F'}$ is Cartier, then $nK_{\mathcal F}$ is Cartier; and 
        \item $(X', X'_0)$ is dlt.
    \end{enumerate}
\end{theorem}
\begin{proof}
Let $r\colon S \rightarrow T$ be a finite morphism.
By Proposition \ref{prop_base_change}, 
$(X_S, \Delta_S, \mathcal F_S) \rightarrow S$ is a 
locally stable family of foliations.
It is clear that the Cartier index of $K_{\mathcal F_{S, s}}$ is equal to the Cartier index  of $K_{\mathcal F_{r(s)}}$ for all 
points $s \in S$.
Thus we are free to replace $(X, \Delta, \mathcal F)$
by $(X_S, \Delta_S, \mathcal F_S)$ and
so by Theorem \ref{thm_locally_stable_reduction}
there exists a  birational modification $\pi\colon Y \rightarrow X$ such that 
$(Y, \pi_*^{-1}\Delta+Y_0+D)$ is a simple normal crossings pair where $D = {\rm exc}(\pi)$ and 
$(\mathcal G, \Gamma+E)$ has canonical singularities, where
$(\mathcal G, \Gamma)$ is the associated foliated pair to $\pi^{-1}\mathcal F$
and $E = {\rm Exc}(\pi)_{{\rm non-inv}}$.

Let $\phi\colon Y \dashrightarrow X'$ be a run of the  $K_{\mathcal G}+\Gamma+E$-MMP over $X$ which is guaranteed to exist by Theorem \ref{thm_mmp}, let $p\colon X' \rightarrow X$ be the induced morphism , let $\mathcal G' \coloneqq p_*\mathcal G$,
$\Gamma' = p_*\Gamma$, $E' = p_*E$
and let $\mathcal F'$ be the integrable distribution associated to $(\mathcal G', \Gamma'+E')$.
Since $\mathcal F'$ is log canonical, Proposition \ref{prop_family_criterion} implies that (1) holds.

Since $\mathcal F$ has log canonical singularities we see that 
$K_{\mathcal G}+\Gamma+E = \pi^*K_{\mathcal F}+F$ where $F \ge 0$.
Since $K_{\mathcal G'}+\Gamma'+E'$ is nef over $X$, the negativity lemma implies that $\phi_*E = 0$ and so $K_{\mathcal F'} = p^*K_{\mathcal F}$, thus (2) holds.

By well-formedness we see that $\Gamma+E \ge (\pi_*^{-1}\Delta+{\rm Exc}(\pi)+Y_0)_{{\rm non-inv}}$ and so by Lemma \ref{lem_KFmmp_singularities} we see that  $(X', p_*^{-1}\Delta+{\rm Exc}(p)+X'_0)$ is dlt, 
in particular $(X', X'_0)$ is dlt 
and so (4) holds. 

We now show that (3) holds. 
We first make two general observations.
\begin{enumerate}
    \item \label{no_lc_places_central} Every log canonical centre of $(X, \Delta)$ 
dominates $T$.
Indeed, since $(X, \Delta) \to T$ is locally stable,  $(X, \Delta+X_0)$
is log canonical, and so any non-klt centre of $(X, \Delta)$ is not contained in 
$X_0$. 

\item \label{bpf_index_doesnt_change} Let $\rho\colon U \to V$ be a projective morphism between normal quasi-projective varieties, let $B \ge 0$
be a $\mathbb Q$-divisor such that $(U, B)$ is klt and $-(K_U+B)$ is $\rho$-ample.  Then, an application of the base point free theorem, see \cite[Theorem 3.7.(4)]{KM98}, implies that if $M$ is a Cartier divisor on $U$ such that $M$ is $\rho$-numerically trivial, then there exists a Cartier divisor $N$ on $V$ such that $M \sim \rho^*N$.
\end{enumerate}

Since $(X', \phi_*D)$ is dlt and $X'$ is $\mathbb Q$-factorial, 
we may run a $K_{X'}+\phi_*D - \eta\phi_*\pi_*^{-1}\Delta$-MMP for $0<\eta \ll 1$ over $X$, call it 
$\psi\colon X' \dashrightarrow X''$.  Each step of this MMP is $K_{\mathcal F'}$ trivial and so by Observation  (\ref{bpf_index_doesnt_change}) above we see that the Cartier index of $K_{\mathcal F'}$ does not increase after each step of the MMP.  Thus, if $nK_{\mathcal F'}$ is Cartier, then $n\phi_*K_{\mathcal F'}$ is Cartier.

Since the generic fibre of $X \to T$ is klt, an application of the negativity lemma shows that $X'' \to X$ is an isomorphism in a neighbourhood of the generic fibre.  Another application of negativity lemma shows that the exceptional divisors of $X'' \to X$ are all log canonical places of $(X, \Delta)$.  Observation (\ref{no_lc_places_central}) above therefore implies that there are therefore no exceptional divisors contained in the central fibre of $X'' \to T$ and so the natural morphism $\mu\colon X'' \to X$ contracts no divisors, i.e., $\mu$ is a small modification.  
So, if we set $\Delta'' = \mu_*^{-1}\Delta$, then $K_{X''}+\Delta'' = \mu^*(K_X+\Delta)$.  Again, by Observation (\ref{no_lc_places_central}), no log canonical centre of $(X'', \Delta'')$ is contained in the fibre over $0$ we may find a $\mathbb Q$-divisor $B \ge 0$ so that $(X'', \Delta''+B)$ is log canonical and $-(K_{X''}+\Delta''+B)$ is $\mu$-ample.  Since the generic fibre of $X'' \to T$ is klt, by perturbing $\Delta''+B$ slightly we may find a $\mathbb Q$-divisor $\Theta \ge 0$ so that $(X'', \Theta)$ is klt and $-(K_{X''}+\Theta)$ is $\mu$-ample.
A final application of Observation (\ref{bpf_index_doesnt_change}) above
shows that if $n\phi_*K_{\mathcal F'}$ is Cartier then 
$\mu_*n\phi_*K_{\mathcal F'} = nK_{\mathcal F}$ is Cartier.
\end{proof}

\begin{theorem}
\label{thm_index_is_stable}
Fix a positive integer $N$ such that $4|N$.
    Let $f\colon (x \in X, \Delta, \mathcal F) \rightarrow T$
    be a germ of a locally stable family of rank one integrable distributions where $\dim (X/T) = 2$ and $T$ is a smooth curve.  Fix a point $0 \in T$.  
    Suppose that 
    \begin{itemize}
        \item $NK_{\mathcal F_t}$ is Cartier for all $t \in T\setminus \{0\}$; and
        \item either $\mathcal F_0$ is not semi-log terminal or $X_0$ is normal.
    \end{itemize}

     Then $NK_{\mathcal F_0}$ is Cartier.   
\end{theorem}

\begin{proof}
If $X_0$ is normal and $\mathcal F_0$ is log terminal, we may apply Proposition \ref{prop_lt_index} to conclude.  So, it only remains to consider the case where $\mathcal F_0$ is not semi-log terminal.

By Lemma \ref{lem_index_on_normalisation} it suffices to show that if $(X^n, \Delta^n, \mathcal F^n)\to S$ is the 
normalised family of integrable distributions, cf. Lemma \ref{lem_norm_loc_stable},
then $2K_{\mathcal F^n}$ is Cartier.
Let $x \in X^n$ be a closed point.  Up to replacing $X^n$ by a neighbourhood of $x \in X^n$, we may form the index one cover associated to $K_{\mathcal F^n}$, denote it $s\colon W \to X^n$, with Galois group $G = \mathbb Z/m\mathbb Z$.  Let $\partial$ be a vector field which generates $s^{-1}\mathcal F^n$ in a neighbourhood of $s^{-1}(x)$.

We argue in cases based on whether or not $\partial$ is a radial vector field.

If the linear part of $\partial$ at $s^{-1}(x)$ is radial, i.e., up to scaling 
$\partial = \sum n_ix_i\partial_{x_i}$ where $n_i \in \mathbb Z_{>0}$, then
the proof of \cite[Lemma 2.11]{SS23} applies to show that $\partial$ is $G$-invariant, in particular, $K_{\mathcal F^n}$ is Cartier.

If the linear part of $\partial$ is not radial \cite[Fact III.i.3]{MP13}
implies that $\mathcal F^n$ has canonical singularities, and so the restriction 
of $\mathcal F^n$ to the generic fibre $X^n_\eta$ has canonical singularities, see Remark \ref{rem_very_general_canonical}.
By \cite[Fact I.2.4]{McQuillan08} we deduce that $X^n_\eta$ has quotient singularities, in particular it is klt.
    We may then apply Theorem \ref{thm_dlt_mod_fam} so that
     we may freely assume that $(X^n, X^n_0)$ is dlt
     in which case Proposition 
     \ref{prop_index_dlt_case} implies that $2K_{\mathcal F^n}$ is Cartier, as required.
\end{proof}

\begin{lemma}
\label{lem_index_on_normalisation}
Let $(X, \Delta) \to S$ be a locally stable family where $\dim (X/S) = 2$ and let $K$ be a relative Mumford divisor on $X$.
Let $n\colon X^n \to X$ be the normalisation.

Suppose that for some positive integer $m$ we have
\begin{itemize}
\item $mn^*K$ is Cartier; and 
\item $mK_{\eta}$ is Cartier where $K_\eta$ is the restriction of $K$ to the generic fibre of $X \to S$.
\end{itemize}
Then $2mK$ is Cartier.
\end{lemma}
\begin{proof}
Let $r\colon \tilde{X} \rightarrow X$ be the natural double cover 
    of $X$ described in \cite[\S 5.23]{Kollar13}.  We remark
    that by the construction given there we see that $\tilde{X}$
    is $S_2$ and $r$ is quasi-\'etale and Galois with Galois group $\mathbb Z/2\mathbb Z$.
    (Note that if $X$ is normal, then $\tilde{X}$ is a disjoint union of two copies of $X$.)
    In particular, $(\tilde{X}, r^*\Delta) \to S$ is locally stable.  We next note that the irreducible components of $\tilde{X}$ are all normal.
    Indeed, by construction the irreducible components are $R_1$.  By the previous paragraph they are  $S_2$, hence they are normal.
    Let $Y_1, \dots, Y_k$ be the the irreducible components of $\tilde{X}$.  
    
    If $\tilde{n}\colon \tilde{X}^n \to \tilde{X}$ denotes the normalisation, then observe that $m\tilde{n}^*r^*K$ is Cartier.
    If we can show that $mr^*K$ is Cartier, then since $r$ is Galois with Galois group $\mathbb Z/2\mathbb Z$ we can conclude 
    that $2mK$ is Cartier.

Since $K$ is a relative Mumford divisor and since $mK_\eta$ is Cartier, it follows that $mK$ is Cartier in a neighbourhood of any codimension $\le 2$ point of $X$.  So, suffices to show that $mr^*K$ is Cartier in a neighbourhood of
an arbitrary closed point $x \in \tilde{X}$.
Moreover, up to shrinking $X$ we may assume that $mr^*K|_{Y_i} \sim 0$ for any $i \in \{1, \dots, k\}$.

For any $j \in \{1, \dots, k\}$ set $B_j = Y_j \cap (\bigcup_{i =1}^{j-1}Y_i)$ considered as a divisor on $Y_j$.
Let $\sqcup_{i = 1}^k(Y_i, \Delta_{Y_i}) \to S$ denote the normalisation of the locally stable family $(\tilde{X}, r^*\Delta)\to S$.
Since $(Y_j, \Delta_{Y_j}) \rightarrow T$ is locally
stable, it follows that for any $j \in \{1, \dots, k\}$ if $y \in Y_j$ is the pre-image of $x$ in $Y_j$
is not a log canonical centre of $(Y_j, \Delta_{Y_j})$.
Since $B_j \leq \Delta_{Y_j}$ we then apply \cite[Theorem 7.20]{Kollar13}
to see that ${\rm depth}_y\mathcal O_{Y_j}(-B_j) \ge 3$ and so 
\begin{align*}H^1(Y_j \setminus \{y\}, \mathcal O_{Y_j}(-B_j)|_{Y_j \setminus \{y\}}) = 0.\end{align*}
Note that away from $y$, $\mathcal O_{Y_j}(-B_j)|_{Y_j\setminus y}$ is precisely the restriction of the kernel of \begin{align*}\mathcal O(mr^*K)|_{\bigcup_{i  = 1}^j Y_i} \to O(mr^*K)|_{\bigcup_{i  = 1}^{j-1}Y_i}\end{align*} and so we may then use \cite[Proposition 5.21.(3)]{Kollar13} applied to $mr^*K$ to conclude that $\mathcal O(mr^*K) \cong \mathcal O_{\tilde{X}}$, in particular, $mr^*K$
is Cartier, as required.
\end{proof}

\section{Properness}

In this section we verify that our moduli functor satisfies the valuative criterion of properness for DVRs which are essentially of finite type over $\mathbb C$.

\begin{theorem}
\label{thm_compactify_to_family}
Let $S$ be a smooth curve, let $s \in S$ be a closed point and set $S^\circ = S \setminus \{s\}$.  Let $f\colon (X^\circ, \Delta^\circ, \mathcal F^\circ) \rightarrow S^\circ$ be a stable
family of rank one integrable distributions where $\dim (X^\circ/S^\circ) = 2$ and $X_\eta$ is normal.

Then there exists a finite morphism $S' \rightarrow S$
and a family of integrable distributions $(X', \Delta', \mathcal F') \rightarrow S'$ such that 
\begin{enumerate}
\item $(X', \Delta', \mathcal F')\rightarrow S'$ is stable; and 

\item $(X', \Delta', \mathcal F')\times_{S'}(S')^\circ \cong (X^\circ, \Delta^\circ, \mathcal F^\circ) \times_{S^\circ}(S')^\circ$.
\end{enumerate}

\end{theorem}
\begin{proof}

Let $(\tilde{\mathcal F}^\circ, \Gamma^\circ)$ be the associated foliated pair
to the integrable distribution $\mathcal F^\circ$.

By the well-formedness assumption in the definition of stability, (S5), we have 
that $\Delta^\circ_{{\rm non-inv}} \leq \Gamma^\circ$.

\medskip

{\bf Step 1: }{\it  In this step we resolve and compactify our family to a
locally stable model.}

Let $(X, \Delta, \tilde{\mathcal F}) \rightarrow S$ be an arbitrary extension of $(X^\circ, \tilde{\mathcal F}^\circ, \Delta^\circ)$
to a family of log integrable distributions over $S$ (this is always possible by Lemma \ref{lem_unique_extension})
and let $\Gamma$ be the closure of $\Gamma^\circ$.
By Theorem \ref{thm_locally_stable_reduction},
up to replacing $S$ by a finite cover, we may assume that $(\tilde{\mathcal F}, \Gamma)$ admits a locally stable
resolution of singularities.
Call this resolution $p\colon Y \rightarrow X$, let $\mathcal G = p^{-1}\tilde{\mathcal F}$, let $\Delta_Y = p_*^{-1}\Delta+\sum E_i $ and let $\Gamma_Y = p_*^{-1}\Gamma+\sum \varepsilon(E_i)E_i$
where $E_i$ are all the $p$-exceptional divisors.

\medskip

{\bf Step 2: }{
\it In this step we run a $K_{\mathcal G}+\Gamma_Y$-MMP followed by a partial $K_{Y}+\Delta_Y$-MMP.}

By Theorem \ref{thm_mmp} we may run a $K_{\mathcal G}+\Gamma_Y$-MMP over $X$, 
call it $\phi\colon Y \dashrightarrow Y_1$, set $\mathcal G_1 = \phi_*\mathcal G$, 
$\Delta_1 = \phi_*\Delta_Y$ and $\Gamma_1 = \phi_*\Gamma_Y$.
 
By Lemma \ref{lem_KFmmp_singularities} $(Y_1, \Delta_1+Y_{1, s})$ is log canonical
for any $s \in S$. So we may run 
a partial $K_{Y_1}+\Delta_1$-MMP over $X$ only contracting/flipping 
$K_{\mathcal G_1}+\Gamma_1$-trivial extremal rays.
Note that this MMP will also be a $K_{Y_1}+\Delta_1+Y_{1, s}$-MMP
for any $s \in S$.

Call this MMP $\psi\colon Y_1 \dashrightarrow Y_2$ and set $\mathcal G_2 = \psi_*\mathcal G_1$, $\Delta_2 = \psi_*\Delta_1$ and $\Gamma_2 = \psi_*\Delta_2$.
Note that $(\mathcal G_2, \Gamma_2)$ is log canonical and $(Y_2, \Delta_2+Y_{2, s})$
is log canonical for all $s \in S$, 
in particular $(Y_2, \Delta_2) \rightarrow S$ is a locally stable, cf.  \cite[Definition-Theorem 2.3]{modbook}.

Observe that since $\phi$ and $\psi$ are two (partial) MMPs they preserve $\mathbb Q$-factorial singularities, and so $Y_2$ is $\mathbb Q$-factorial.

\medskip

{\bf Step 3: }{\it In this step we contract all those curves which are 
$K_{\mathcal G_2}+\Gamma_2$ and $K_{Y_2}+\Delta_2$-trivial.}

Now, $(K_{\mathcal G_2}+\Gamma_2)+t(K_{Y_2}+\Delta_2)$ is relatively big and nef for all $0<t \ll1$
and $K_{\mathcal G_2}+\Gamma_2$ is nef and $\mathbb Q$-Cartier.
Taking ${\bf M}$ to be the Cartier closure of  $\frac{1}{t}(K_{\mathcal G_2}+\Gamma_2)$
we see that the generalised pair $(Y_2, \Delta_2+{\bf M})$ satisfies
\begin{align*}K_{Y_2}+\Delta_2+{\bf M} \sim_{\mathbb Q} \frac{1}{t}((K_{\mathcal G_2}+\Gamma_2)+t(K_{Y_2}+\Delta_2)).\end{align*}
Since $(Y_2, \Delta_2)$ is log canonical 
it follows that $(Y_2, \Delta_2+{\bf M})$
is generalised log canonical.
Since $(Y_2, \Delta_2) \rightarrow S$ is locally stable it follows that if $Z$ is a log canonical centre 
of $(Y_2, \Delta_2)$, then $Z$ dominates $S$. In particular, any generalised log canonical centre of $(Y_2, \Delta_2+{\bf M})$ dominates $S$.

Next, by construction we have a morphism between generic fibres 
$\rho\colon Y_{2, \eta} \rightarrow X_{\eta}$ and by the negativity lemma, \cite[Lemma 3.38]{KM98}  we have \begin{align*}(K_{\mathcal G_2}+\Gamma_2)+t(K_{Y_2}+\Delta_2) = \rho^*((K_{\mathcal F}+\Gamma)+t(K_X+\Delta)),\end{align*} 
in particular, $(K_{\mathcal G_2}+\Gamma_2)+t(K_{Y_2}+\Delta_2)\vert_{Y_{2, \eta}}$ is semi-ample. 
It follows that $(Y, \Delta_2+{\bf M})$ 
is a good minimal model on a neighbourhood of the generic fibre. We may then apply 
\cite[Theorem 1.3]{MR4596359} (cf. also \cite[Theorem 5.1]{MR4596359})
to conclude that 
$K_Y+\Delta_2+{\bf M}$, and hence $(K_{\mathcal G_2}+\Gamma_2)+t(K_{Y_2}+\Delta_2)$, is semi-ample over $S$.

Let $r\colon Y_2 \rightarrow Y_3$ be the associated contraction for some $0<t \ll 1$.
Note that $r$ only contracts curves which are $K_{\mathcal G_2}+\Gamma_2$ and $K_{Y_2}+\Delta_2$-trivial.
Let $\mathcal G_3 = r_*\mathcal G_2$, 
$\Delta_3 = r_*\Delta_2$ and $\Gamma_3 = r_*\Gamma_2$.

By construction $K_{\mathcal G_3}+\Gamma_3+t(K_{X_3}+\Delta_3)$
is $\mathbb Q$-Cartier. Since $K_{Y_2}+\Delta_2$ is $r$-numerically trivial it follows that $(Y_2, \Delta_2)$ is a good minimal model over $Y_3$ and so $K_{Y_3}+\Delta_3$
is in fact $\mathbb Q$-Cartier as well and $(Y_3, \Delta_3+Y_{3, s})$
is log canonical for all $s \in S$ (hence $(Y_3, \Delta_3) \to S$ is locally stable, \cite[Definition-Theorem 2.3]{modbook}).
From this we also deduce that $K_{\mathcal G_3}+\Gamma_3$ is $\mathbb Q$-Cartier.

Take $(X', \Delta', \mathcal F') = (Y_3, \Delta_3, \mathcal F_3)$ where $\mathcal F_3$ is the integrable distribution associated to $(\mathcal G_3, \Gamma_3)$.

We apply Proposition \ref{prop_family_criterion}
to conclude that $(Y_3, \Delta_3, \mathcal F_3) \rightarrow S$
is a locally stable family 
of integrable distributions.
By construction $K_{\mathcal G_3}+\Gamma_3+t(K_{X_3}+\Delta_3)$
is ample for all $0<t \ll 1$ and so in fact 
it is a stable family of integrable distributions.
\end{proof}

\begin{lemma}
\label{lem_unique_extension}
    Let $X$ be an $S_2$ scheme, let $X_0 \subset X$ be an open subset such that
    $(X \setminus X_0) \cap \sing X$ is 
    codimension at least 2 in $X$.

Let $E$ be a coherent sheaf on $X$, let $L_0$ be a divisorial sheaf on $X_0$ and suppose that we have a morphism 
\begin{align*}\phi_0\colon E\vert_{X_0} \rightarrow L_0.\end{align*}
Then there exists a unique divisorial sheaf $L$ on $X$
and a unique morphism $\phi\colon E \rightarrow L$
such that 
\begin{enumerate}
    \item there exists an isomorphism $\psi\colon L\vert_{X_0} \rightarrow L_0$;
    \item $\psi\circ (\phi\vert_{X_0}) = \phi_0$; and 
    \item $\phi$ is surjective at all the codimension one points of $X$ contained in $X \setminus X_0$.
\end{enumerate}
\end{lemma}
\begin{proof}
This follows by recalling that a divisorial sheaf on any open subset of $X$ can be extended to a divisorial sheaf on all of $X$, and moreover, in the case where the complement is codimension at least 2 in $X$, then the extension is unique.  Moreover, this extension can be chosen in a way so that the morphism $\phi_0$
extends to a morphism $\phi\colon E \to L$.  Perhaps replacing $L$ by $L(-D)$ where $D$ is supported on $X \setminus X_0$ we may assume also that $\phi$ is surjective at all the codimension one points of $X$ contained in $X \setminus X_0$.
\end{proof}

\section{The moduli stack}
We are now ready to prove the representability results stated in the introduction. 
We first consider an embedded version of our moduli problem.

\begin{lemma}  \label{emb}
Fix $N \in \mathbb Z_{>0}$ and $v \in \mathbb R_{>0}$.  Then
    there exist natural numbers $m_0,n_0,k_0>0$ such that for every family $f: (X, \Delta, \FF)\to S$ in $\mathcal M^{2,1}_{N, v}(S)$, the sheaf $\OO_X(m_0 K_\FF + n_0 (K_{X/S}+\Delta))$ is $f$-very ample and $f_*\OO_X(m_0 K_\FF + n_0 (K_{X/S}+\Delta))$ is locally free of rank $k_0+1$.
\end{lemma}
\begin{proof}
Corollary \ref{boundedness.slc.triples.cor} guarantees the existence of a family $f\colon (X_T,\Delta_T, \FF_T)\to T$ bounding $\mathcal M^{2,1}_{N, v}$. By Corollary \ref{cor_stab_rep}, up to replacing $T$ by a locally closed partial decomposition,  we can freely assume that $(X_T, \Delta_T, \FF_T) \to T$ is  stable of index $=N$ and hence by Proposition \ref{prop_threshold} the divisor $5NK_{\mathcal F_T}+(K_{X_T/T}+\Delta_T)$
is $f$-ample. Since $K_{X_T/T}+\Delta_T$ is $\mathbb Q$-Cartier 
there exists an integer $m_1>0$ (depending only on $f\colon (X_T, \Delta_T, \FF_T) \to T$) 
such that $m_1(5NK_{\mathcal F_T}+(K_{X_T/T}+\Delta_T)$ is Cartier. Moreover, by \cite[Theorem 1.7.6]{Lazarsfeld04a} there exists an integer $m_2>0$ (depending only on $f\colon (X_T, \Delta_T, \FF_T) \to T$) 
such that  $\LL\simeq \OO_X(m_2m_1(5N K_{\FF_T}+(K_{X_T/T}+\Delta_T)))$ is $f$-very ample and $R^if_*(X,\LL^{\otimes \ell})=0$ for every $i,\ell>0$ -- implying also that $f_*\LL$ is locally free.
In particular, setting $k=rank(f_*\LL)-1$ and $(m,n)=(5Nm_1m_2,m_1m_2)$ we can conclude that for every log-foliated surface $(X_0,\Delta_0,\FF_0)$  the sheaf $\LL_0=\OO_X(m_0 K_{\FF_0}+n_0(K_{X_0}+\Delta_0)))$ is a very ample line bundle with space of global sections of dimension $k+1$ and satisfying  $H^i(X_0,\LL_0^{\otimes \ell})=0$ for every $i,m>0$.

If now $f\colon (X,\Delta,\FF)\to S$ is any family in $\mathcal M^{2,1}_{N, v}(S)$, then, being a flat family of line bundles, the sheaf $\LL'=\OO_X(m_0K_\FF +n_0(K_{X/S}+\Delta))$ is itself a line bundle. Since both properties are local on the base, we may freely assume that $S$ is the spectrum of a local ring. The Cohomology and Base Change Theorem and the vanishing of $H^i(X_s,\LL_s^m)=0$ at the closed point $s$ now imply that $f_*\LL'$ is free of rank $k+1$. In particular,  
we can extend any embedding $X_s\hookrightarrow \PP^k_\CC$ defined by $m_0K_{\FF_s}+n_0(K_{X_s}+\Delta_s)$ to some $X\hookrightarrow \PP^k_S$ induced by a basis of $f_*\LL'$ and the Lemma follows.
\end{proof}

By Corollary \ref{boundedness.slc.triples.cor} there is a finite collection of polynomials $p_1(l), \dots, p_r(l) \in \mathbb Q[l]$ such for any $(X, \Delta, \mathcal F) \in \mathcal M^{2, 1}_{N, v}(\mathbb C)$ we have that the Hilbert polynomial $\chi(X, \mathcal O(l(m_0K_{\mathcal F}+n_0(K_X+\Delta)))) = p_i(l)$ for some $i \in \{1, \dots, r\}$
(here $m_0, n_0$ are the positive integers guaranteed to exist by Lemma \ref{emb}). Distinct Hilbert polynomials give disjoint components of the moduli space, so to construct the moduli space, it suffices to fix a polynomial $p(l) \in \mathbb Q[l]$ and consider only those stable foliated surfaces with 
$\chi(X, \mathcal O(l(m_0K_{\mathcal F}+n_0(K_X+\Delta)))) = p(l)$.  For $p(l) \in \mathbb Q(l)$ we denote 
by $\mathcal M^{2, 1}_{N, v, p}$ the subfunctor of $\mathcal M^{2, 1}_{N, v}$ whose $\mathbb C$-valued points are 
triples $(X, \Delta, \mathcal F)$ such that $\chi(X, \mathcal O(l(m_0K_{\mathcal F}+n_0(K_X+\Delta)))) = p(l)$. 
    Similarly, the Hilbert polynomial of both $K_X$ and $\Delta$ varies in a finite set $F$, provided that $(X,\Delta,\FF)\in \mathcal M^{2, 1}_{N, v}(\CC)$, and  different choices of these polynomials also correspond to disjoint components in moduli.

Let $H$ be the Hilbert scheme of closed subschemes of $\PP^{k_0}$ with Hilbert polynomial $p(l)$ and let $X_H \to H$ be the universal family.
Let $M_r$ be the subscheme of the relative Hilbert scheme ${\rm Hilb}(X_H/H)$ parametrising flat families of relative Mumford divisors with fixed Hilbert polynomial $q\in F$ and its associated universal family
$$ (X_{M_q},\Delta_{M_q})\to M_q.$$
We will first construct a space encoding all possible families of integrable distributions along $(X_{M_q},\Delta_{M_q})$. 
Following \cite[Chapter 9]{modbook}, let us consider the algebraic space 
$${\rm QHusk}_r(\Omega^1_{X_{M_q}/ {M_q}})\to M,$$ which is a fine moduli space for the functor of quotient husks for $\Omega^1_{X_{M_q}/{M_q}}$ with full support and Hilbert polynomial $r\in F$. By construction, a morphism $S\to {\rm QHusk}_r(\Omega^1_{X_{M_q}/{M_q}})$ corresponds to a canonically defined triple 
$$(X_T,\Omega^1_{X_T/T}\to \Q ,\Delta_T)$$ 
where $\Q$ is a flat sheaf with Hilbert polynomial $r$ and $\Omega^1_{X_t}\to \Q_t$ is surjective (and non-zero) at all generic points of $X_t$.

\begin{lemma}
    There is an algebraic space ${\rm Fol}_{q,r}(X_{H_p}/H_p)$ of finite type and a locally closed immersion 
    $${\rm Fol}_{q,r}(X_{H_p}/H_p)\hookrightarrow {\rm QHusk}_r(\Omega^1_{X_{M_q}/{M_q}})$$ 
    such that a morphism $S\to{\rm QHusk}_r(\Omega^1_{X_{M_q}/{M_q}})$ factorizes through ${\rm Fol}_{q,r}(X_H/H)$ if and only if
    \begin{enumerate}
        \item the induced family $f: (X_T,\Omega^1_{X_T/T}\to \Q ,\Delta_T)\to T$ is a stable family of integrable distributions, and
        \item there is an isomorphism $\OO_{X_T}(1)\simeq \OO_X((m_0 K_\FF + n_0 (K_{X/S}+\Delta)))\otimes f^*B$ for some line bundle $B$ on $T$.
    \end{enumerate}
\end{lemma}
\begin{proof}
By \cite[Theorem 12.1.6]{EGAIV}, there is a maximal open substack parametrising triples such that  $\Q_T$ is a flat family of divisorial sheaves.
Applying Theorem \ref{thm_local_stab_rep}, we get a locally closed partial decomposition 
parametrizing stable foliated surfaces of index $=N$. To conclude, we only have to restrict to the closed substack where condition (2) holds, whose existence is guaranteed by \cite[Corollary 3.22]{modbook}.
\end{proof}

Let us consider the algebraic space of finite type
$$ {\rm Fol}(X_{H_p}/{H_p})=\coprod_{q,r\in F} {\rm Fol}_{q,r}(X_{H_p}/H_p). $$
By construction, this is a fine moduli space for the embedded moduli problem.
In other words, 
$$ \Hom (S,{\rm Fol}(X_{H_p}/H_p)) = \left\{  \begin{array}{l}
    ((X,\Delta, \FF)\xrightarrow[]{\pi} S , \varphi) \, | \, (X, \Delta, \FF)\to S \in\mathcal M^{2,1}_{N, v,p}(S)  \\
  \mbox{ and }  \varphi: X\hookrightarrow\PP^k_S \mbox{ satisfies } \chi(X,\OO_X(\ell))=p(\ell) \mbox{ and} \\ 
  \varphi^*\OO_{\PP^k_S}(1)=\OO_X(m_0 K_\FF + n_0 (K_{X/S}+\Delta))\otimes \pi^*B  \\
  \mbox{for some line bundle } B \mbox{ on } S. 
  \end{array}\right\}
 $$
A morphism between elements $((X_1,\Delta_1, \FF_1)\xrightarrow[]{\pi_1} S , \varphi_1)$ and $(X_2,\Delta_2, \FF_2)\xrightarrow[]{\pi_2} S , \varphi_2)$ in the category $\Fol(X_{H_p}/H_p)(S)$ consists of an isomorphism $f:(X_1,\Delta_1, \FF_1)\to (X_2,\Delta_2, \FF_2)$ over $S$, as in Section \ref{s_precise_definition}, satisfying $\varphi_1=\varphi_2\circ f$. This yields a canonical isomorphism 
\begin{align*}
    \OO_X(m K_{\FF_1} + n (K_{X_1/S}+\Delta_1))\otimes \pi_1^*B_1&\simeq f^*(\OO_X(m_0 K_{\FF_2} + n_0 (K_{X_2/S}+\Delta_2))\otimes \pi_2^*B_2)\\
&\simeq f^*(\OO_X(m_0 K_{\FF_2} + n_0 (K_{X_2/S}+\Delta_2)))\otimes \pi_1^*B_2.
\end{align*}
On the other hand, the isomorphism $f$ satisfies $f(\Delta_1)=\Delta_2$ and $f^*\FF_2=\FF_1$, implying that $\OO_X(m_0 K_{\FF_1} + n_0 (K_{X_1/S}+\Delta_1))\simeq f^*(\OO_X(m_0 K_{\FF_2} + n_0 (K_{X_2/S}+\Delta_2)))$ canonically, from which we can conclude that $B_1$ and $B_2$ are also canonically isomorphic.

\begin{remark} \label{rmk_univfamily}
There is a natural transformation
$$\Fol(X_{H_p}/H_p)\to \mathcal{M}^{2,1}_{N, v,p}$$
which forgets the embedding $\varphi$ and corresponds to the universal family over $\Fol(X_{H_p}/H_p)$.
    Observe further that the group $\PP \rm GL({k+1})$ acts on ${\rm Fol}(X_{H_p}/H_p)$ by means of changing $\varphi$, and the above morphism is invariant under this action.
\end{remark}

Naturally, the arrow above descends to a morphism from the quotient stack $$[{\rm Fol}(X_{H_p}/H_p)/{\rm \PP GL}(k+1)]\to \mathcal{M}^{2,1}_{N, v,p}.$$ 
We are now ready to prove the main statement of this section.

\begin{theorem} \label{thm:rep}
    Fix a positive integer $N$
    and a positive real number $v$.

         Then, $\mathcal M_{N, v}^{2, 1}$ is a Deligne-Mumford stack and is coarsely represented by the geometric quotient $\mathsf M_{N, v}^{2, 1}=\coprod_{p \in F} {\rm Fol}(X_{H_p}/H_p)/\!\!/GL(k+1)$, an algebraic space of finite type over $\mathbb C$.
\end{theorem}
\begin{proof}
As noted above, it suffices to fix $p(l) \in \mathbb Q[l]$ and restrict to 
$\mathcal M_{N, v, p}^{2, 1}$ in place of $\mathcal M_{N, v}^{2, 1}$.

We claim that the above morphism is fully faithful and essentially surjective, and hence an isomorphism of stacks (see for instance \cite[Proposition 3.1.10]{olsson2016algebraic}).

For the first part of the claim, it suffices to show that the map of prestacks 
$$[{\rm Fol}(X_{H_p}/H_p)/{\rm \PP GL}(k+1)]^{pre}\to \mathcal{M}^{2,1}_{N, v,p}$$
is fully faithful, since so is the stackification morphism \begin{align*}[{\rm Fol}(X_{H_p}/H_p)/{\rm \PP GL}(k+1)]^{pre}\to [{\rm Fol}(X_{H_p}/H_p)/{\rm \PP GL}(k+1)].\end{align*}
Let $(X,\Delta,\FF,\varphi)$ be a family over $S$ corresponding to some arrow $S\to {\rm Fol}(X_{H_p}/H_p)$. Observe that 
\begin{align}\label{eq:iso_polarization}
\begin{split}
    H^0(\PP^k_\CC,\OO_{\PP^k_\CC}(1))\otimes \OO_S&\simeq \pi_*\OO_X(1) \\
    &\simeq \pi_*(\OO_X(m_0 K_\FF + n_0 (K_{X/S}+\Delta))\otimes \pi^*B) \\
    &\simeq \pi_*\OO_X(m_0 K_\FF + n_0 (K_{X/S}+\Delta))\otimes B
\end{split}
\end{align}
for some line bundle $B$ on $S$.
An automorphism $f$ of $(X,\Delta,\FF)$ in $\mathcal M^{2,1}_{N, v,p}(S)$ as in Section \ref{s_precise_definition} induces a unique automorphism of $\OO_X((m_0 K_\FF + n_0 (K_{X/S}+\Delta))$ and therefore an  automorphism of $\pi_*\OO_X(m_0 K_\FF + n_0 (K_{X/S}+\Delta))\otimes B$ 
which indeed corresponds to a (unique) canonically defined element in $\PP \rm GL(k+1,S)$ preserving $(X,\Delta,\FF)$. This is, the natural map 
$${\rm Aut}_{\PP \rm GL(k+1)}(X,\Delta,\FF,\varphi)\to {\rm Aut}(X,\Delta,\FF)$$ is bijective, which ensures the first part of the claim.

In order to show essential surjectivity, it is enough to check that for any family $(X,\Delta,\FF)\in \mathcal M^{2,1}_{N, v,p}(S)$, there is an \'etale cover $\{S_i \to S \}$ such that 
$(X,\Delta,\FF)\vert_{S_i}$ is in the image of $[{\rm Fol}(X_{H_p}/H_p)/{\rm \PP GL}(k+1)](S_i)\to \mathcal{M}^{2,1}_{N, v,p} (S_i)$. 
By Lemma \ref{emb}, this holds automatically for an open cover trivializing the vector bundle $\pi_*\OO_X((m_0 K_\FF + n_0 (K_{X/S}+\Delta))$ on $S$.

The above argument shows that our moduli stack satisfies 
\begin{align*}
\mathcal M^{2,1}_{N, v,p} \simeq [{\rm Fol}(X_H/H)/{\rm \PP GL}(k+1)]
\end{align*}
and is therefore an algebraic stack.
On the other hand,  Corollary \ref{cor_finite_automorphisms} establishes that every point of $\mathcal M^{2,1}_{N, v,p}$ has finite stabilizer group. 
Since we are working over characteristic zero, this group is also reduced and hence 
$\mathcal M^{2,1}_{N,v,p}$ is Deligne-Mumford, see \cite[Theorem 4.21]{DM69}.

It is now clear that if the geometric quotient ${\rm Fol}(X_{H_p}/H_p)/\!\!/GL(k+1)$ exists, then it is a coarse moduli space for the functor $\mathcal M^{2,1}_{N, v,p}$. In order to address its existence, 
we claim that the properness of the action is equivalent to Theorem \ref{thm_separatedness}. Indeed, for two families $(X_1,\Delta_1,\FF_1,\varphi_1)$ and $(X_2,\Delta_2, \FF_2,\varphi_2)$ over some germ $S$, an isomorphism $(X_1, \Delta_1,\FF_1)\simeq (X_2,\Delta_2,\FF_2)$ over $S$ also induces isomorphisms 
   $$\pi_1^*\OO_X(m_0K_{\FF_1} + n (K_{X_1/S}+\Delta_1)) \simeq \pi_2^*\OO_X(m_0K_{\FF_2} + n (K_{X_2/S}+\Delta_2))$$
   and hence by \eqref{eq:iso_polarization} a $\PP \rm GL(k+1)$-equivalence. 
   If now $T$ is a germ of smooth curve with generic point $T^0$ and the families $(X_1,\Delta_1,\FF_1,\varphi_1)$ and $(X_2,\Delta_2, \FF_2,\varphi_2)$ are $\PP \rm GL(k+1)$-equivalent over $T^0$, then by Theorem \ref{thm_separatedness} and the argument above they are also $\PP \rm GL(k+1)$-equivalent over $T$. This is, the action is proper.
   We can now make use of  \cite[Theorem 1.5]{Kollar97} in order to guarantee that ${\rm Fol}(X_{H_p}/H_p)/\!\!/\PP \rm GL(k+1)$ is a separated algebraic space of finite type. 
\end{proof}

\begin{corollary} \label{cor:rigid_univ_family}
The open subset $\mathsf M^{2, 1, {\rm rigid}}_{N, v}$ of $\mathsf M^{2, 1}_{N, v}$ corresponding to integrable distributions without automorphisms admits a universal family.    
\end{corollary}
\begin{proof}
    This is a special case of \cite[Theorem 2.2.5]{conrad2007arithmetic}. 
\end{proof}

\begin{theorem}
\label{thm:rep2}
    The moduli functor of stable foliated surfaces, denoted $\mathcal M^{2, 1}$
    is a Deligne-Mumford stack locally of finite type over $\mathbb C$ and satisfies the valuative criterion for properness for families with normal generic fibre over DVRs which are essentially finite type over $\mathbb C$.
\end{theorem}
\begin{proof}
The natural morphisms $\bigsqcup_{N \in \mathbb N, v \in \mathbb Q_{>0}} \mathcal M^{2, 1}_{N, v} \hookrightarrow \mathcal M^{2, 1}$ define an open cover of the stack $\mathcal M^{2, 1}$. Since being Deligne-Mumford is a local condition, the first part of the statement follows from Theorem \ref{thm:rep} above. 
Our functor satisfies the valuative criterion of properness for families with normal generic fibres by Theorem \ref{thm_compactify_to_family}.
\end{proof}

Again, we emphasise that in general valuative properness will fail when the generic fibre of the family is not normal, see Examples \ref{non-normal-not-proper}.

\begin{theorem}
\label{thm_versalilty_stable_def}
    Let $(X_0,\Delta_0,\FF_0)$ be a stable foliated triple, where $X_0$ is a surface. Then there exists a  stable deformation of $(X,\Delta,\FF)\to (D,0)$ which is miniversal among stable deformations, i.e. for every stable deformation $(X',\Delta',\FF')\to (V,s)$ of $(X_0,\Delta_0,\FF_0)$ there exists an analytic neighbourhood $V\ni s$ and a morphism $\varphi: (V,s)\to (D,0)$ such that $(X',\Delta',\FF')\simeq \varphi^*(X,\Delta,\FF)$. Moreover, this yields an isomorphism between $T_0D$ and the tangent space to the functor of stable deformations of $(X_0,\Delta_0,\FF_0)$.
\end{theorem}
\begin{proof}
Fix $N>0$ such that $NK_{\mathcal F_0}$ is Cartier and $\nu = (5NK_{\mathcal F_0}+(K_{X_0}+\Delta_0))^2$.  

Observe that any stable deformation of $(X_0, \Delta_0, \mathcal F_0)$ is necessarily stable deformation of index $=N$ and of adjoint volume $=\nu$.

Being a Deligne-Mumford stack, $\mathcal M^{2, 1}_{N, v}$ has an \'etale atlas $U\to \mathcal M^{2, 1}_{N, v}$ where $U$ is a scheme equipped with a family $(X_U,\Delta_U,\FF_U)\to U$. 
Let $(X_0,\Delta_0,\FF_0)$ be a stable foliated triple and $u\in U$ a closed point together with an isomorphism $(X_0,\Delta_0,\FF_0)\simeq (X_u,\Delta_u,\FF_u)$. If $(X,\Delta,\FF)\to (S,s)$ is another family such that $(X_s,\Delta_s,\FF_s)\simeq^{\alpha} (X_0,\Delta_0,\FF_0) $, then there exists an \'etale neighbourhood $V$ of $s$ and a morphism $\varphi:V\to U$ such that $(X,\Delta,\FF)\vert_V\simeq \varphi^*(X_U,\Delta_U,\FF_U)$ (this arrow is not unique, but rather depends on the choice of the automorphism $\alpha$).

We may take $(X_U,\Delta_U,\FF_U)\to U$
as our desired miniversal family.
\end{proof}

\section{Examples}
\label{s_examples}

\begin{ej}\label{ex:QHusk is necessary} 
To make use of the theory of Quot schemes in a moduli problem one relies on the flatness of a certain sheaf — a condition that is not always automatic. The following examples illustrate why the use of quotient husks is better suited for constructing compact moduli spaces of foliated varieties.
\begin{enumerate}
    \item Let $\FF_1$ be a codimension $1$ foliation on $\mathbb{P}^3$ and let $H$ be a linearly embedded $\PP^2$ that is not tangent to $\FF_1$ in codimension $1$. After a change of coordinates, we can freely assume $H=\{x_0=0\}$. Let us consider the family of foliations $\FF$ along $\PP^3\times \mathbb A^1_t\to \mathbb A^1_t$ such that $\FF_t=\varphi_t^*\FF_1$, where $\varphi_t:$ is the map defined in homogeneous coordinates by $[x_0:x_1:x_2:x_3]\mapsto [tx_0:x_1:x_2:x_3]$. For a general $\FF_1$, this construction yields a flat family of foliations. At the central fibre, the foliation $\FF_0$ is a pullback under a linear projection and has a tangent sheaf of the form $T_{\FF_0}\simeq \OO_{\PP^3}(1)\oplus \OO_{\PP^3}(d)$. In fact, it is not difficult to see that foliations admitting such a splitting are exactly the ones that arise as linear pullbacks of a foliation on $\mathbb P^2$ of the same degree. On the other hand, the sheaf $T_{\FF_0}$ is rigid, and therefore the flatness of $N_\FF$ implies that for $t\neq0$ the tangent sheaf of $\FF_t\simeq \FF_1$ has the same splitting type. We can then conclude that $N_\FF$ is not flat whenever $\FF_1$ is not a linear pullback from a foliation on $\PP^2$.
    \item Let $\FF_0$ be a $1$-dimensional foliation on $\PP^3$ with $K_{\FF_0}$ ample arising as the intersection of two codimension $1$ distributions which are transversal in codimension $1$ (see, for instance, \cite[Example 5.5]{correa2023classification})
  $\FF_0$ can be constructed as the intersection of two codimension 1 distributions $\DD_1$ and $\DD_2$ in general position is equivalent to the kernel of its associated Pfaff field splitting as a direct sum of line bundles. In particular, this implies that its singular scheme is (locally) defined by the maximal minors of the $2\times 3$ matrix whose rows are the coefficients of the $1$ forms inducing $\DD_1$ and $\DD_2$ and is therefore of pure dimension $1$ by \cite[Theorem 5.1]{artin1976lectures}. This stands in contrast with the behaviour of a generic deformation of 
$\FF_0$, whose singularities consist of a finite set of points. Arguing as in the previous example, we see that for a generic deformation $\FF$ of $\FF_0$, the image of its associated Pfaff field fails to be flat over the base.
\end{enumerate}

We refer to \cite{quallbrunn2015families}
for more interesting examples along these lines.

\end{ej}

\begin{ej}
\label{not_ample_example}
    In contrast to the case of varieties, a foliation of general type does not always admit a log canonical model, i.e., a birational model with log canonical singularities such that $K_{\mathcal F}$ is ample.  Consider the following examples.

    \begin{enumerate}
        \item Let $X$ be the Bailey-Borel compactification of a Hilbert modular surface, and let $\mathcal F$ be one of the tautological foliations on $X$.  Let $r\colon Y \to X$ be a finite cover ramified along a general ample divisor and let $p\colon \tilde{Y} \to Y$ be a resolution of the cusp singularities of $Y$ and let $E$ be a connected component of the $p$-exceptional locus. 
Let $\tilde{\mathcal G} = (r \circ p)^{-1}\mathcal F$. It is easy to verify that $\tilde{\mathcal G}$ has canonical singularities and that $K_{\tilde{\mathcal G}}$
is big.  However, $K_{\tilde{\mathcal G}}\vert_E$ is a numerically trivial line bundle, which is not torsion, \cite[Theorem IV.2.2]{McQuillan08}- in particular $K_{\tilde{\mathcal G}}$ is big and nef, but not semi-ample.  This implies that $\tilde{\mathcal G}$ does not have a log canonical model.  See also \cite[Corollary IV.2.3]{McQuillan08}.

\item Let $C$ be a curve of genus $\ge 2$ and let $\mathcal F$ be the foliation on $C \times C$ induced by projection onto the first coordinate.  Let $\Delta \subset C \times C$ be the diagonal.  Then $K_{\mathcal F}+\Delta$
is big and nef, but not semi-ample.  Again 
$(\mathcal F, \Delta)$ does  not admit a log canonical model.
 
\item     In fact it is too much to hope that $K_{\mathcal F}$ is even big on every component of stable integrable distribution.
\\
Let $g \ge 3$, and fix $C$ a smooth curve of genus $g-1$ and an elliptic curve $E$.  We may find $C \subset \overline{M}_g$ such that the family $X$ over $C$ corresponds to the gluing of the diagonal in $C\times C$ to a section of $C \times E$.  Observe that the projection
$X \to C$ gives a semi-log canonical foliation, however, $K_{\mathcal F}\vert_{C \times E} = p^*(K_{E}+P)$ where $P$ is a point and $p\colon C \times E \to E$ is the projection.  In particular, $K_{\mathcal F}\vert_{C\times E}$ is not big.
\\
We may find a family of complete intersection curves $C_t \subset \overline{M}_g$ such that $C_t$ is smooth and such that $C_0 = C+C'_0$.  
If we consider the families of curve corresponding to $C_t$, $X_t \to C_t$, 
this gives us a family of integrable distributions such that for $t \neq 0$, 
$K_{\mathcal F_t}$ is ample, but such that there are components of $X_0$ on which $K_{\mathcal F_0}$ is not big.
\end{enumerate}
These examples show that the ampleness of $K_{\mathcal F}$ is not a natural or achievable condition. Moreover, even if we are only interested in the locus of the moduli space corresponding to foliations with ample canonical divisor, (3) shows that this locus is not closed.
\end{ej}

\begin{ej}
\label{example_index}
Let $X = \{xyz = 0\} \subset \mathbb A^3$.
We can define an integrable distribution on $X$
by specifying a vector field on each of the three components which agree (up to scalar multiple) on the coordinate axes. 

Take the vector fields
\begin{align*} 
&x\partial_x+y\partial_y \text{ on the component } \{z = 0\}, \\
&y\partial_y+z\partial_z \text{ on the component } \{x = 0\},  \text{ and } \\
&\lambda x\partial_x+z\partial_z
\text{ on the component } \{y = 0\}.
\end{align*}
This defines an integrable distribution $\mathcal F_\lambda$ on $X$, and in fact, these integrable distributions fit into a flat family of integrable distributions $(X \times \mathbb A_\lambda^1, \{\mathcal F_\lambda\}) \to \mathbb A_\lambda^1$.
An easy calculation shows that $nK_{\mathcal F_\lambda}$ is Cartier if and only if $\lambda^n = 1$.  This is a local example, but can be easily compactified into a global example where $K_{\mathcal F_\lambda}$
is ample and $\mathcal F_\lambda$ is semi-log canonical if and only if $\lambda$ is a root of unity. 

This example shows (at least) two things.
\begin{enumerate}
    \item On a non-normal surface, the Cartier index of the canonical class of an integrable distribution cannot be bounded in terms of the Hilbert polynomial 
$\chi(X, \mathcal O(mK_{\mathcal F}+nK_X))$.

\item In order to get a bounded family of semi-log canonical integrable distributions one must also fix the index.
\end{enumerate}
\end{ej}

\begin{ej}
\label{easiest_example}
    Consider the vector field $\partial = x\partial_x+ty\partial_y$ on $\mathbb A^3$.
    This vector field is tangent to the projection $\mathbb A^3 \to \mathbb A^1_t$, and so we can think of it as giving
    a family of foliations $(\mathbb A_{(x, y)}^2 \times \mathbb A_t^1, \mathcal F) \to \mathbb A_t^1$.

We make the following observations about this family.

\begin{enumerate}
    \item 
    The integrable distribution restricted to the fibre over $t$ is always log canonical, and is in fact a foliation whenever $t \neq 0$.  However, when $t=0$, the restricted integrable distribution is not a foliation: its singular locus is the divisor $\{x = t = 0\} \subset \{t = 0\}$.  In particular, we see that a family of foliations may degenerate to an integrable distribution.

    \item The integrable distribution restricted to the fibre over $t$ has canonical singularities whenever $t \notin \mathbb Q_{\ge 0}$. It follows that having canonical singularities is not a Zariski open condition in a family.  This is in contrast to the case of varieties.  This implies that to have a representable moduli functor it is necessary to work with log canonical foliation singularities.
    \end{enumerate}
\end{ej}

The following examples show that it is necessary to allow our underlying variety to have singularities in order to get a proper moduli space.

\begin{ej}
\begin{enumerate}
    \item Let $X = \{xy+z^q+w^p = 0\}\subset \mathbb A^4$ and consider the function $f\colon X \to \mathbb A^1$ given by $f(x, y, z, w) = w$.
The vector field $x\partial_x-y\partial_y$
defines a foliation $\mathcal F$ which is tangent to $f$ and so defines a locally stable family of foliations $f\colon (X, \mathcal F) \to \mathbb A^1$.  
\\
For $w \neq 0$, the fibre $X_w$ is smooth, but the fibre $X_0$ is singular.  We remark that Theorem \ref{thm_separatedness} shows that $(X_0, \mathcal F_0)$ is the unique locally stable limit of this family, and so singular fibres are unavoidable.

\item Let $X = \{xy+zw = 0\}$ and let $f\colon X \to \mathbb A^1$ be given by
$f(x, y, z, w) = xy$. The vector field $x\partial_x-y\partial_y+\lambda z\partial_z-\lambda w\partial_w$ defines a foliation $X$ which is tangent to $f$ and so defines a locally stable family of foliations $(X, \mathcal F) \to \mathbb A^1$. For $t \neq 0$ the fibre $X_t$ is smooth and $\mathcal F_t$ is smooth, however, the fibre $X_0$ is not normal with 4 irreducible components and $\mathcal F_0$ is singular.

    \end{enumerate}
\end{ej}

The following example shows that in general it is not possible to extend a (locally) stable over a punctured curve $C \setminus \{c\}$ to all of $C$, in particular, it follows that the moduli space of stable foliations is not proper. 

\begin{ej}
\label{non-normal-not-proper}
    Let $X = \{xy = 0\} \subset \mathbb A^4_{(x, y, z, t)}$ and consider the (trivial) family of varieties given by the projection $f\colon X \to \mathbb A^1_t$ and consider the vector field $\partial = xt\partial_x+y\partial_y+tz\partial_z$.

Since $\partial$ leaves $X$ invariant we may consider its restriction as a vector field on $X$ and we observe that $\partial$ defines a family of locally stable integrable distributions on $X^\circ = f^{-1}(\{t \neq 0\})$.

However, $\partial$ vanishes along $\{t = y = 0\}$, hence along an irreducible component of $X_0$ and therefore $\partial$ does not define a family of integrable distributions on all of $X$.  In order to obtain a family of integrable distributions it is necessary to remove this codimension one zero.  This can be done, and gives a family $\mathcal G$ of integrable distributions on $X$ defined corresponding to the Pfaff field $\phi\colon \Omega^1_{X/T} \to L$ where $L = \mathcal O(-\{t = y = 0\})$ where, away from $\{t = 0\}$,  $\phi$ is given by contraction with the vector field $\partial$.

Since $L$ is not a $\mathbb Q$-Cartier divisorial sheaf, it follows that $\mathcal G$ is not a locally stable family of integrable distributions.

However, if we consider the pull-back of $\mathcal G$ to a family of integrable distributions on the normalisation $X^n \to T$, then it is easy to verify that the resulting family of integrable distributions is locally stable.  On the $\{x = 0\}$ component $\mathcal G^n$ (the pulled back integrable distribution) is defined by $y\partial_y+tz\partial_z$ (which is a log canonical vector field), and on the $\{y = 0\}$ component $\mathcal G^n$ is defined by $x\partial_x+z\partial_z$ (again a log canonical vector field).

This is a local example, but can easily by globalised.  By compactifying and taking ramified covers we can produce an example of a family of integrable distributions
$f\colon (X, \mathcal F) \to \mathbb A^1$ such that 
\begin{itemize}
    \item $(X^\circ := f^{-1}(\mathbb A^{1}\setminus \{0\}), ~\mathcal F|_{X^\circ}) \to \mathbb A^{1}\setminus \{0\}$
    is a stable family;
    \item the pulled-back family to the normalisation $(X^n, \mathcal F^n) \to \mathbb A^1$ is stable; and 
    \item $(X, \mathcal F) \to \mathbb A^1$
    is not stable.
\end{itemize}

Since stable limits, if they exist, are unique
and since a family of integrable distributions is uniquely determined by its normalisation and the generic fibre of the family we see that in this case there is no possible compactification of $(X^\circ, \mathcal F|_{X^\circ})$ to a stable family over $\mathbb A^1$.    
\end{ej}

\appendix

\section{Results from the Minimal Model Program}
\label{mmp}

We recall some basic terminology from the Minimal Model Program.

Let 
$\varphi\colon X\dashrightarrow Y$ be a birational map between normal varieties.  The {\bf exceptional locus}  of $\varphi$ is the closed subset 
${\rm Exc\, }\varphi$ of $X$ where $\varphi$ is not an isomorphism. We say that $\varphi$ is  a {\bf birational contraction} if ${\rm Exc\, } \varphi^{-1}$ does not contain any divisor. Let $\varphi\colon X\dashrightarrow Y$ be a birational contraction between normal varieties and let  $D$ be a $\mathbb Q$-Cartier $\mathbb Q$-divisor on $X$ such that $D_Y\coloneqq \varphi_*D$ is also $\mathbb Q$-Cartier. Then $\varphi$ is called {\bf $D$-negative} (resp. {\bf $D$-trivial}) if there exist a normal variety $W$ and birational morphisms $p\colon W\to X$ and $q\colon W\to Y$  resolving the indeterminacy locus of $\varphi$  such that $p^*D=q^*D_Y+E$ where  $E\ge 0$ is a $\mathbb Q$-divisor whose support coincides with the exceptional locus of $q$ 
(resp. $p^*D=q^*D_Y$). Note that by the negativity lemma (cf. \cite[Lemma 3.39]{KM98}), it follows that if $\varphi$ is a composition of steps of a $D$-MMP, then $\varphi$ is $D$-negative.

Consider a projective morphism $f\colon X \to Z$, 
a foliation $\mathcal F$ on $X$ and a $\mathbb Q$-Weil divisor on $X$ such that $K_{\mathcal F}+\Delta$ is $\mathbb Q$-Cartier and $f$-pseudo-effective.
We say that $(\mathcal F, \Delta)$
admits a minimal model over $Z$ if 
there exists a $K_{\mathcal F}+\Delta$-negative birational contraction $\varphi\colon X \dashrightarrow Y/Z$ such that $p_*(K_{\mathcal F}+\Delta)$ is $g$-nef, where $g\colon Y \to Z$ is the structure morphism. 

\subsection{Existence of Minimal Models for rank one foliations on threefolds}

For convenience of the reader, we state the main result on the existence of minimal models of rank one foliations
needed in this paper.  
There is are several papers proving versions of the minimal model theorem for rank one foliations.  To our knowledge, the version stated below does not appear in any of the references, however, it can be proven as consequence of the results currently in the literature using standard techniques in the Minimal Model Program.

\begin{theorem}
\label{thm_mmp}
    Let $f\colon X \to Z$ be a projective morphism between quasi-projective varieties.
    Suppose that $\dim X = 3$ and $X$ is klt.

    Let $\mathcal F$ be a rank one foliation on $X$ and let $\Delta \ge 0$ be a $\mathbb Q$-Weil divisor such that $(\mathcal F, \Delta)$ has log canonical singularities.
Suppose that $K_{\mathcal F}+\Delta$ is $f$-pseudo-effective.

Then $(\mathcal F, \Delta)$ admits a minimal model over $Z$.
\end{theorem}
\begin{proof}[Proof sketch]
Since $X$ is klt it admits a small $\mathbb Q$-factorialisation $p\colon X' \to X$, see \cite[Corollary 1.4.3]{BCHM10}.  We may freely replace $X'$ by $X$ and so we may assume that $X$ is $\mathbb Q$-factorial.

    If $Z$ is a point, $\Delta = 0$, $\mathcal F$ has canonical singularities and $X$ has quotient singularities, this was first proven in \cite{McQ05}.
    More generally, in the case that $Z$ is a point this follows from \cite[Theorem 8.10]{CS20}.

    The case where $Z$ is general follows from \cite[Theorem 8.10]{CS20} via standard techniques in the MMP, cf. for instance \cite[Theorem 5.9]{MR4669316}.
\end{proof}

\begin{lemma}
\label{lem_KFmmp_singularities}
Let $\pi\colon X \rightarrow W$ be a projective morphism 
between quasi-projective varieties of dimension $\leq 3$ and let $\mathcal F$ be a rank one foliation on $X$.

Let $D \ge 0$ be a $\mathbb Q$-divisor such that $(\mathcal F, D_{{\rm non-inv}})$ is log canonical and 
$(X, D)$ is dlt (resp. log canonical).  Let $\phi\colon X \dashrightarrow X'/W$ be a sequence 
of steps in a $K_{\mathcal F}+D_{{\rm non-inv}}$-MMP over $W$.  Then $(X', \phi_*D)$ is dlt (resp. log canonical).
\end{lemma}
\begin{proof}
Arguing by induction on the number of steps of the MMP we see that it suffices
to prove the case where $\phi\colon X \dashrightarrow X'$ is a single step in the MMP, i.e., 
$\phi$ is a divisorial contraction or a flip associated to an extremal ray $R = \mathbb R_+[C]$.

Suppose that $(\mathcal F, D_{{\rm non-inv}})$ has simple singularities (we refer to \cite[Definition 2.32]{CS20} for the definition of simple singularities).
The proof of existence of divisorial contractions/flips in \cite[Theorem 6.2 and Theorem 6.5]{CS20}
shows that there exists a $\mathbb Q$-divisor $G \ge 0$
such that $(X, D+G)$ is log canonical and $(K_X+D+G)\cdot C<0$.
For $0<t \ll 1$ we have that $(X, D+G-tG)$ is dlt (resp. log canonical) and that $(K_X+D+G-tG)\cdot C<0$.  The $K_{\mathcal F}+D_{{\rm non-inv}}$-contraction (resp. flip) is therefore the 
$K_X+D+G-tG$- contraction (resp. flip).  The negativity lemma implies that 
$(X', \phi_*(D+G-tG))$ is dlt (resp. log canonical), and since $X'$ is $\mathbb Q$-factorial    it follows that $(X', \phi_*D)$ is dlt (resp. log canonical), as required.

Now suppose that $(\mathcal F, D_{{\rm non-inv}})$ is log canonical. Let
$\pi\colon Y \rightarrow X$ be a foliated plt modification of $(\mathcal F, D_{{\rm non-inv}})$ (see \cite[Theorem 8.4]{CS20}) and  
let $D_Y = \pi_*{-1}D+{\rm exc}(\pi)$.
By construction, $(\mathcal F_Y, D_{Y,{\rm non-inv}})$ has simple singularities 
and $(Y, D_Y)$ is log canonical.  The proof of \cite[Theorem 8.4]{CS20} in fact shows that if $(X, D)$ is dlt then $(Y, D_Y)$ is also dlt.
Following the proof of existence of log canonical flips in \cite[Theorem 8.8]{CS20} we see that
the $K_{\mathcal F}+D_{{\rm non-inv}}$-flip/divisorial contraction is the output of a (partial) $K_{\mathcal F_Y}+D_{Y, {\rm non-inv}}$-MMP, call it $\psi\colon Y \dashrightarrow X'$.  By the previous case, we deduce that $(X', \psi_*D_Y)$ is dlt
(resp. log canonical) and we may conclude.
\end{proof}

\bibliographystyle{alpha}
\bibliography{bib}

\end{document}